\documentclass{article}
\usepackage{amssymb}
\usepackage{amsmath}
\usepackage{amsthm}
\usepackage{mathtools}
\usepackage{graphicx}
\newtheorem{thm}{Theorem}
\newtheorem{cor}[thm]{Corollary}
\newtheorem{lem}[thm]{Lemma}

\newtheorem{defn}[thm]{Definition}
\newtheorem*{rem}{Remark}
\begin{document}
\title{Moduli Spaces of Unordered $n\ge5$ Points on the Riemann Sphere and Their Singularities}
\author{Yue Wu$^\dagger$ and Bin Xu$^\ddagger$}
\date{}

%----------------------------------------------------------------
\maketitle

%----------------------------------------------------------------
\begin{abstract}
For $n\ge5$, it is well known that the moduli space $\mathfrak{M_{0,\:n}}$ of unordered $n$ points on the Riemann sphere is a quotient space of the Zariski open set $K_n$ of $\mathbb C^{n-3}$ by an $S_n$ action. The stabilizers of this $S_n$ action at certain points of this Zariski open set $K_n$ correspond to the groups fixing the sets of $n$ points on the Riemann sphere.
Let $\alpha$ be a subset of $n$ distinct points on the Riemann sphere. We call the group of all linear fractional transformations leaving $\alpha$ invariant the stabilizer of $\alpha$, which is finite by observation. For each non-trivial finite subgroup $G$ of the group ${\rm PSL}(2,{\Bbb C})$ of linear fractional transformations, we give the necessary and sufficient condition for finite subsets of the Riemann sphere under which the stabilizers of them are conjugate to $G$. We also prove that there does exist some finite subset of the Riemann sphere whose stabilizer coincides with $G$.
Next we obtain the irreducible decompositions of the representations of the stabilizers on the tangent spaces at the singularities of $\mathfrak{M_{0,\:n}}$.
At last, on $\mathfrak{M_{0,\:5}}$ and $\mathfrak{M_{0,\:6}}$, we work out explicitly the singularities and the representations of their stabilizers on the tangent spaces at them.
\end{abstract}

\footnote{
$\dagger$ The first author is supported in part by the National Natural Science Foundation of China (Grant No. 11671371).\\
$^\ddagger$
The second author is supported in part by the National Natural Science Foundation of China (Grant Nos. 11571330 and 11271343) and the Fundamental Research Funds for the Central
Universities (Grant No. WK3470000003).}
%----------------------------------------------------------------
\tableofcontents

%----------------------------------------------------------------
\section{Introduction}
This manuscript is originated from the bachelor's thesis of the first author. We study in great detail some algebraic properties of orbifold singularities of the moduli space $\mathfrak{M_{0,\:n}}$ of $n\geq 5$ unordered points on the Riemann sphere, such as the stabilizers of them and the corresponding linear representations over the tangent spaces at them. A lot of interesting results are observed and proved by hand computation, at least some of which, we do admit, may be well known to experts. This elementary and systematic exposition on $\mathfrak{M_{0,\:n}}$ might be still of some value to be open access to the community through arXiv.
We should mention some references for the widely known classical cases of $n=4,\,5$ and $6$ about the stabilzers. In particular, see \cite{Wiman} and \cite[Chapter 8 ]{Dol} for the case of $n=4$ and $5$, and
\cite{Bolza} for the case of $n=6$.

%----------------------------------------------------------------
\subsection{The Moduli Space $\mathfrak{M_{0,\:n}}$}

Let $C$ and $C'$ be two compact Riemann surfaces of genus $g$, and
$$\{p_1,\:p_2,\:\cdots,\:p_n\}\subseteq C,\: \{p'_1,\:p'_2,\:\cdots,\:p'_n\}\subseteq C'.$$ $(C,\:\{p_1,\:p_2,\:\cdots,\:p_n\})$ and $(C',\:\{p'_1,\:p'_2,\:\cdots,\:p'_n\})$ are isomorphic if there exists some biholomorphic map $f:\:C\to C'$ such that
\begin{displaymath}
f(\{p_1,\:p_2,\:\cdots,\:p_n\})=\{p'_1,\:p'_2,\:\cdots,\:p'_n\}.
\end{displaymath}

\begin{defn}[\cite{Mu-Pe}]
The moduli space $\mathfrak{M_{g,\:n}}$ is the set of isomorphism classes of compact Riemann surfaces of genus $g$ with $n$ unordered marked points.
\end{defn}

It is well-known that
\begin{thm}[Deligne-Mumford\cite{De-Mu}]
$\mathfrak{M_{g,\:n}}$ is both a complex orbifold and an irreducible quasiprojective variety of dimension $3g-3+n$, where $n\ge3$ if $g=0$, $n\ge1$ if $g=1$ and $n\ge0$ if $g\ge2$.
\end{thm}

In Section \ref{orbif} using elementary methods we shall prove that
\begin{thm}
$\mathfrak{M_{0,\:n}}=K_n/G$ for $n\ge4$, where $K_n$ is the Zariski open set
$$
K_n=\{\boldsymbol{\lambda}=(\lambda_1,\:\lambda_2,\:\cdots,\:\lambda_{n-3})\in\mathbb C^{n-3}|\lambda_i\ne0,\:1,\:\lambda_i\ne\lambda_j,\:\forall i,\:j=1,\:2,\:\cdots,\:n-3,\:i\ne j\}
$$
and $G$ is a finite group of birational transformations of $\mathbb C^{n-3}$, whose restrictions to $K_n$ are isomorphisms on $K_n$.
\end{thm}
\begin{cor}
$\mathfrak{M_{0,\:n}}=K_n/G$ is a complex orbifold of dimension $n-3$ for $n\ge4$.
\end{cor}
\begin{rem}
$\mathfrak{M_{0,\:n}}$ is a single point if $n\le3$. Each point in $\mathfrak{M_{0,\:4}}$ can be viewed as an elliptic curve, and vice versa. Hence $\mathfrak{M_{0,\:4}}$ is isomorphic to the complex plane.
\end{rem}

In Section \ref{grou} we shall find that
\begin{thm}
$G$ is isomorphic to $S_n$ when $n\ge5$.
\end{thm}
\begin{cor}
The moduli space $\mathfrak{M_{0,\:n}}$ is the quotient space of a Zariski open set $K_n$ of $\mathbb C^{n-3}$ by an $S_n$ action  when $n\ge5$.
\end{cor}

For $\boldsymbol {\lambda}\in K_n$, set
$$
[\boldsymbol {\lambda}]=\{0,\:1,\:\infty,\:\lambda_1,\:\lambda_2,\:\cdots,\:\lambda_{n-3}\}
$$
and $G_{\boldsymbol \lambda}$ the stabilizer of $\boldsymbol {\lambda}$
$$
G_{\boldsymbol \lambda}=\{g_{\sigma}\in G|g_\sigma(\boldsymbol \lambda)=\boldsymbol \lambda\}.
$$
\begin{defn}
For $n\ge5$, we call $\overline{[\boldsymbol {\lambda}]}\in \mathfrak{M_{0,\:n}}$ an oribfold singularity of $\mathfrak{M_{0,\:n}}$ if $G_{\boldsymbol \lambda}$ is non-trivial.
\end{defn}

Given a finite subset $\alpha$ of the Riemann sphere, let $\mathcal{A}_{\alpha}$ denote the group of linear fractional transformations that fix $\alpha$. In Section \ref{sing} we shall prove that
\begin{thm}
For $n\ge5$, $G_{\boldsymbol \lambda}$ is isomorphic to $\mathcal{A}_{[\boldsymbol {\lambda}]}$.
\end{thm}

Thus the study of the orbifold singularities of $\mathfrak{M_{0,\:n}}$ for $n\geq 5$ can be reduced to that of $\mathcal{A_\alpha}$ for subsets $\alpha$ with $n$ elements on the Riemann sphere. Next we shall find a way to determine $\mathcal{A_\alpha}$ for any finite subset $\alpha$ of the Riemann sphere when $|\alpha|\ge4$.

%----------------------------------------------------------------
\subsection{The Subset Fixed by a Specific Group}

Given a finite subset $\alpha$ of the Riemann sphere, $|\alpha|=n$, $n\ge4$. It is easy to see that $\mathcal{A}_{\alpha}$ has at most $n(n-1)(n-2)$ elements. Thus $\mathcal{A}_{\alpha}$ is finite. There are exactly five kinds of non-trivial finite linear fractional transformation groups: the icosahedral group $I$, the octahedral group $O$, the tetrahedral group $T$, the dihedral group $D_n$ and the finite cyclic group $\mathbb Z_n$, $n\ge2$.

Let $G$ be a finite group of linear fractional transformations. In Section \ref{orbit}, by discussing the orbits of the action of  $G$ on $S^2$ and $\widehat{\mathbb{C}}$, we shall find all finite subsets $\alpha$ of $\widehat{\mathbb{C}}$ such that $\mathcal{A}_{\alpha}\simeq G$.  Thus for any given finite $\alpha\subseteq \widehat{\mathbb{C}}$, $|\alpha|\ge4$, we can find $\mathcal{A}_{\alpha}$. The following five theorems constitute the whole result. Notice that we shall not distinguish a point on $S^2$ from its image on the extended complex plane under the stereographic projection.

Suppose that $G$ is the icosahedral group $I$ that fixes a regular dodecahedron whose center is the origin. Let $V_I,\:F_I,\:E_I$ denote the vertices, the projections of the central points of the faces on $S^2$ and the projections of the middle points of the edges on $S^2$ respectively, with the origin being the central of the projection.
For any $X\in S^2 \backslash (V_I\cup F_I\cup E_I)$, define $B_I(X)$ as the orbit of $X$.
\begin{thm}
For any finite subset $\alpha$ of $S^2$, $\mathcal{A}_{\alpha}=G(\:\simeq A_5)$ if and only if $\alpha$ is a union of certain elements in $\{V_I,\:F_I,\:E_I\}\cup\{B_I(X)|X\in S^2 \backslash (V_I\cup F_I\cup E_I)\}$.
\end{thm}

Now suppose $G$ is the octahedral group $O$ that fixes a cube whose center is the origin. Let $V_O,\:F_O,\:E_O$ denote the vertices, the projections of the central points of the faces on $S^2$ and the projections of the middle points of the edges on $S^2$ respectively, with the origin being the center of the projection. For any $X\in S^2 \backslash (V_O\cup F_O\cup E_O)$, define $B_O(X)$ as the orbit of $X$.
\begin{thm}
For any finite subset $\alpha$ of $S^2$, $\mathcal{A}_{\alpha}=G(\:\simeq S_4)$ if and only if $\alpha$ is a union of certain elements in $\{V_O,\:F_O,\:E_O\}\cup\{B_O(X)|X\in S^2 \backslash (V_O\cup F_O\cup E_O)\}$.
\end{thm}

Now suppose $G$ is the tetrahedral group $T$ that fixes a regular tetrahedron whose center is the origin. Let $V_T,\:F_T,\:E_T$ denote the vertices, the projections of the central points of the faces on $S^2$ and the projections of the middle points of the edges on $S^2$ respectively, with the origin being the center of the projection. For any $X\in S^2 \backslash (V_T\cup F_T\cup E_T)$, define $B_T(X)$ as the orbit of $X$. Thus we have (see Figure \ref{iA_5, A_4})
$$
V_I=V_T\cup F_T\cup B_T(B);\:
F_I=B_T(N);\:
E_I=E_T\cup B_T(Q)\cup B_T(R);
$$$$
B_I(X)=B_T(X)\cup B_T(g(X))\cup B_T(g^2(X))\cup B_T(g^3(X))\cup B_T(g^4(X)),
$$
for $X\in S^2 \backslash (V_I\cup F_I\cup E_I)$.
\begin{figure}[!h]
\centering
\includegraphics[width=3in]{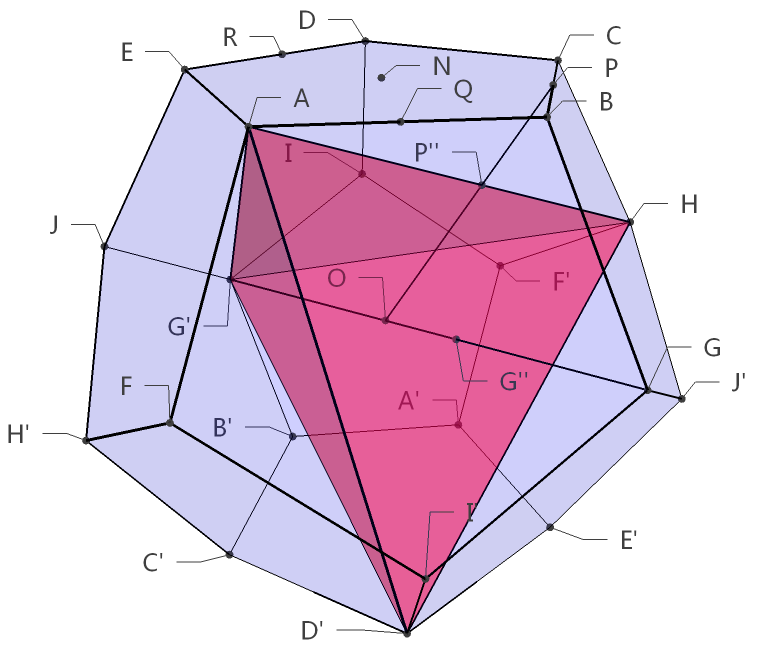}
\caption{The dodecahedron fixed by $\mathcal{A}_{\alpha}$ and the tetrahedron fixed by $G$.}
\label{iA_5, A_4}
\end{figure}

And (see Figure \ref{iS_4, A_4})
$$
V_O=V_T\cup F_T;
F_O=E_T;\:
E_O=B_T(P);\:
$$$$
B_O(X)=B_T(X)\cup B_T(g(X))
$$
for $X\in S^2 \backslash (V_O\cup F_O\cup E_O)$.
\begin{figure}[!h]
\centering
\includegraphics[width=2.5in]{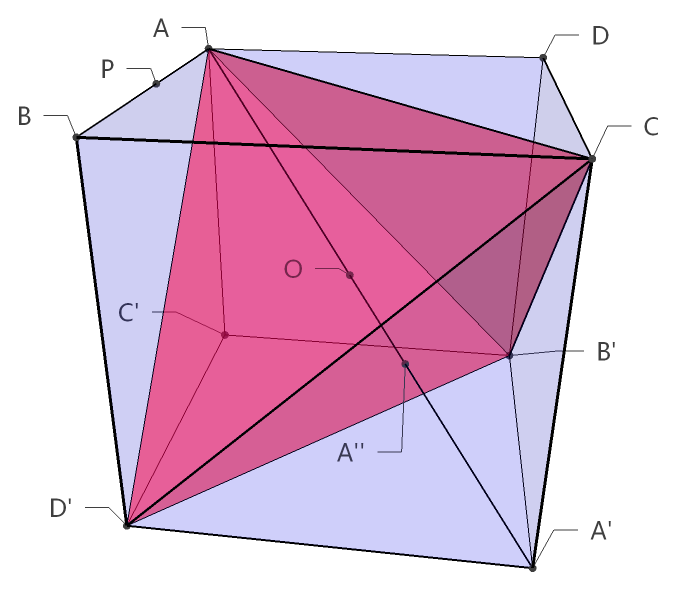}
\caption{The dodecahedron fixed by $\mathcal{A}_{\alpha}$ and the tetrahedron fixed by $G$.}
\label{iS_4, A_4}
\end{figure}

\begin{thm}
For any finite subset $\alpha$ of $S^2$, $\mathcal{A}_{\alpha}=G(\:\simeq A_4)$ if and only if all of the three claims are true:
\begin{enumerate}
\item $\alpha$ is a union of certain elements in
$$
\{V_T,\:F_T,\:E_T\}\cup\{B_T(X)|X\in S^2 \backslash (V_T\cup F_T\cup E_T)\};
$$
\item $\alpha$ is NOT a union of certain elements in
$$
\{V_I,\:F_I,\:E_I\}\cup\{B_I(X)|X\in S^2 \backslash (V_I\cup F_I\cup E_I)\};
$$
\item  $\alpha$ is NOT a union of certain elements in
$$
\{V_O,\:F_O,\:E_O\}\cup\{B_O(X)|X\in S^2 \backslash (V_O\cup F_O\cup E_O)\}.
$$
\end{enumerate}
\end{thm}

\begin{rem}
{\rm Theorem 11 does not claim the existence of $\alpha$ though it does give a necessary and sufficient condition for $\alpha$ under which
$\mathcal{A}_\alpha$ coincides with $A_4$. Theorems 12 and 13 have the same flavor as Theorem 11.
We overcome this shortcoming by giving a unified existence result for $\alpha$ in Theorem 24. }
\end{rem}

Now we go back to $\widehat{\mathbb C}$. Assume that
\begin{displaymath}
f(z)=e^{\frac{1}{n}2\pi i}z,\:g(z)=\frac{1}{z},\:n\ge2
\end{displaymath}
and
\begin{displaymath}
G=\langle f, g\rangle\simeq D_n.
\end{displaymath}
Set
\begin{displaymath}
V=\{0, \infty\},
\end{displaymath}
\begin{displaymath}
A_n=\{e^{\frac{k}{n}2\pi i}|k\in\mathbb{Z}\},
\end{displaymath}
\begin{displaymath}
B_n=\{e^{\frac{2k+1}{2n}2\pi i}|k\in\mathbb{Z}\}
\end{displaymath}
and
\begin{displaymath}
C_n(z)=\{ze^{\frac{k}{n}2\pi i}|k\in\mathbb{Z}\}\cup\{z^{-1}e^{\frac{k}{n}2\pi i}|k\in\mathbb{Z}\}
\end{displaymath}
for $z\in{\mathbb C}^*\backslash\{e^{\frac{l}{2n}2\pi i}|l\in\mathbb{Z}\}$.

\begin{thm}
For any finite subset $\alpha$ of $S^2$, $|\alpha|\ge4$, $\mathcal{A}_{\alpha}=G(\simeq D_n)$ if and only if all of the four claims are true:
\begin{enumerate}
\item $\alpha$ is a union of certain elements in
$$
\{V,\:A_n,\:B_n\}\cup\{C_n(z)|z\in{\mathbb C}^*\backslash\{e^{\frac{l}{2n}2\pi i}|l\in\mathbb{Z}\}\};
$$
\item $\alpha$ is NOT a union of certain elements in
$$
\{V,\:A_{pn},\:B_{pn}\}\cup\{C_{pn}(z)|z\in{\mathbb C}^*\backslash\{e^{\frac{l}{2n}2\pi i}|l\in\mathbb{Z}\}\},\:p\ge2;
$$
\item $\alpha$ is NOT in the icosahedral, the octahedral or the tetrahedral case;
\item when $n=2$, $\alpha$ is NOT a union of certain elements in
$$
\{A_2,\:\phi^{-1}(A_{2p}),\:\phi^{-1}(B_{2p})\}\cup\{\phi^{-1}(C_{2p}(z))|z\in\mathbb C^*\backslash\{e^{\frac{l}{4p}2\pi i}|l\in\mathbb{Z}\}\},\:p\ge2,
$$
where
$$\phi(z)=\frac{1+z}{1-z}.$$
\end{enumerate}
\end{thm}

Assume that
\begin{displaymath}
G=\langle z \mapsto e^{\frac{2\pi i}{n}}z\rangle\simeq\mathbb{Z}_n,\:n\ge 2.
\end{displaymath}
Set
$$
C_n(z)=\{e^{\frac{k}{n}2\pi i}z|k\in\mathbb{Z}\}
$$
for $z\in\mathbb C^*$.

\begin{thm}
For any finite subset $\alpha$ of $S^2$, $|\alpha|\ge4$, $\mathcal{A}_{\alpha}=G(\simeq \mathbb Z_n), n\ge2$ if and only if all of the three claims are true:
\begin{enumerate}
\item $\alpha$ is a union of certain elements in
$$
\{\{\infty\},\:\{0\}\}\cup\{C_n(z)|z\in{\mathbb C}^*\};
$$
\item $\alpha$ is NOT a union of certain elements in
$$
\{\{\infty\},\:\{0\}\}\cup\{C_{pn}(z)|z\in{\mathbb C}^*\},\:p\ge2;
$$
\item $\alpha$ is NOT in the icosahedral, the octahedral the tetrahedral or the dihedral case.
\end{enumerate}
\end{thm}

%----------------------------------------------------------------
\subsection{Representations of the Stabilizers of Singularities}

For any $\boldsymbol {\lambda}\in K_n$, ${n\ge5}$ such that $\overline{[\boldsymbol {\lambda}]}$ is a singularity of $\mathfrak{M_{0,\:n}}$, each $g_\sigma \in G_{\boldsymbol {\lambda}}$ introduces a tangential mapping $g_\sigma^*$ of $\text{T}_{\boldsymbol {\lambda}}K_n$ such that
$$
g^*_\sigma(\frac{\partial}{\partial\lambda^i}\bigg|_{\boldsymbol{\lambda}})(\lambda^j)=\frac{\partial}{\partial\lambda^i}\bigg|_{\boldsymbol{\lambda}}(\lambda^j\circ g_\sigma)=\frac{\partial}{\partial\lambda^i}\bigg|_{\boldsymbol{\lambda}}(g^{(j)}_\sigma)=\frac{\partial g^{(j)}_\sigma}{\partial\lambda^i}\bigg|_{\boldsymbol{\lambda}}
$$
for $\:i,\:j=1,\:2,\:\cdots,\:n-3.$

\begin{defn}
Let $\boldsymbol{J_\sigma}$ denote the Jacobian matrix of $g_\sigma$. Define a representation of $G_{\boldsymbol{\lambda}}$
$$
X_{\boldsymbol{\lambda}}:\:G_{\boldsymbol{\lambda}}\to\text{GL}_{n-3},\:g_\sigma\mapsto \boldsymbol{J_\sigma}\bigg|_{\boldsymbol{\lambda}}.
$$
\end{defn}
It is obvious that $X_{\boldsymbol{\lambda}}$ is a representation of $G_{\boldsymbol {\lambda}}$ of degree $n-3$. %In Section \ref{rep} by calculating its character $\chi_{\boldsymbol{\lambda}}$, we shall find the irreducible decomposition of $X_{\boldsymbol{\lambda}}$ for each $\boldsymbol {\lambda}\in K_n$ such that $\overline{[\boldsymbol {\lambda}]}$ is a singularity of $\mathfrak{M_{0,\:n}}$.
Let $K^{(1)}_{\boldsymbol{\lambda}},\: K^{(2)}_{\boldsymbol{\lambda}},\:\cdots,\:K^{(k)}_{\boldsymbol{\lambda}}$ denote the conjugate classes of $G_{\boldsymbol{\lambda}}$, and $X^{(1)}_{\boldsymbol{\lambda}},\: X^{(2)}_{\boldsymbol{\lambda}},\:\cdots,\:X^{(k)}_{\boldsymbol{\lambda}}$ its irreducible representations (for the exact definition, please refer to Section \ref{rep}).

\begin{defn}
For any $\boldsymbol {\lambda}\in K_n$, ${n\ge5}$ such that $\overline{[\boldsymbol {\lambda}]}$ is a singularity of $\mathfrak{M_{0,\:n}}$, assume that  $X_{\boldsymbol{\lambda}}=p_1X^{(1)}_{\boldsymbol{\lambda}}\oplus\cdots\oplus p_kX^{(k)}_{\boldsymbol{\lambda}}$. Then call $(p_1,\:p_2,\:\cdots,\:p_k)$ the \textbf{multiplicity vector} of $\boldsymbol{\lambda}$.
\end{defn}

From the above discussion we know that if $G_{\boldsymbol{\lambda}}$ is isomorphic to the icosahedral group $I$, $[\boldsymbol{\lambda}]$ is equivalent to one of the following sets (which are categorized into eight types):
\begin{enumerate}
\item \textbf{F+mB}: $F_I\cup B_I(a_1)\cup B_I(a_2)\cup\cdots\cup B_I(a_m)$, where $B_I(a_1),\:B_I(a_2),\:\cdots,\:B_I(a_m)$ are different orbits of order 60, $m\in\mathbb {N}$;
\item \textbf{V+mB}: $V_I\cup B_I(a_1)\cup B_I(a_2)\cup\cdots\cup B_I(a_m)$, where $B_I(a_1),\:B_I(a_2),\:\cdots,\:B_I(a_m)$ are different orbits of order 60, $m\in\mathbb {N}$;
\item \textbf{E+mB}: $E_I\cup B_I(a_1)\cup B_I(a_2)\cup\cdots\cup B_I(a_m)$, where $B_I(a_1),\:B_I(a_2),\:\cdots,\:B_I(a_m)$ are different orbits of order 60, $m\in\mathbb {N}$;
\item \textbf{FV+mB}: $F_I\cup V_I\cup B_I(a_1)\cup B_I(a_2)\cup\cdots\cup B_I(a_m)$, where $B_I(a_1),\:B_I(a_2),\:\cdots,\:B_I(a_m)$ are different orbits of order 60, $m\in\mathbb {N}$;
\item \textbf{VE+mB}: $V_I\cup E_I\cup B_I(a_1)\cup B_I(a_2)\cup\cdots\cup B_I(a)_m)$, where $B_I(a_1),\:B_I(a_2),\:\cdots,\:B_I(a_m)$ are different orbits of order 60, $m\in\mathbb {N}$;
\item \textbf{EF+mB}: $E_I\cup F_I\cup B_I(a_1)\cup B_I(a_2)\cup\cdots\cup B_I(a_m)$, where $B_I(a_1),\:B_I(a_2),\:\cdots,\:B_I(a_m)$ are different orbits of order 60, $m\in\mathbb {N}$;
\item \textbf{FVE+mB}: $F_I\cup V_I\cup E_I\cup B_I(a_1)\cup B_I(a_2)\cup\cdots\cup B_I(a_m)$, where $B_I(a_1),\:B_I(a_2),\:\cdots,\:B_I(a_m)$ are different orbits of order 60, $m\in\mathbb {N}$;
\item \textbf{(1+m)B}: $B_I(a_0)\cup B_I(a_1)\cup\cdots\cup B_I(a_m)$, where $B_I(a_0),\:B_I(a_1),\:\cdots,\:B_I(a_m)$ are different orbits of order 60, $m\in\mathbb {N}$.
\end{enumerate}

\begin{thm}
For any $\boldsymbol{\lambda}\in K_n$ such that $G_{\boldsymbol{\lambda}}$ is isomorphic to the icosahedral group $I$, we have found its multiplicity vector
\begin{enumerate}
\item \textbf{F+mB}: $(m,\:1+4m,\:1+5m,\:3m,\:3m)$;
\item \textbf{V+mB}: $(m,\:1+4m,\:2+5m,\:3m,\:1+3m)$;
\item \textbf{E+mB}: $(m,\:2+4m,\:2+5m,\:1+3m,\:2+3m)$;
\item \textbf{FV+mB}: $(m,\:2+4m,\:3+5m,\:1+3m,\:1+3m)$;
\item \textbf{VE+mB}: $(m,\:3+4m,\:4+5m,\:2+3m,\:3+3m)$;
\item \textbf{EF+mB}: $(m,\:3+4m,\:3+5m,\:2+3m,\:2+3m)$;
\item \textbf{FVE+mB}: $(m,\:4+4m,\:5+5m,\:3+3m,\:3+3m)$;
\item \textbf{(1+m)B}: $(1+m,\:4+4m,\:5+5m,\:2+3m,\:3+3m)$.
\end{enumerate}
\end{thm}

If $G_{\boldsymbol{\lambda}}$ is isomorphic to the octahedral group $O$, $[\boldsymbol{\lambda}]$ is equivalent to one of the following sets (which are categorized into eight types):
\begin{enumerate}
\item \textbf{F+mB}: $F_O\cup B_O(a_1)\cup B_O(a_2)\cup\cdots\cup B_O(a_m)$, where $B_O(a_1),\:B_O(a_2),\:\cdots,\:B_O(a_m)$ are different orbits of order 24, $m\in\mathbb {N}$;
\item \textbf{V+mB}: $V_O\cup B_O(a_1)\cup B_O(a_2)\cup\cdots\cup B_O(a_m)$, where $B_O(a_1),\:B_O(a_2),\:\cdots,\:B_O(a_m)$ are different orbits of order 24, $m\in\mathbb {N}$;
\item \textbf{E+mB}: $E_O\cup B_O(a_1)\cup B_O(a_2)\cup\cdots\cup B_O(a_m)$, where $B_O(a_1),\:B_O(a_2),\:\cdots,\:B_O(a_m)$ are different orbits of order 24, $m\in\mathbb {N}$;
\item \textbf{FV+mB}: $F_O\cup V_O\cup B_O(a_1)\cup B_O(a_2)\cup\cdots\cup B_O(a_m)$, where $B_O(a_1),\:B_O(a_2),\:\cdots,\:B_O(a_m)$ are different orbits of order 24, $m\in\mathbb {N}$;
\item \textbf{VE+mB}: $V_O\cup E_O\cup B_O(a_1)\cup B_O(a_2)\cup\cdots\cup B_O(a_m)$, where $B_O(a_1),\:B_O(a_2),\:\cdots,\:B_O(a_m)$ are different orbits of order 24, $m\in\mathbb {N}$;
\item \textbf{EF+mB}: $E_O\cup F_O\cup B_O(a_1)\cup B_O(a_2)\cup\cdots\cup B_O(a_m)$, where $B_O(a_1),\:B_O(a_2),\:\cdots,\:B_O(a_m)$ are different orbits of order 24, $m\in\mathbb {N}$;
\item \textbf{FVE+mB}: $F_O\cup V_O\cup E_O\cup B_O(a_1)\cup B_O(a_2)\cup\cdots\cup B_O(a_m)$, where $B_O(a_1),\:B_O(a_2),\:\cdots,\:B_O(a_m)$ are different orbits of order 24, $m\in\mathbb {N}$;
\item \textbf{(1+m)B}: $B_O(a_0)\cup B_O(a_1)\cup B_O(a_2)\cup\cdots\cup B_O(a_m)$, where $B_O(a_1),\:B_O(a_2),\:\cdots,\:B_O(a_m)$ are different orbits of order 24, $m\in\mathbb {N}$.
\end{enumerate}

\begin{thm}
For any $\boldsymbol{\lambda}\in K_n$ such that $G_{\boldsymbol{\lambda}}$ is isomorphic to the octahedral group $O$, we have found its multiplicity vector
\begin{enumerate}
\item \textbf{F+mB}: $(m,\:m,\:1+3m,\:3m,\:2m)$;
\item \textbf{V+mB}: $(m,\:m,\:1+3m,\:3m,\:1+2m)$;
\item \textbf{E+mB}: $(m,\:1+m,\:1+3m,\:1+3m,\:1+2m)$;
\item \textbf{FV+mB}: $(m,\:m,\:2+3m,\:1+3m,\:1+2m)$;
\item \textbf{VE+mB}: $(m,\:1+m,\:2+3m,\:2+3m,\:2+2m)$;
\item \textbf{EF+mB}: $(m,\:1+m,\:2+3m,\:2+3m,\:1+2m)$;
\item \textbf{FVE+mB}: $(m,\:1+m,\:3+3m,\:3+3m,\:2+2m)$;
\item \textbf{(1+m)B}: $(1+m,\:1+m,\:3+3m,\:2+3m,\:2+2m)$.
\end{enumerate}
\end{thm}

If $G_{\boldsymbol{\lambda}}$ is isomorphic to the Tetrahedral Group $T$, then $[\boldsymbol{\lambda}]$ is equivalent to one of the following sets (which are categorized into six types):
\begin{enumerate}
\item \textbf{F+mB}: $F_T\cup B_T(a_1)\cup B_T(a_2)\cup\cdots\cup B_T(a_m)$, where $B_T(a_1),\:B_T(a_2),\:\cdots,\:B_T(a_m)$ are different orbits of order 12, $m\in\mathbb {N^*}$;
\item \textbf{E+mB}: $E_T\cup B_T(a_1)\cup B_T(a_2)\cup\cdots\cup B_T(a_m)$, where $B_T(a_1),\:B_T(a_2),\:\cdots,\:B_T(a_m)$ are different orbits of order 12, $m\in\mathbb {N^*}$;
\item \textbf{FV+mB}: $F_T\cup V_T\cup B_T(a_1)\cup B_T(a_2)\cup\cdots\cup B_T(a_m)$, where $B_T(a_1),\:B_T(a_2),\:\cdots,\:B_T(a_m)$ are different orbits of order 12, $m\in\mathbb {N^*}$;
\item \textbf{FE+mB}: $F_T\cup E_T\cup B_T(a_1)\cup B_T(a_2)\cup\cdots\cup B_T(a_m)$, where $B_T(a_1),\:B_T(a_2),\:\cdots,\:B_T(a_m)$ are different orbits of order 12, $m\in\mathbb {N}$;
\item \textbf{FVE+mB}: $F_T\cup V_T\cup E_T\cup B_I(a_1)\cup B_I(a_2)\cup\cdots\cup B_I(a_m)$, where $B_I(a_1),\:B_I(a_2),\:\cdots,\:B_I(a_m)$ are different orbits of order 12, $m\in\mathbb {N^*}$;
\item \textbf{(1+m)B}: $B_T(a_1)\cup B_T(a_2)\cup\cdots\cup B_T(a_m)$, where $B_T(a_1),\:B_T(a_2),\:\cdots,\:B_T(a_m)$ are different orbits of order 12, $m\in\mathbb {N}$.
\end{enumerate}

\begin{thm}
For each singularity $[\boldsymbol{\lambda}]$ such that $G_{\boldsymbol{\lambda}}$ is isomorphic to the tetrahedral group $T$, we have found its multiplicity vector
\begin{enumerate}
\item \textbf{F+mB}: $(m,\:1+m,\:m,\:3m)$;
\item \textbf{E+mB}: $(m,\:m,\:m,\:1+3m)$;
\item \textbf{FV+mB}: $(m,\:1+m,\:1+m,\:1+3m)$;
\item \textbf{FE+mB}: $(m,\:1+m,\:m,\:2+3m)$;
\item \textbf{FVE+mB}: $(m,\:1+m,\:1+m,\:3+3m)$;
\item \textbf{(1+m)B}: $(m,\:1+m,\:1+m,\:2+3m)$.
\end{enumerate}
\end{thm}

For any $\boldsymbol{\lambda}\in K_n$ such that $G_{\boldsymbol{\lambda}}$ is isomorphic to the dihedral group $D_p$, $[\boldsymbol{\lambda}]$ is equivalent to one of the following sets (which are categorized into six types):
\begin{enumerate}
\item \textbf{mC}: $C_p(a_1)\cup C_p(a_2)\cup\cdots\cup C_p(a_m)$, where $C_p(a_1),\:C_p(a_2),\:\cdots,\:C_p(a_m)$ are different orbits of order $2p$, $m\in\mathbb {N^*}$;
\item \textbf{A+mC}: $A_p\cup C_p(a_1)\cup C_p(a_2)\cup\cdots\cup C_p(a_m)$, where $C_p(a_1),\:C_p(a_2),\:\cdots,\:C_p(a_m)$ are different orbits of order $2p$, $m\in\mathbb {N}$;
\item \textbf{AB+mC}: $A_p\cup B_p\cup C_p(a_1)\cup C_p(a_2)\cup\cdots\cup C_p(a_m)$, where $C_p(a_1),\:C_p(a_2),\:\cdots,\:C_p(a_m)$ are different orbits of order $2p$, $m\in\mathbb {N^*}$;
\item \textbf{2+mC}: $\{0,\:\infty\}\cup C_p(a_1)\cup C_p(a_2)\cup\cdots\cup C_p(a_m)$, where $C_p(a_1),\:C_p(a_2),\:\cdots,\:C_p(a_m)$ are different orbits of order $2p$, $m\in\mathbb {N^*}$;
\item \textbf{A+2+mC}: $A_p\cup\{0,\:\infty\}\cup C_p(a_1)\cup C_p(a_2)\cup\cdots\cup C_p(a_m)$, where $C_p(a_1),\:C_p(a_2),\:\cdots,\:C_p(a_m)$ are different orbits of order $2p$, $m\in\mathbb {N}$;
\item \textbf{AB+2+mC}: $A_p\cup B_p\cup\{0,\:\infty\}\cup C_p(a_1)\cup C_p(a_2)\cup\cdots\cup C_p(a_m)$, where $C_p(a_1),\:C_p(a_2),\:\cdots,\:C_p(a_m)$ are different orbits of order $2p$, $m\in\mathbb {N^*}$.
\end{enumerate}

\begin{thm}
For any $\boldsymbol{\lambda}\in K_n$ such that $G_{\boldsymbol{\lambda}}$ is isomorphic to the dihedral group $D_p$ ($p$ is odd), we have found its multiplicity vector
\begin{enumerate}
  \item \textbf{mC}:
    \begin{itemize}
      \item $(2m-1,\:m,\:m-1)$ for $p=3$;
      \item $(2m-1,\:2m,\:\cdots,\:2m,\:m,\:m-1)$ for $p\ge5$;
    \end{itemize}
  \item \textbf{A+mC}:
    \begin{itemize}
      \item $(2m,\:m,\:m)$ for $p=3$;
      \item $(2m,\:2m+1,\:\cdots,\:2m+1,\:m,\:m)$ for $p\ge5$;
    \end{itemize}
  \item \textbf{AB+mC}:
    \begin{itemize}
      \item $(2m+1,\:m,\:m+1)$ for $p=3$;
      \item $(2m+1,\:2m+2,\:\cdots,\:2m+2,\:m,\:m+1)$ for $p\ge5$;
    \end{itemize}
  \item \textbf{2+mC}: $(2m,\:\cdots,\:2m,\:m,\:m-1)$;
  \item \textbf{A+2+mC}: $(2m+1,\:\cdots,\:2m+1,\:m,\:m)$;
  \item \textbf{AB+2+mC}: $(2m+2,\:\cdots,\:2m+2,\:m,\:m+1)$.
\end{enumerate}
\end{thm}

\begin{thm}
For any $\boldsymbol{\lambda}\in K_n$ such that $G_{\boldsymbol{\lambda}}$ is isomorphic to the dihedral group $D_p$ ($p$ is even), we have found its multiplicity vector
\begin{enumerate}
  \item \textbf{mC}:
    \begin{itemize}
      \item $(m,\:m-1,\:m-1,\:m-1)$ for $p=2$;
      \item $(2m-1,\:m,\:m-1,\:m,\:m)$ for $p=4$;
      \item $(2m-1,\:2m,\:\cdots,\:2m,\:m,\:m-1,\:m,\:m)$ for $p\ge6$;
    \end{itemize}
  \item \textbf{A+mC}:
    \begin{itemize}
      \item $(m,\:m,\:m-1,\:m)$ for $p=2$;
      \item $(2m,\:m,\:m,\:m,\:m+1)$ for $p=4$;
      \item $(2m,\:2m+1,\:\cdots,\:2m+1,\:m,\:m,\:m,\:m+1)$ for $p\ge6$;
    \end{itemize}
  \item \textbf{AB+mC}:
    \begin{itemize}
      \item $(m,\:m+1,\:m,\:m)$ for $p=2$;
      \item $(2m+1,\:m,\:m+1,\:m+1,\:m+1)$ for $p=4$;
      \item $(2m+1,\:2m+2,\:\cdots,\:2m+2,\:m,\:m+1,\:m+1,\:m+1)$ for $p\ge6$;
    \end{itemize}
  \item \textbf{2+mC}:
    \begin{itemize}
      \item $(m,\:m-1,\:m,\:m)$ for $p=2$;
      \item $(2m,\:\cdots,\:2m,\:m,\:m-1,\:m,\:m)$ for $p\ge4$;
    \end{itemize}
  \item \textbf{A+2+mC}:
    \begin{itemize}
      \item $(m,\:m,\:m,\:m+1)$ for $p=2$;
      \item $(2m+1,\:\cdots,\:2m+1,\:m,\:m,\:m,\:m+1)$ for $p\ge4$;
    \end{itemize}
  \item \textbf{AB+2+mC}:
    \begin{itemize}
      \item $(m,\:m+1,\:m+1,\:m+1)$ for $p=2$;
      \item $(2m+2,\:\cdots,\:2m+2,\:m,\:m+1,\:m+1,\:m+1)$ for $p\ge4$.
    \end{itemize}
\end{enumerate}
\end{thm}

For any $\boldsymbol{\lambda}\in K_n$ such that $G_{\boldsymbol{\lambda}}$ is isomorphic to the cyclic group $\mathbb{Z}_p$, $[\boldsymbol{\lambda}]$ is equivalent to one of the following sets (which are categorized into three types):
\begin{enumerate}
\item \textbf{mC}: $C_p(a_1)\cup C_p(a_2)\cup\cdots\cup C_p(a_m)$, where $C_p(a_1),\:C_p(a_2),\:\cdots,\:C_p(a_m)$ are different orbits of order $p$, $m\in\mathbb {N^*}$;
\item \textbf{1+mC}: $\{0\}\cup C_p(a_1)\cup C_p(a_2)\cup\cdots\cup C_p(a_m)$, where $C_p(a_1),\:C_p(a_2),\:\cdots,\:C_p(a_m)$ are different orbits of order $p$, $m\in\mathbb {N^*}$;
\item \textbf{2+mC}: $\{0,\:\infty\}\cup C_p(a_1)\cup C_p(a_2)\cup\cdots\cup C_p(a_m)$, where $C_p(a_1),\:C_p(a_2),\:\cdots,\:C_p(a_m)$ are different orbits of order $p$, $m\in\mathbb {N^*}$.
\end{enumerate}

\begin{thm}
For any $\boldsymbol{\lambda}\in K_n$ such that $G_{\boldsymbol{\lambda}}$ is isomorphic to the cyclic group $\mathbb{Z}_p$, we have found its multiplicity vector
\begin{enumerate}
\item \textbf{mC}:
\begin{itemize}
\item $(m-2,\:m-1)$ for $p=2$;
\item $(m-1,\:m-1,\:m-1)$ for $p=3$;
\item $(m-1,\:m,\:\cdots,\:m,\:m-1,\:m-1)$ for $p\ge4$;
\end{itemize}
\item \textbf{1+mC}:
\begin{itemize}
\item $(m-1,\:m-1)$ for $p=2$;
\item $(m,\:\cdots,\:m,\:m-1,\:m-1)$ for $p\ge3$.
\end{itemize}
\item \textbf{2+mC}: $(m,\:\cdots,\:m,\:m-1)$.
\end{enumerate}
\end{thm}

%----------------------------------------------------------------
\subsection{The Group that Fixes Five of Six Points}

In Section \ref{n=5} and \ref{n=6} we will investigate $\mathcal A_\alpha$ when $|\alpha|=5$ or $6$, and find their classifications and multiplicity vectors.

From the above discussion we shall drive the following conclusion

\begin{thm}

Set $\alpha=\{z_1,\: z_2,\: z_3,\: z_4,\: z_5\}\subseteq\widehat {\mathbb{C}}$.

\begin{enumerate}

\item If there exists some linear fractional transformation $\psi$ such that
$$
\psi(\alpha)=\{1,\: w,\: w^2,\: w^3\:, w^4\},\:w=e^{\frac{2\pi}{5}i},
$$
then
$$
\mathcal{A_\alpha}={\psi}^{-1}\langle z \mapsto e^{\frac{2\pi}{5}i}z,\:z \mapsto \frac{1}{z}\rangle {\psi}\simeq D_5,
$$
and its multiplicity vector is $(0,\:1,\:0,\:0)$;

\item if there exists some linear fractional transformation $\psi$ such that
$$
\psi(\alpha) =\{0,\: 1,\: i,\: -1,\: -i\},
$$
then
$$
\mathcal{A_\alpha}={\psi}^{-1}\langle z \mapsto iz\rangle {\psi}\simeq \mathbb Z _4,
$$
and its multiplicity vector is $(1,\:1,\:0,\:0)$;

\item if there exists some linear fractional transformation $\psi$ such that
$$
\psi(\alpha)=\{0,\: \infty,\: 1, \:w,\: w^2\},\:w=e^{\frac{2\pi}{3}i},
$$
then
$$
\mathcal{A_\alpha}={\psi}^{-1}\langle z \mapsto e^{\frac{2\pi}{3}i}z,\:z \mapsto \frac{1}{z}\rangle\psi \simeq D_3,
$$
and its multiplicity vector is $(1,\:0,\:0)$;

\item if there exists some linear fractional transformation $\psi$ such that
$$
\psi(\alpha)=\{0,\: 1, \:-1,\: a,\: -a\},\:a\ne0,\:\pm 1,
$$
then

\begin{enumerate}

\item if $a=\pm(\sqrt{5}+2)\text{ or }\pm(\sqrt{5}-2)$, then there exists some linear fractional transformation $\phi$ such that
$$
\phi(\alpha)=\{1,\: w,\: w^2,\: w^3\:, w^4\},\:w=e^{\frac{2\pi}{5}i}
$$
and this is case 1;
\item if $a=\pm i$, then
$$
\alpha =\{0,\: 1,\: i,\: -1,\: -i\}
$$
and this is case 2;

\item if $a=\pm\sqrt{3}i\text{ or }\pm\frac{1}{\sqrt{3}}i$, then there exists some linear fractional transformation $\phi$ such that
$$
\phi(\alpha)=\{0,\: \infty,\: 1, \:w,\: w^2\},\:w=e^{\frac{2\pi}{3}i}
$$
and this is case 3;

\item otherwise,
$$
\mathcal{A_\alpha}={\psi}^{-1}\langle z \mapsto -z\rangle {\psi}\simeq \mathbb Z _2,
$$
and its multiplicity vector is $(1,\:1)$;

\end{enumerate}

\item otherwise,
$$\mathcal{A_\alpha}=\{\mathrm{Id}\}.$$

\end{enumerate}

\end{thm}

\begin{thm}

Set $\alpha=\{z_1,\: z_2,\: z_3,\: z_4,\: z_5,\: z_6\}\subseteq \widehat{\mathbb{C}}$.

\begin{enumerate}

\item If there exists some linear fractional transformation $\psi$ such that
$$
\psi(\alpha)=\{1,\: w,\: w^2,\: w^3,\: w^4, \:w^5\},\:w=e^{\frac{\pi}{3}i},
$$
then
$$
\mathcal{A_\alpha}={\psi}^{-1}\langle z \mapsto e^{\frac{\pi}{3}i}z,\:z \mapsto \frac{1}{z}\rangle {\psi}\simeq D_6,
$$
and its multiplicity vector is $(0,\:1,\:0,\:0,\:0,\:1)$;

\item if there exists some linear fractional transformation $\psi$ such that
$$
\psi(\alpha) =\{0,\: 1,\: w, \:w^2, \:w^3,\: w^4\},\:w=e^{\frac{2\pi}{5}i},
$$
then
$$
\mathcal{A_\alpha}={\psi}^{-1}\langle z \mapsto e^{\frac{2\pi}{5}i}z\rangle {\psi}\simeq \mathbb Z _5,
$$
and its multiplicity vector is $(1,\:1,\:1,\:0,\:0)$;

\item if there exists some linear fractional transformation $\psi$ such that
$$\psi(\alpha) =\{0,\: \infty, \:1, \:i, \:-1,\: -i\},$$
then
$$\mathcal{A_\alpha}={\psi}^{-1}\langle z\mapsto iz,\: z\mapsto \frac{iz+1}{z+i}\rangle {\psi}\simeq S_4,
$$
and its multiplicity vector is $(0,\:0,\:1,\:0,\:0)$;

\item if there exists some linear fractional transformation $\psi$ such that
$$
\psi(\alpha)=\{1,\:w,\:w^2,\:a,\:aw,\:aw^2\},\:|a|\ge1,\:a\ne1,\:w,\:w^2,\:w=e^{\frac{2\pi}{3}i},
$$
\begin{enumerate}

\item if $a=\sqrt w,\:w\sqrt w\text{ or }w^2\sqrt w$,
$$\alpha=\{1,\:\sqrt w,\: w,\:w\sqrt w,\: w^2,\:w^2\sqrt w\}$$
and this is case 1 ;

\item if $a=-(2+\sqrt3),\:-(2+\sqrt3)w\text{ or }-(2+\sqrt3)w^2$, then there exists some linear fractional transformation $\phi$ such that
$$\phi(\alpha) =\{0,\: \infty, \:1, \:i, \:-1,\: -i\}$$
and this is case 3;

\item otherwise
$$\mathcal{A_\alpha}={\psi}^{-1}\langle z \mapsto e^{\frac{2\pi}{3}i}z,\:z \mapsto \frac{a}{z}\rangle\psi \simeq D_3,
$$
and its multiplicity vector is $(1,\:1,\:0)$; \label{a}

\end{enumerate}

\item if there exists some linear fractional transformation $\psi$ such that
$$\psi(\alpha)=\{1,\: -1, \:a,\: -a,\: b, \:-b\},\:a\ne\pm b, \:a,\:b \in \mathbb{C} \backslash \{ 0, \pm 1 \},$$

\begin{enumerate}

\item if $\{\pm1,\:\pm a,\: \pm b\}=$
$$\{\pm(7-4\sqrt3), \:\pm(2-\sqrt3),\:\pm 1\},\:\{\pm(2-\sqrt3),\:\pm 1,\:\pm(2+\sqrt3)\},\:\{\pm 1,\: \pm(2+\sqrt3),\:\pm(7+4\sqrt3)\}$$
then there exists some linear fractional transformation $\phi$ such that
$$\phi(\alpha)=\{1,\: w,\: w^2,\: w^3,\: w^4, \:w^5\},\:w=e^{\frac{\pi}{3}i}$$
and this is case 1;

\item if $\{\pm1,\:\pm a,\: \pm b\}=$
$$\{\pm (\sqrt2-1)i,\:\pm(\sqrt2+1)i,\:\pm 1\},\:\{\pm(3-2\sqrt2), \:\pm 1,\pm(\sqrt2-1)i\},\:\{\pm 1, \:\pm(3+2\sqrt2),\:\pm(\sqrt2+1)i\},$$
then there exists some linear fractional transformation $\phi$ such that
$$\phi(\alpha) =\{0,\: \infty, \:1, \:i, \:-1,\: -i\}$$
and this is case 3;

\item if $a^2b+ab^2+a^2-6ab+b^2+a+b=0,$ then there exists some linear fractional transformation $\phi$ such that
$$\phi(\alpha)=\{1,\:w,\:w^2,\:c,\:cw,\:cw^2\},\:c\ne0,\:1,\:w,\:w^2,\:w=e^{\frac{2\pi}{3}i}$$
and this is case 4c;

\item otherwise,

\begin{enumerate}
 \item if $ab=\pm1$,
$$
\mathcal{A_\alpha}={\psi}^{-1}\langle z \mapsto -z,\:z \mapsto  \frac{1}{z}\rangle\psi \simeq K_4,
$$
and its multiplicity vector is $(1,\:0,\:1,\:1)$ (viewed as type 2+mC);
 \item if $ab\ne\pm1$,
$$
\mathcal{A_\alpha}={\psi}^{-1}\langle z \mapsto -z\rangle\psi \simeq \mathbb{Z}_2,
$$
and its multiplicity vector is $(1,\:2)$;
\end{enumerate}

\end{enumerate}

\item otherwise,
$$\mathcal{A_\alpha}=\{\mathrm{Id}\}.$$

\end{enumerate}

\end{thm}

Finally we will come to the conclusion in Section \ref{all} that every non-trivial finite group of linear fractional transformations is a stabilizer of certain finite subset of $\widehat{\mathbb{C}}$.
\begin{thm}
For any finite non-trivial group $G$ of linear fractional transformations, there exists a subset $\alpha=\{z_1, z_2, \cdots, z_n\}\subseteq \widehat{\mathbb{C}}$ such that $\mathcal{A}_{\alpha}\simeq G$.
\end{thm}

%----------------------------------------------------------------
\section{The Moduli Space $\mathfrak{M_{0,\:n}}$}
\label{orbifold}

It is well-known that
\begin{thm}[Deligne-Mumford]
The moduli space $\mathfrak{M_{g,\:n}}$ is both a complex orbifold and an irreducible quasiprojective variety of dimension $3g-3+n$, where $n\ge3$ if $g=0$, $n\ge1$ if $g=1$ and $n\ge0$ if $g\ge2$.\cite{De-Mu}
\end{thm}

In this section using elementary methods we shall prove that $\mathfrak{M_{0,\:n}}$ is a complex orbifold of dimension $n-3$ when $n\ge4$. Moreover, we shall show that it is the quotient space of a Zariski open set $K_n$ of $\mathbb C^{n-3}$ by an $S_n$ action when $n\ge5$. The stabilizers of this $S_n$ action at points of this Zariski open set correspond to the groups fixing the sets of $n$ points on the Riemann sphere.

%----------------------------------------------------------------
\subsection{The Orbifold Structure of Moduli Space $\mathfrak{M_{0,\:n}}$}
\label{orbif}

For $n\in\mathbb Z_{\ge4}$, set
$$
K_n=\{\boldsymbol{\lambda}=(\lambda^1,\:\lambda^2,\:\cdots,\:\lambda^{n-3})\in\mathbb C^{n-3}|\lambda^i\ne0,\:1,\:\lambda^i\ne\lambda^j,\:\forall i,\:j=1,\:2,\:\cdots,\:n-3,\:i\ne j\}.
$$
For any $\boldsymbol{\lambda}\in K_n$, set
$$
z^{\boldsymbol{\lambda}}_1=0,\:z^{\boldsymbol{\lambda}}_2=1,\:z^{\boldsymbol{\lambda}}_3=\infty,\:z^{\boldsymbol{\lambda}}_{i+3}=\lambda^{i},\:i=1,\:2,\:\cdots,\:n-3.
$$
For any $\boldsymbol{\lambda}\in K_n$ and $\sigma\in S_n$, define $f^{\boldsymbol {\lambda}}_\sigma$ as the unique linear fractional transformation such that
$$
f^{\boldsymbol {\lambda}}_\sigma (z^{\boldsymbol{\lambda}}_{\sigma^{-1}(1)})=0,\:f^{\boldsymbol {\lambda}}_\sigma (z^{\boldsymbol{\lambda}}_{\sigma^{-1}(2)})=1,\:f^{\boldsymbol {\lambda}}_\sigma (z^{\boldsymbol{\lambda}}_{\sigma^{-1}(3)})=\infty.
$$
For $i=1,\:2,\:\cdots,\:n-3$ and $\sigma\in S_n$ define a mapping $g^{(i)}_\sigma$
$$
g^{(i)}_\sigma:\:K_n\to\mathbb{C},\:\boldsymbol{\lambda}\mapsto f^{\boldsymbol {\lambda}}_\sigma (z^{\boldsymbol{\lambda}}_{\sigma^{-1}(i+3)}),
$$
and $g_\sigma$
$$
g_\sigma :\:K_n\to K_n,\:\boldsymbol{\lambda}\mapsto(f^{\boldsymbol {\lambda}}_\sigma (z^{\boldsymbol{\lambda}}_{\sigma^{-1}(4)}),\:f^{\boldsymbol {\lambda}}_\sigma (z^{\boldsymbol{\lambda}}_{\sigma^{-1}(5)}),\:\cdots,\:f^{\boldsymbol {\lambda}}_\sigma (z^{\boldsymbol{\lambda}}_{\sigma^{-1}(n)})).
$$
By definition $g_\sigma$ is a rational map from $\mathbb C^{n-3}$ to $\mathbb C^{n-3}$ which restricts to an isomorphism of $K_n$. Here is a useful observation:
\begin{rem}
For any $\boldsymbol{\lambda}\in K_n$ and $\sigma\in S_n$, we have
$$
f^{\boldsymbol {\lambda}}_\sigma (z^{\boldsymbol{\lambda}}_k)=z^{g_\sigma(\boldsymbol{\lambda})}_{\sigma(k)}
$$
holds for $k=1,\:2,\:\cdots,\:n$.
\end{rem}

Define $G$
$$
G=\{g_\sigma|\sigma\in S_n\}.
$$
Notice that $|G|\le n!$ and
\begin{thm}
$G$ is a finite group acting on $K_n$.
\end{thm}
\begin{proof}
To prove the theorem first notice that
$$
f^{\boldsymbol {\lambda}}_\sigma (z^{\boldsymbol{\lambda}}_{\sigma^{-1}(k)})=z^{g_\sigma(\boldsymbol{\lambda})}_k,\:k=1,\:2,\:\cdots,\:n.
$$
So we have
$$
f^{g_\sigma(\boldsymbol {\lambda})}_\pi\circ f^{\boldsymbol {\lambda}}_\sigma (z^{\boldsymbol{\lambda}}_{(\pi\cdot\sigma)^{-1}(k)})=f^{g_\sigma(\boldsymbol {\lambda})}_\pi(z^{g_\sigma(\boldsymbol{\lambda})}_{\pi^{-1}(k)})=z^{g_\pi\circ g_\sigma(\boldsymbol {\lambda})}_k,\:k=1,\:2,\:\cdots,\:n.
$$
Thus we conclude that
$$
f^{g_\sigma(\boldsymbol {\lambda})}_\pi\circ f^{\boldsymbol {\lambda}}_\sigma =f^{\boldsymbol {\lambda}}_{\pi\cdot\sigma},\:z^{g_\pi\circ g_\sigma(\boldsymbol {\lambda})}_k=z_k^{g_{\pi\cdot\sigma}(\boldsymbol {\lambda})},\:k=1,\:2,\:\cdots,\:n.
$$
So $g_\pi\circ g_\sigma=g_{\pi\cdot\sigma}\in G$. Thus $G$ is a group.
\end{proof}

As $G$ is a finite group acting on $K_n$, we have another conclusion
\begin{thm}
$\mathfrak{M_{0,\:n}}=K_n/G.$
\end{thm}
\begin{proof}
Given two elements
$$
\overline{\{0,\:1,\:\infty,\:\lambda^1,\:\lambda^2,\:\cdots,\:\lambda^{n-3}\}},\:\overline{\{0,\:1,\:\infty,\:\mu^1,\:\mu^2,\:\cdots,\:\mu^{n-3}\}}\in \mathfrak{M_{0,\:n}},
$$
set
$$
\boldsymbol{\lambda}=(\lambda^1,\:\lambda^2,\:\cdots,\:\lambda^{n-3}),\:\boldsymbol{\mu}=(\mu^1,\:\mu^2,\:\cdots,\:\mu^{n-3}).
$$

If $\boldsymbol{\mu}\in G(\boldsymbol{\lambda})$, then there exists some $\sigma\in S_n$ s.t.~$\boldsymbol{\mu}=g_\sigma(\boldsymbol{\lambda}).$ So we have
$$
f^{\boldsymbol {\lambda}}_\sigma (z^{\boldsymbol{\lambda}}_{\sigma^{-1}(k)})=z^{g_\sigma(\boldsymbol{\lambda})}_k=z^{\boldsymbol{\mu}}_k,\:k=1,\:2,\:\cdots,\:n,
$$
which means
$$
f^{\boldsymbol {\lambda}}_\sigma(\{0,\:1,\:\infty,\:\lambda^1,\:\lambda^2,\:\cdots,\:\lambda^{n-3}\})=\{0,\:1,\:\infty,\:\mu^1,\:\mu^2,\:\cdots,\:\mu^{n-3}\}.
$$
So
$$
\overline{\{0,\:1,\:\infty,\:\lambda^1,\:\lambda^2,\:\cdots,\:\lambda^{n-3}\}}=\overline{\{0,\:1,\:\infty,\:\mu^1,\:\mu^2,\:\cdots,\:\mu^{n-3}\}}.
$$

On the other hand, if
$$
\overline{\{0,\:1,\:\infty,\:\lambda^1,\:\lambda^2,\:\cdots,\:\lambda^{n-3}\}}=\overline{\{0,\:1,\:\infty,\:\mu^1,\:\mu^2,\:\cdots,\:\mu^{n-3}\}},
$$
then there exists some linear fractional transformation $h$ s.t.
$$
h(\{0,\:1,\:\infty,\:\lambda^1,\:\lambda^2,\:\cdots,\:\lambda^{n-3}\})=\{0,\:1,\:\infty,\:\mu^1,\:\mu^2,\:\cdots,\:\mu^{n-3}\}
$$
and some $\sigma\in S_n$ s.t.~
$$
h(z^{\boldsymbol{\lambda}}_{\sigma^{-1}(k)})=z^{\boldsymbol{\mu}}_k,\:k=1,\:2,\:\cdots,\:n.
$$
So $h=f^{\boldsymbol {\lambda}}_\sigma$ and $\boldsymbol{\mu}=g_\sigma(\boldsymbol{\lambda})$.

Thus we conclude that
$$
\overline{\{0,\:1,\:\infty,\:\lambda^1,\:\lambda^2,\:\cdots,\:\lambda^{n-3}\}}=\overline{\{0,\:1,\:\infty,\:\mu^1,\:\mu^2,\:\cdots,\:\mu^{n-3}\}}
$$
if and only if
$$
\boldsymbol{\mu}\in G(\boldsymbol{\lambda}),
$$
and therefore
$$
\mathfrak{M_{0,\:n}}=K_n/G.
$$
\end{proof}

\begin{cor}
$\mathfrak{M_{0,\:n}}$ is a complex orbifold of dimension $n-3$.
\end{cor}

%----------------------------------------------------------------
\subsection{The Explicit Structure of the Group $G$}
\label{grou}

Now let us explore $G$. %The map
%\begin{displaymath}
%p:\: S_n\to G,\:\sigma\mapsto g_{\sigma^{-1}}
%\end{displaymath}
%satisfies
%\begin{displaymath}
%p(\sigma\cdot\pi)=g_{(\sigma\cdot\pi)^{-1}}=g_{\pi^{-1}\cdot\sigma^{-1}}=g_{\sigma^{-1}}\circ g_{\pi^{-1}}=p(\sigma)\circ p(\pi).
%\end{displaymath}
%So $p$ is an group epimorphism.
Define subgroup $V$ of $S_n$
\begin{displaymath}
V=S_{\{1,\:2,\:3\}}\times S_{\{4,\:5,\:\cdots,\:n\}}
\end{displaymath}
and subset $S$ of $S_n$
\[
\begin{split}
S=\{&e,\:(p,\:1),\:(p,\:2),\:(p,\:3),\:(p,\:1)(q,\:2),\:(p,\:2)(q,\:3),\:(p,\:3)(q,\:1),\:(p,\:1)(q,\:2)(r,\:3)|\\
    &p,\:q,\:r=4,\:5,\:\cdots,\:n,\:p\ne q,\:q\ne r,\:r\ne p\}.
\end{split}
\]
Thus $S$ forms a complete list of left coset representatives of $V\subset S_n$.

Define group $H$ of linear fractional transformations
\begin{displaymath}
H=\{\mathrm{Id},\:\lambda\mapsto 1-\lambda,\: \lambda\mapsto \frac{1}{\lambda},\:\lambda\mapsto\frac{\lambda}{\lambda-1} ,\:\lambda\mapsto \frac{\lambda-1}{\lambda},\:\lambda\mapsto \frac{-1}{\lambda-1}\},
\end{displaymath}
then we have
\begin{displaymath}
\{g_\sigma|\sigma\in V\}=\{\boldsymbol{\lambda}\mapsto(h(\lambda^{\tau(1)}),\:h(\lambda^{\tau(2)}),\:\cdots,\:h(\lambda^{\tau(n-3)}))|h\in H,\:\tau\in S_{n-3}\}.
\end{displaymath}

For $p=4,\:5,\:\cdots,\:n$, set $\sigma=(1,\:p)$ and we have
\begin{displaymath}
f^{\boldsymbol {\lambda}}_{(1,\:p)}(\lambda^{p-3})=
f^{\boldsymbol {\lambda}}_{(1,\:p)} (z^{\boldsymbol{\lambda}}_{\sigma^{-1}(1)})=0,
\end{displaymath}
\begin{displaymath}
f^{\boldsymbol {\lambda}}_{(1,\:p)}(1)=
f^{\boldsymbol {\lambda}}_{(1,\:p)}(z^{\boldsymbol{\lambda}}_{\sigma^{-1}(2)})=1,
\end{displaymath}
\begin{displaymath}
f^{\boldsymbol {\lambda}}_{(1,\:p)}(\infty)=
f^{\boldsymbol {\lambda}}_{(1,\:p)}(z^{\boldsymbol{\lambda}}_{\sigma^{-1}(3)})=\infty.
\end{displaymath}
So we have
\begin{displaymath}
f^{\boldsymbol {\lambda}}_{(1,\:p)}(z)=\frac{z-\lambda^{p-3}}{1-\lambda^{p-3}},
\end{displaymath}
and
$$
g_{(1,\:p)}^{(k)}(\boldsymbol {\lambda})=
\begin{dcases}
\frac{\lambda^k-\lambda^{p-3}}{1-\lambda^{p-3}}, &k\ne p-3; \\
\frac{-\lambda^{p-3}}{1-\lambda^{p-3}},          &k=p-3.
\end{dcases}
$$
Set $\sigma=(2,\:p)$ and we have
\begin{displaymath}
f^{\boldsymbol {\lambda}}_{(2,\:p)}(0)
=f^{\boldsymbol {\lambda}}_{(2,\:p)}(z^{\boldsymbol{\lambda}}_{\sigma^{-1}(1)})=0,
\end{displaymath}
\begin{displaymath}
f^{\boldsymbol {\lambda}}_{(2,\:p)}(\lambda^{p-3})
=f^{\boldsymbol {\lambda}}_{(2,\:p)} (z^{\boldsymbol{\lambda}}_{\sigma^{-1}(2)})=1,
\end{displaymath}
\begin{displaymath}
f^{\boldsymbol {\lambda}}_{(2,\:p)}(\infty)
=f^{\boldsymbol {\lambda}}_{(2,\:p)}(z^{\boldsymbol{\lambda}}_{\sigma^{-1}(3)})=\infty.
\end{displaymath}
So we have
\begin{displaymath}
f^{\boldsymbol {\lambda}}_{(2,\:p)}(z)=\frac{z}{\lambda^{p-3}},
\end{displaymath}
and
$$
g_{(2,\:p)}^{(k)}(\boldsymbol {\lambda})=
\begin{dcases}
\frac{\lambda^k}{\lambda^{p-3}}, &k\ne p-3; \\
\frac{1}{\lambda^{p-3}},         &k=p-3.
\end{dcases}
$$
Set $\sigma=(3,\:p)$ and we have
\begin{displaymath}
f^{\boldsymbol {\lambda}}_{(3,\:p)}(0)=
f^{\boldsymbol {\lambda}}_{(3,\:p)} (z^{\boldsymbol{\lambda}}_{\sigma^{-1}(1)})=0,
\end{displaymath}
\begin{displaymath}
f^{\boldsymbol {\lambda}}_{(3,\:p)}(1)=
f^{\boldsymbol {\lambda}}_{(3,\:p)}(z^{\boldsymbol{\lambda}}_{\sigma^{-1}(2)})=1,
\end{displaymath}
\begin{displaymath}
f^{\boldsymbol {\lambda}}_{(3,\:p)}(\lambda^{p-3})=
f^{\boldsymbol {\lambda}}_{(3,\:p)}(z^{\boldsymbol{\lambda}}_{\sigma^{-1}(3)})=\infty.
\end{displaymath}
So we have
\begin{displaymath}
f^{\boldsymbol {\lambda}}_{(3,\:p)}(z)=\frac{z(1-\lambda^{p-3})}{z-\lambda^{p-3}},
\end{displaymath}
and
$$
g_{(3,\:p)}^{(k)}(\boldsymbol {\lambda})=
\begin{dcases}
\frac{\lambda^k(1-\lambda^{p-3})}{\lambda^k-\lambda^{p-3}}, &k\ne p-3; \\
\frac{\lambda^{p-3}-1}{\lambda^{p-3}},                      &k=p-3.
\end{dcases}
$$

For $p,\:q=4,\:5,\:\cdots,\:n$ and $p\ne q$ set $\sigma=(1,\:p)(2\:,q)$ and we have
\begin{displaymath}
f^{\boldsymbol {\lambda}}_{(1,\:p)(2\:,q)}(\lambda^{p-3})=
f^{\boldsymbol {\lambda}}_{(1,\:p)(2\:,q)} (z^{\boldsymbol{\lambda}}_{\sigma^{-1}(1)})=0,
\end{displaymath}
\begin{displaymath}
f^{\boldsymbol {\lambda}}_{(1,\:p)(2\:,q)}(\lambda^{q-3})=
f^{\boldsymbol {\lambda}}_{(1,\:p)(2\:,q)}(z^{\boldsymbol{\lambda}}_{\sigma^{-1}(2)})=1,
\end{displaymath}
\begin{displaymath}
f^{\boldsymbol {\lambda}}_{(1,\:p)(2\:,q)}(\infty)=
f^{\boldsymbol {\lambda}}_{(1,\:p)(2\:,q)}(z^{\boldsymbol{\lambda}}_{\sigma^{-1}(3)})=\infty.
\end{displaymath}
So we have
\begin{displaymath}
f^{\boldsymbol {\lambda}}_{(1,\:p)(2\:,q)}(z)=\frac{z-\lambda^{p-3}}{\lambda^{q-3}-\lambda^{p-3}},
\end{displaymath}
and
$$
g_{(1,\:p)(2\:,q)}^{(k)}(\boldsymbol {\lambda})=
\begin{dcases}
\frac{\lambda^k-\lambda^{p-3}}{\lambda^{q-3}-\lambda^{p-3}}, &k\ne p-3,\:q-3;\\
\frac{-\lambda^{p-3}}{\lambda^{q-3}-\lambda^{p-3}},          &k=p-3;\\
\frac{1-\lambda^{p-3}}{\lambda^{q-3}-\lambda^{p-3}},         &k=q-3.
\end{dcases}
$$
Set $\sigma=(2,\:p)(3,\:q)$ and we have
\begin{displaymath}
f^{\boldsymbol {\lambda}}_{(2,\:p)(3,\:q)}(0)
=f^{\boldsymbol {\lambda}}_{(2,\:p)(3,\:q)}(z^{\boldsymbol{\lambda}}_{\sigma^{-1}(1)})=0,
\end{displaymath}
\begin{displaymath}
f^{\boldsymbol {\lambda}}_{(2,\:p)(3,\:q)}(\lambda^{p-3})
=f^{\boldsymbol {\lambda}}_{(2,\:p)(3,\:q)} (z^{\boldsymbol{\lambda}}_{\sigma^{-1}(2)})=1,
\end{displaymath}
\begin{displaymath}
f^{\boldsymbol {\lambda}}_{(2,\:p)(3,\:q)}(\lambda^{q-3})
=f^{\boldsymbol {\lambda}}_{(2,\:p)(3,\:q)}(z^{\boldsymbol{\lambda}}_{\sigma^{-1}(3)})=\infty.
\end{displaymath}
So we have
\begin{displaymath}
f^{\boldsymbol {\lambda}}_{(2,\:p)(3,\:q)}(z)=\frac{(\lambda^{p-3}-\lambda^{q-3})z}{\lambda^{p-3}(z-\lambda^{q-3})},
\end{displaymath}
and
$$
g_{(2,\:p)(3\:,q)}^{(k)}(\boldsymbol {\lambda})=
\begin{dcases}
\frac{(\lambda^{p-3}-\lambda^{q-3})\lambda^k}{\lambda^{p-3}(\lambda^k-\lambda^{q-3})},&k\ne p-3,\:q-3;\\
\frac{\lambda^{p-3}-\lambda^{q-3}}{\lambda^{p-3}(1-\lambda^{q-3})},&k=p-3;\\
\frac{\lambda^{p-3}-\lambda^{q-3}}{\lambda^{p-3}},&k=q-3.
\end{dcases}
$$
Set $\sigma=(3,\:p)(1,\:q)$ and we have
\begin{displaymath}
f^{\boldsymbol {\lambda}}_{(3,\:p)(1,\:q)}(\lambda^{q-3})=
f^{\boldsymbol {\lambda}}_{(3,\:p)(1,\:q)} (z^{\boldsymbol{\lambda}}_{\sigma^{-1}(1)})=0,
\end{displaymath}
\begin{displaymath}
f^{\boldsymbol {\lambda}}_{(3,\:p)(1,\:q)}(1)=
f^{\boldsymbol {\lambda}}_{(3,\:p)(1,\:q)}(z^{\boldsymbol{\lambda}}_{\sigma^{-1}(2)})=1,
\end{displaymath}
\begin{displaymath}
f^{\boldsymbol {\lambda}}_{(3,\:p)(1,\:q)}(\lambda^{p-3})=
f^{\boldsymbol {\lambda}}_{(3,\:p)(1,\:q)}(z^{\boldsymbol{\lambda}}_{\sigma^{-1}(3)})=\infty.
\end{displaymath}
So we have
\begin{displaymath}
f^{\boldsymbol {\lambda}}_{(3,\:p)(1,\:q)}(z)=\frac{(\lambda^{p-3}-1)(z-\lambda^{q-3})}{(\lambda^{q-3}-1)(z-\lambda^{p-3})},
\end{displaymath}
and
$$
g_{(3,\:p)(1\:,q)}^{(k)}(\boldsymbol {\lambda})=
\begin{dcases}
\frac{(\lambda^{p-3}-1)(\lambda^k-\lambda^{q-3})}{(\lambda^{q-3}-1)(\lambda^k-\lambda^{p-3})},&k\ne p-3,\:q-3;\\
\frac{\lambda^{p-3}-1}{\lambda^{q-3}-1},&k=p-3;\\
\frac{(\lambda^{p-3}-1)\lambda^{q-3}}{(\lambda^{q-3}-1)\lambda^{p-3}},&k=q-3.
\end{dcases}
$$

For $p,\:q,\:r=4,\:5,\:\cdots,\:n$, and $p\ne q,\:q\ne r,\:r\ne p$, set $\sigma=(1,\:p)(2\:,q)(3,\:r)$ and we have
\begin{displaymath}
f^{\boldsymbol {\lambda}}_{(1,\:p)(2\:,q)(3,\:r)}(\lambda^{p-3})=
f^{\boldsymbol {\lambda}}_{(1,\:p)(2\:,q)(3,\:r)} (z^{\boldsymbol{\lambda}}_{\sigma^{-1}(1)})=0,
\end{displaymath}
\begin{displaymath}
f^{\boldsymbol {\lambda}}_{(1,\:p)(2\:,q)(3,\:r)}(\lambda^{q-3})=
f^{\boldsymbol {\lambda}}_{(1,\:p)(2\:,q)(3,\:r)}(z^{\boldsymbol{\lambda}}_{\sigma^{-1}(2)})=1,
\end{displaymath}
\begin{displaymath}
f^{\boldsymbol {\lambda}}_{(1,\:p)(2\:,q)(3,\:r)}(\lambda^{r-3})=
f^{\boldsymbol {\lambda}}_{(1,\:p)(2\:,q)(3,\:r)}(z^{\boldsymbol{\lambda}}_{\sigma^{-1}(3)})=\infty.
\end{displaymath}
So we have
\begin{displaymath}
f^{\boldsymbol {\lambda}}_{(1,\:p)(2\:,q)(3,\:r)}(z)=\frac{(\lambda^{q-3}-\lambda^{r-3})(z-\lambda^{p-3})}{(\lambda^{q-3}-\lambda^{p-3})(z-\lambda^{r-3})},
\end{displaymath}
and
$$
g_{(1,\:p)(2\:,q)(3,\:r)}^{(k)}(\boldsymbol {\lambda})=
\begin{dcases}
\frac{(\lambda^{q-3}-\lambda^{r-3})(\lambda^k-\lambda^{p-3})}{(\lambda^{q-3}-\lambda^{p-3})(\lambda^k-\lambda^{r-3})}, &k\ne p-3,\:q-3,\:r-3;\\
\frac{(\lambda^{q-3}-\lambda^{r-3})\lambda^{p-3}}{(\lambda^{q-3}-\lambda^{p-3})\lambda^{r-3}},&k=p-3;\\
\frac{(\lambda^{q-3}-\lambda^{r-3})(1-\lambda^{p-3})}{(\lambda-{q-3}-\lambda^{p-3})(1-\lambda^{r-3})},&k=q-3;\\
\frac{\lambda^{q-3}-\lambda^{r-3}}{\lambda^{q-3}-\lambda^{p-3}},&k=r-3.
\end{dcases}
$$

So we have the conclusions
\begin{thm}
$G=\tilde S\tilde V$, where
$$
\tilde V=\{\boldsymbol{\lambda}\mapsto(h(\lambda^{\tau(1)}),\:h(\lambda^{\tau(2)}),\:\cdots,\:h(\lambda^{\tau(n-3)}))|h\in H,\:\tau\in S_{n-3}\};
$$
$$
H=\{\mathrm{Id},\:\lambda\mapsto 1-\lambda,\: \lambda\mapsto \frac{1}{\lambda},\:\lambda\mapsto\frac{\lambda}{\lambda-1} ,\:\lambda\mapsto \frac{\lambda-1}{\lambda},\:\lambda\mapsto \frac{-1}{\lambda-1}\};
$$
\[
\begin{split}
\tilde S=\{g_\sigma|&\sigma=e,\:(p,\:1),\:(p,\:2),\:(p,\:3),\:(p,\:1)(q,\:2),\:(p,\:2)(q,\:3),\:(p,\:3)(q,\:1),\:(p,\:1)(q,\:2)(r,\:3)\\
&p,\:q,\:r=4,\:5,\:\cdots,\:n,\:p\ne q,\:q\ne r,\:r\ne p\};
\end{split}
\]
and
$$
g_{(1,\:p)}^{(k)}(\boldsymbol {\lambda})=
\begin{dcases}
\frac{\lambda^k-\lambda^{p-3}}{1-\lambda^{p-3}}, &k\ne p-3; \\
\frac{-\lambda^{p-3}}{1-\lambda^{p-3}},          &k=p-3;
\end{dcases}
$$
$$
g_{(2,\:p)}^{(k)}(\boldsymbol {\lambda})=
\begin{dcases}
\frac{\lambda^k}{\lambda^{p-3}}, &k\ne p-3; \\
\frac{1}{\lambda^{p-3}},         &k=p-3;
\end{dcases}
$$
$$
g_{(3,\:p)}^{(k)}(\boldsymbol {\lambda})=
\begin{dcases}
\frac{\lambda^k(1-\lambda^{p-3})}{\lambda^k-\lambda^{p-3}}, &k\ne p-3; \\
\frac{\lambda^{p-3}-1}{\lambda^{p-3}},                      &k=p-3;
\end{dcases}
$$
$$
g_{(1,\:p)(2\:,q)}^{(k)}(\boldsymbol {\lambda})=
\begin{dcases}
\frac{\lambda^k-\lambda^{p-3}}{\lambda^{q-3}-\lambda^{p-3}}, &k\ne p-3,\:q-3;\\
\frac{-\lambda^{p-3}}{\lambda^{q-3}-\lambda^{p-3}},          &k=p-3;\\
\frac{1-\lambda^{p-3}}{\lambda^{q-3}-\lambda^{p-3}},         &k=q-3;
\end{dcases}
$$
$$
g_{(2,\:p)(3\:,q)}^{(k)}(\boldsymbol {\lambda})=
\begin{dcases}
\frac{(\lambda^{p-3}-\lambda^{q-3})\lambda^k}{\lambda^{p-3}(\lambda^k-\lambda^{q-3})},&k\ne p-3,\:q-3;\\
\frac{\lambda^{p-3}-\lambda^{q-3}}{\lambda^{p-3}(1-\lambda^{q-3})},&k=p-3;\\
\frac{\lambda^{p-3}-\lambda^{q-3}}{\lambda^{p-3}},&k=q-3;
\end{dcases}
$$
$$
g_{(3,\:p)(1\:,q)}^{(k)}(\boldsymbol {\lambda})=
\begin{dcases}
\frac{(\lambda^{p-3}-1)(\lambda^k-\lambda^{q-3})}{(\lambda^{q-3}-1)(\lambda^k-\lambda^{p-3})},&k\ne p-3,\:q-3;\\
\frac{\lambda^{p-3}-1}{\lambda^{q-3}-1},&k=p-3;\\
\frac{(\lambda^{p-3}-1)\lambda^{q-3}}{(\lambda^{q-3}-1)\lambda^{p-3}},&k=q-3;
\end{dcases}
$$
$$
g_{(1,\:p)(2\:,q)(3,\:r)}^{(k)}(\boldsymbol {\lambda})=
\begin{dcases}
\frac{(\lambda^{q-3}-\lambda^{r-3})(\lambda^k-\lambda^{p-3})}{(\lambda^{q-3}-\lambda^{p-3})(\lambda^k-\lambda^{r-3})}, &k\ne p-3,\:q-3,\:r-3;\\
\frac{(\lambda^{q-3}-\lambda^{r-3})\lambda^{p-3}}{(\lambda^{q-3}-\lambda^{p-3})\lambda^{r-3}},&k=p-3;\\
\frac{(\lambda^{q-3}-\lambda^{r-3})(1-\lambda^{p-3})}{(\lambda^{q-3}-\lambda^{p-3})(1-\lambda^{r-3})},&k=q-3;\\
\frac{\lambda^{q-3}-\lambda^{r-3}}{\lambda^{q-3}-\lambda^{p-3}},&k=r-3.
\end{dcases}
$$
\end{thm}

\begin{cor}
$G$ is isomorphic to $S_n$ when $n\ge5$.
\end{cor}

%----------------------------------------------------------------
\subsection{The Singularities of Moduli Space $\mathfrak{M_{0,\:n}}$ }
\label{sing}

For $\boldsymbol {\lambda}\in K_n$, $n\ge5$, set
$$
[\boldsymbol {\lambda}]=\{0,\:1,\:\infty,\:\lambda^1,\:\lambda^2,\:\cdots,\:\lambda^{n-3}\}=\{z^{\boldsymbol {\lambda}}_k|k=1,\:2,\:\cdots,\:n\}
$$
and $G_{\boldsymbol \lambda}$ the stabilizer of $\boldsymbol {\lambda}$
$$
G_{\boldsymbol \lambda}=\{g_{\sigma}\in G|g_\sigma(\boldsymbol \lambda)=\boldsymbol \lambda\}.
$$
\begin{defn}
\label{def}
When $n\ge5$, call $\overline{[\boldsymbol {\lambda}]}\in \mathfrak{M_{0,\:n}}$ an oribfold singularity of $\mathfrak{M_{0,\:n}}$ if $G_{\boldsymbol \lambda}$ is non-trivial.
\end{defn}

Set
$$
\alpha=\{z_1, z_2, \cdots, z_n\}\subseteq \widehat{\mathbb{C}},
$$
and $\mathcal{A}_{\alpha}$ the group of linear fractional transformations fixing $\alpha$.

\begin{defn}
For any $\boldsymbol {\lambda}\in K_{n}$, $n\ge5$, define $\Phi_{\boldsymbol \lambda}$ as the mapping
$$
\Phi_{\boldsymbol \lambda}:\:G_{\boldsymbol \lambda}\to\mathcal{A}_{[\boldsymbol {\lambda}]},\: g_\sigma\mapsto f^{\boldsymbol {\lambda}}_\sigma.
$$
\end{defn}

For each $g_\sigma\in G_{\boldsymbol \lambda}$, we have
$$
f^{\boldsymbol {\lambda}}_\sigma (z^{\boldsymbol{\lambda}}_{\sigma^{-1} (k)})=z^{g_\sigma(\boldsymbol{\lambda})}_k=z^{\boldsymbol{\lambda}}_k,\:k=1,\:2,\:\cdots,\:n.
$$
Thus $f^{\boldsymbol {\lambda}}_\sigma\in\mathcal{A}_{[\boldsymbol {\lambda}]}$, and $\Phi_{\boldsymbol \lambda}$ is well-defined.

\begin{thm}
The map $\Phi_{\boldsymbol \lambda}$ is a group isomorphism between $G_{\boldsymbol \lambda}$ and $\mathcal{A}_{[\boldsymbol {\lambda}]}$.
\end{thm}
\begin{proof}
First notice that for $g_\pi,\:g_\sigma\in G_{\boldsymbol \lambda}$
$$
f^{\boldsymbol {\lambda}}_\pi\circ f^{\boldsymbol {\lambda}}_\sigma (z^{\boldsymbol{\lambda}}_{(\pi\cdot\sigma)^{-1}(k)}) =f^{\boldsymbol {\lambda}}_\pi(z^{\boldsymbol{\lambda}}_{\pi^{-1}(k)})=z^{\boldsymbol {\lambda}}_k,\:k=1,\:2,\:\cdots,\:n.
$$
So $f^{\boldsymbol {\lambda}}_\pi\circ f^{\boldsymbol {\lambda}}_\sigma =f^{\boldsymbol {\lambda}}_{\pi\cdot\sigma}$, which means that $\Phi_{\boldsymbol \lambda}$ is a group homomorphism.

For any $h\in\mathcal A_{[\boldsymbol {\lambda}]}$, there exists a unique $\sigma\in S_n$ such that
$$
h(z^{\boldsymbol {\lambda}}_{\sigma^{-1}(k)})=z^{\boldsymbol {\lambda}}_k,\:k=1,\:2,\:\cdots,\:n.
$$
Thus we have $h=f^{\boldsymbol \lambda}_{\sigma}$ and $g_\sigma(\boldsymbol \lambda)=\boldsymbol \lambda$, and
$$
\Phi_{\boldsymbol \lambda}(g_\sigma)=f^{\boldsymbol \lambda}_{\sigma}=h.
$$
So $\Phi_{\boldsymbol \lambda}$ is a group epimorphism.

For any $g_\sigma\in G_{\boldsymbol \lambda}$ such that
$$
\Phi_{\boldsymbol \lambda}(g_\sigma)=f^{\boldsymbol {\lambda}}_\sigma=\mathrm{Id},
$$
we have
$$
z^{\boldsymbol{\lambda}}_k
=f^{\boldsymbol {\lambda}}_\sigma (z^{\boldsymbol{\lambda}}_k)
=z^{\boldsymbol{\lambda}}_{\sigma(k)}
$$
holds for $k=1,\:2,\:\cdots,\:n$. Thus $\sigma$ is the identity element.

Hence $\Phi_{\boldsymbol \lambda}$ is a group isomorphism.
\end{proof}

Now the study of the orbifold singularities of $\mathfrak{M_{0,\:n}}$ for $n\geq 5$ can be reduced to that of $\mathcal{A_\alpha}$ for subset $\alpha$ with $n$ elements on the Riemann sphere. In Section \ref{orbit} I shall develope a method to find $\mathcal{A_\alpha}$ for an arbitrary finite subset $\alpha$ of the Riemann sphere when $|\alpha|\ge4$.

%----------------------------------------------------------------
\section{The Subset Fixed by a Specific Group}
\label{orbit}

Given a finite subset $\alpha=\{z_1, z_2, \cdots, z_n\}\subseteq \widehat{\mathbb{C}}$, $n\ge4$, let $\mathcal{A}_{\alpha}$ be the group of linear fractional transformations that fix $\alpha$. For any $f$ in $\mathcal{A}_{\alpha}$, since $f(z_1),\:f(z_2),\:f(z_3)$ are all in $\alpha$, there are at most $n(n-1)(n-2)$ ways to choose the images of $z_1, z_2, z_3$. Thus $\arrowvert \mathcal{A}_{\alpha} \arrowvert \le n(n-1)(n-2)$.

Let $G$ be a finite group of linear fractional transformations. There are only five kinds of non-trivial finite linear fractional transformation groups: the icosahedral group $I$, the octahedral group $O$, the tetrahedral group $T$, the dihedral group $D_n$ and the finite cyclic group $\mathbb Z_n$, $n\ge2$. In the rest of this section, by discussing the orbits of the action of $G$ on $S^2$ and $\widehat{\mathbb{C}}$, we shall find all finite subsets $\alpha$ of $\widehat{\mathbb{C}}$ such that $\mathcal{A}_{\alpha}\simeq G$. Thus given a finite $\alpha\subseteq \widehat{\mathbb{C}}$, $|\alpha|\ge4$, we can find $\mathcal{A}_{\alpha}$. Notice that we shall not distinguish a point on $S^2$ from its image on the extended complex plane under the stereographic projection.

%----------------------------------------------------------------
\subsection{$G$ is the Icosahedral Group $I$}
\label{A_5}

Suppose that $G$ is the icosahedral group $I$ that fixes a regular dodecahedron whose center is the origin.
\begin{figure}[!h]
\centering
\includegraphics[width=3in]{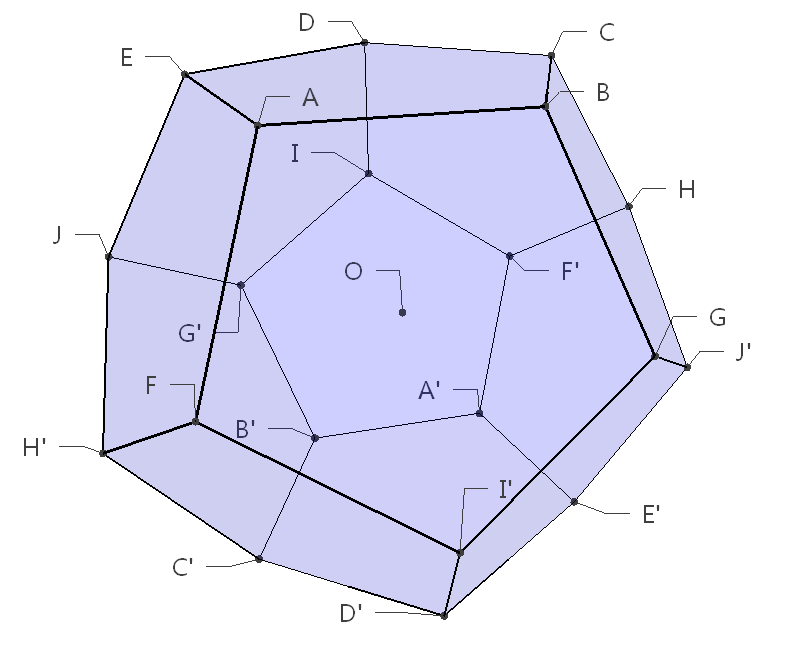}
\caption{The dodecahedron fixed by $G$.}
\end{figure}

There are $10$ axes joining opposite vertices of the dodecahedron. Rotations about each axis of angles $2\pi/3$ and $4\pi/3$ carry the dodecahedron into itself. There are, thus, $2\times10=20$ such rotations in $G$.

There are $6$ axes joining central points of opposite faces of the dodecahedron. Rotations about each axis of angles $2\pi/5$, $4\pi/5$, $6\pi/5$ and $8\pi/5$ carry the dodecahedron into itself. There are, thus, $4\times6=24$ such rotations in $G$.

There are $15$ axes joining middle points of opposite edges of the dodecahedron. Rotations about each axis of angle $\pi$ carry the dodecahedron into itself. There are, thus, $1\times15=15$ such rotations in $G$.

We have already got $20+24+15=59$ rotations. To these add the identity transformation, and we get the whole group.

There are four kinds of orbits of $G$ on $S^2$. Let $V_I,\:F_I,\:E_I$ denote the vertices, the projections of the central points of the faces on $S^2$ and the projections of the middle points of the edges on $S^2$ respectively, with the origin being the center of the projection. Then $V_I,\:F_I,\:E_I$ are three different orbits of $G$ and
\begin{displaymath}
|V_I|=20,\:|F_I|=12,\:|E_I|=30.
\end{displaymath}
For any $X\in S^2 \backslash (V_I\cup F_I\cup E_I)$, define $B_I(X)$ as the orbit of $X$. As $X$ is not fixed by any non-trivial element in $G$, it is obvious that
\begin{displaymath}
\arrowvert B_I(X) \arrowvert =\arrowvert G \arrowvert=60.
\end{displaymath}

For any finite subset $\alpha$ of $S^2$ such that $\mathcal{A}_{\alpha} = G$, $\alpha$ is a finite union of the orbits of $G$. On the other hand, if $\alpha$ is a finite union of the orbits of $G$, we have $G\subseteq\mathcal{A}_{\alpha}$. Since none of $S_4,\:A_4,\:D_n,\:Z_n,\:n\ge2$ has a subgroup isomorphic to $A_5$, we conclude that $G=\mathcal{A}_{\alpha}$. Thus we conclude that
\begin{thm}\label{tA_5}
For any finite subset $\alpha$ of $S^2$, $\mathcal{A}_{\alpha}=G\simeq A_5$ if and only if $\alpha$ is a union of certain elements in $\{V_I,\:F_I,\:E_I\}\cup\{B_I(X)|X\in S^2 \backslash (V_I\cup F_I\cup E_I)\}$.
\end{thm}

%----------------------------------------------------------------
\subsection{$G$ is the Octahedral Group $O$}
\label{S_4}

Now suppose $G$ is the octahedral group $O$ that fixes a cube whose center is the origin.

There are $4$ axes joining opposite vertices of the cube. Rotations about each axis of angles $2\pi/3$ and $4\pi/3$ carry the cube into itself. There are, thus, $2\times4=8$ such rotations in $G$.

There are $3$ axes joining central points of opposite faces of the cube. Rotations about each axis of angles $\pi/2$, $\pi$ and $3\pi/2$ carry the cube into itself. There are, thus, $3\times3=9$ such rotations in $G$.

There are $6$ axes joining middle points of opposite edges of the cube. Rotations about each axis of angle $\pi$ carry the cube into itself. There are, thus, $1\times6=6$ such rotations in $G$.

We have already got $8+9+6=23$ rotations. To these add the identity transformation, and we get the whole group.

There are four kinds of orbits on $S^2$. Let $V_O,\:F_O,\:E_O$ denote the vertices, the projections of the central points of the faces on $S^2$ and the projections of the middle points of the edges on $S^2$ respectively, with the origin being the center of the projection. Then $V_O,\:F_O,\:E_O$ are three different orbits of $G$ and
\begin{displaymath}
|V_O|=8,\:|F_O|=6,\:|E_O|=12.
\end{displaymath}
For any $X\in S^2 \backslash (V_O\cup F_O\cup E_O)$, define $B_O(X)$ as the orbit of $X$. As $X$ is not fixed by any non-trivial element in $G$, it is obvious that
\begin{displaymath}
|B_O(X)|=|G|=24.
\end{displaymath}

For any finite subset $\alpha$ of $S^2$ such that $\mathcal{A}_{\alpha} = G$, $\alpha$ is a finite union of the orbits of $G$. On the other hand, if $\alpha$ is a finite union of the orbits of $G$, we have $G\subseteq\mathcal{A}_{\alpha}$. Since none of $A_5,\:A_4,\:D_n,\:Z_n,\:n\ge2$ has a subgroup isomorphic to $S_4$, we conclude that $G=\mathcal{A}_{\alpha}$. Thus we conclude that
\begin{thm}\label{tS_4}
For any finite subset $\alpha$ of $S^2$, $\mathcal{A}_{\alpha}=G\simeq S_4$ if and only if $\alpha$ is a union of certain elements in $\{V_O,\:F_O,\:E_O\}\cup\{B_O(X)|X\in S^2 \backslash (V_O\cup F_O\cup E_O)\}$.
\end{thm}

%----------------------------------------------------------------
\subsection{$G$ is the Tetrahedral Group $T$}
\label{A_4}

%----------------------------------------------------------------
\subsubsection{The Orbits of $T$}

Now suppose $G$ is the tetrahedral group $T$ that fixes a regular tetrahedron whose center is the origin.

There are $4$ axes, each joining a central point of a face to the opposite vertex. Rotations about each axis of angles $2\pi/3$ and $4\pi/3$ carry the tetrahedral into itself. There are, thus, $2\times4=8$ such rotations in $G$.

There are $3$ axes joining middle points of opposite edges of the tetrahedral. Rotations about each axis of angle $\pi$ carry the tetrahedral into itself. There are, thus, $1\times3=3$ such rotations in $G$.

We have already got $8+3=11$ rotations. To these add the identity transformation, and we get the whole group.

There are three kinds of orbits on $S^2$, too. Let $V_T,\:F_T,\:E_T$ denote the vertices, the projections of the central points of the faces on $S^2$ and the projections of the middle points of the edges on $S^2$ respectively, with the origin being the center of the projection. Then $V_T,\:F_T,\:E_T$ are three different orbits of $G$ and
\begin{displaymath}
|V_T|=|F_T|=4,\:|E_T|=6.
\end{displaymath}
For any $X\in S^2 \backslash (V_T\cup F_T\cup E_T)$, define $B_T(X)$ as the orbit of $X$. As $D$ is not fixed by any non-trivial element in $G$, it is obvious that
\begin{displaymath}
|B_T(X)|=|G|=12.
\end{displaymath}

For any finite subset $\alpha$ of $S^2$ such that $\mathcal{A}_{\alpha} = G$, $\alpha$ is a finite union of the orbits of $G$.

On the other hand, if $\alpha$ is a finite union of the orbits of $G$, we have $G\subseteq\mathcal{A}_{\alpha}$. Since neither of  $D_n, Z_n, n\ge2$ has a subgroup isomorphic to $S_4$, we conclude that $\mathcal{A}_{\alpha}$ is isomorphic to $A_5, S_4$ or $A_4$. In the following sections we shall investigate the two cases when $\mathcal{A}_{\alpha}$ is isomorphic to $A_5$ and $S_4$.

%----------------------------------------------------------------
\subsubsection[$\mathcal{A}_{\alpha}$ is Isomorphic to $A_5$]{Case 1: $\mathcal{A}_{\alpha}$ is Isomorphic to $A_5$}

If $\mathcal{A}_{\alpha}$ is isomorphic to $A_5$, we know that $\mathcal{A}_{\alpha}$ fixes some regular dodecahedron whose center is the origin. For every vertex of the tetrahedron, there exists some element $f\in G\subseteq\mathcal{A}_{\alpha}$ of order three that fixes it. But for any element $g\in\mathcal{A}_{\alpha}$, if $g$ is of order three, both of the fixed points of $g$ are vretices of the dodecahedron. Thus we conclude that every vertex of the tetrahedron is also a vertex of the regular dodecahedron. Their relative positions are shown in Figure \ref{A_5, A_4}.

\begin{figure}[!h]
\centering
\includegraphics[width=5in]{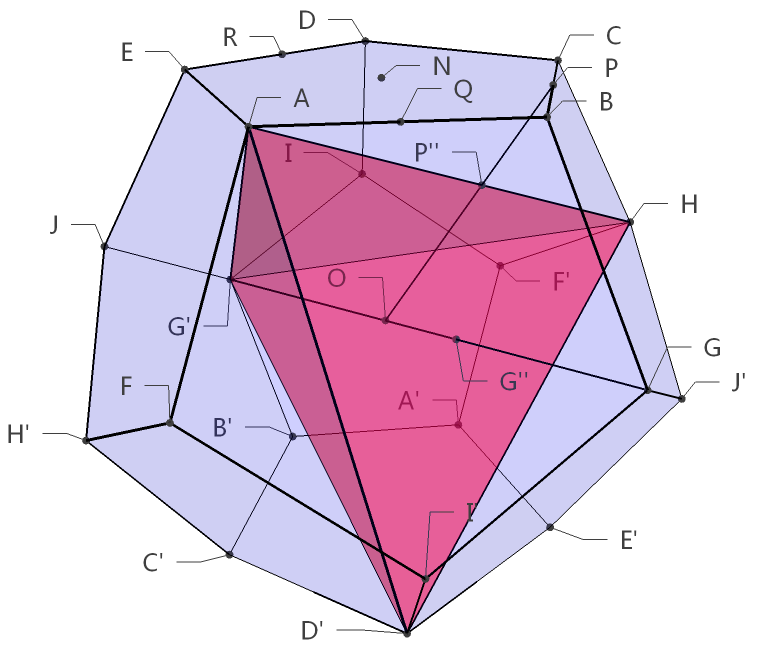}
\caption{The dodecahedron fixed by $\mathcal{A}_{\alpha}$ and the tetrahedron fixed by $G$.}
\label{A_5, A_4}
\end{figure}

It is obvious that
\begin{displaymath}
V_T\subseteq V_I.
\end{displaymath}

Connect the two points $G$ and $G'$. As they are opposite points of the dodecahedron, the segment $GG'$ goes through the origin $O$. Assume that $GG'$ intersects the triangle $AD'H$ at $G''$, then $G''$ must be the center of $AD'H$, and thus $G\in F_T$. As the four points in $F_T$ are congruent, we have
\begin{displaymath}
F_T\subseteq V_I.
\end{displaymath}

Let $P$ be the midpoint of $BC$. Connect the origin $O$ and $P$. Let $f$ be the rotation of $\pi$ with $OP$ as its axis. By observation we have $f\in\mathcal {A}_\alpha$ and
\begin{displaymath}
f(A)=H,\:f(H)=A.
\end{displaymath}
So $OP$ intersects the segment $AH$ at its midpoint $P''$. As the six points in $E_T$ are congruent, we have
\begin{displaymath}
E_T\subseteq E_I.
\end{displaymath}

From the above discussion we see that
\begin{displaymath}
B\notin V_T\cup F_T\cup E_T.
\end{displaymath}
So we have
\begin{displaymath}
|B_T(B)|=12.
\end{displaymath}
Since
\begin{displaymath}
V_T\subseteq V_I,\:F_T\subseteq V_I,\:B_T(B)\subseteq V_I,
\end{displaymath}
and
\begin{displaymath}
|V_T|=|F_T|=4,\:|B_T(B)|=12,\:|V_I|=20,
\end{displaymath}
we have
\begin{displaymath}
V_I=V_T\cup F_T\cup B_T(B).
\end{displaymath}

Let $N$ denote the center of the pentagon $ABCDE$. It is obvious that
\begin{displaymath}
N\notin V_T\cup F_T\cup E_T.
\end{displaymath}
So we have
\begin{displaymath}
|B_T(N)|=12.
\end{displaymath}
Since
\begin{displaymath}
B_T(N)\subseteq F_I,
\end{displaymath}
and
\begin{displaymath}
|B_T(N)|=|F_I|=12,
\end{displaymath}
we have
\begin{displaymath}
F_I=B_T(N).
\end{displaymath}

Let $Q$, $R$ denote the midpoints of $AB$, $DE$. It is obvious that
\begin{displaymath}
Q, R\notin V_T\cup F_T\cup E_T.
\end{displaymath}
Note that for any $f\in G$, $f(Q)\ne R$.
So we have
\begin{displaymath}
B_T(Q)\ne B_T(R).
\end{displaymath}
Since
\begin{displaymath}
E_T\subseteq E_I,\:B_T(Q)\subseteq E_I,\:B_T(R)\subseteq E_I,
\end{displaymath}
and
\begin{displaymath}
|E_T|=6,\:|B_T(Q)|=|B_T(R)|=12,\:|E_I|=30,
\end{displaymath}
we have
\begin{displaymath}
E_I=E_T\cup B_T(Q)\cup B_T(R).
\end{displaymath}

For any $X\in S^2 \backslash (V_I\cup F_I\cup E_I)$, $f(X)\notin V_I\cup F_I\cup E_I$, $\forall f\in \mathcal A_\alpha$. Now let $g$ denote the rotation which fixes the pentagon $ABCDE$ and transforms $A$ to $B$. We have
\begin{displaymath}
X,\:g(X),\:g^2(X),\:g^3(X),\:g^4(X)\in S^2 \backslash (V_T\cup F_T\cup E_T).
\end{displaymath}
Next we shall prove that
\begin{displaymath}
B_T(X),\:B_T(g(X)),\:B_T(g^2(X)),\:B_T(g^3(X)),\:B_T(g^4(X))
\end{displaymath}
are five different orbits of $G$.

Suppose that
\begin{displaymath}
B_T(g^i(X))=B_T(g^j(X))
\end{displaymath}
for some $i,j=0,1,2,3,4,\:i\ne j$.
Then there exists some $h\in G\simeq A_4$ such that
\begin{displaymath}
g^i(X)=h(g^j(X))=h(g^{j-i}(g^i(X))).
\end{displaymath}
Since $g^i(X)\notin V_I\cup F_I\cup E_I$, it can not be fixed by any non-trivial element in $\mathcal A_\alpha$.
Thus we have
\begin{displaymath}
h\circ g^{j-i}=\mathrm{I},
\end{displaymath}
which contradicts the assumption that
\begin{displaymath}
h\in G\simeq A_4.
\end{displaymath}
Now we see that
\begin{displaymath}
B_T(X),\:B_T(g(X)),\:B_T(g^2(X)),\:B_T(g^3(X)),\:B_T(g^4(X))
\end{displaymath}
are five different orbits of $G$.

As
\begin{displaymath}
B_T(X),\:B_T(g(X)),\:B_T(g^2(X)),\:B_T(g^3(X)),\:B_T(g^4(X))\subseteq B_I(X)
\end{displaymath}
and
\begin{displaymath}
|B_T(X)|=|B_T(g(X))|=|B_T(g^2(X))|=|B_T(g^3(X))|=|B_T(g^4(X))|=12,
\end{displaymath}
\begin{displaymath}
|B_I(X)|=60,
\end{displaymath}
we have
\begin{displaymath}
B_I(X)=B_T(X)\cup B_T(g(X))\cup B_T(g^2(X))\cup B_T(g^3(X))\cup B_T(g^4(X)).
\end{displaymath}

In summary,
\begin{displaymath}
V_I=V_T\cup F_T\cup B_T(B),
\end{displaymath}
\begin{displaymath}
F_I=B_T(N),
\end{displaymath}
\begin{displaymath}
E_I=E_T\cup B_T(Q)\cup B_T(R),
\end{displaymath}
\begin{displaymath}
B_I(X)=B_T(X)\cup B_T(g(X))\cup B_T(g^2(X))\cup B_T(g^3(X))\cup B_T(g^4(X)),
\end{displaymath}
$X\in S^2 \backslash (V_I\cup F_I\cup E_I)$.

As $\mathcal{A}_{\alpha}$ is isomorphic to $A_5$, $\alpha$ is a finite union of these above orbits.

%----------------------------------------------------------------
\subsubsection[$\mathcal{A}_{\alpha}$ is Isomorphic to $S_4$]{Case 2: $\mathcal{A}_{\alpha}$ Is isomorphic to $S_4$}

If $\mathcal{A}_{\alpha}$ is isomorphic to $S_4$, we know that $\mathcal{A}_{\alpha}$ fixes some cube whose center is the origin. For every vertex of the tetrahedron, there exists some element $f\in G\subseteq \mathcal{A}_{\alpha}$ of order three that fixes it. But for any element $g\in\mathcal{A}_{\alpha}$, if $g$ is of order three, both of the fixed points are vretices of the cube. Thus we conclude that every vertex of the tetrahedron is also a vertex of the cube. Their relative positions are shown in Figure \ref{S_4, A_4}.

\begin{figure}[!h]
\centering
\includegraphics[width=3in]{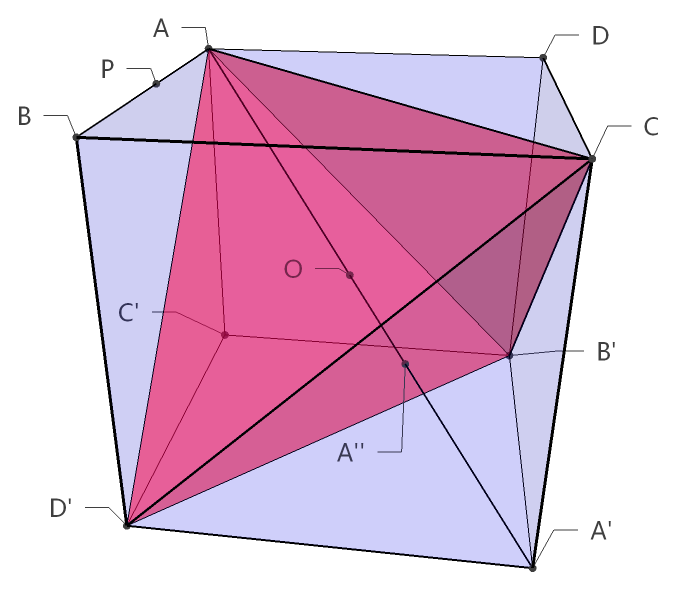}
\caption{The cube fixed by $\mathcal{A}_{\alpha}$ and the tetrahedron fixed by $G$.}
\label{S_4, A_4}
\end{figure}

It is obvious that
\begin{displaymath}
V_T\subseteq V_O.
\end{displaymath}

Connect the two points $A$ and $A'$. As they are opposite points of the cube, the segment $AA'$ goes through the origin $O$. Assume that $AA'$ intersects the triangle $B'CD'$ at $A''$, then $A''$ must be the center of $B'CD'$, and $A'\in F_T$. As the four points in $F_T$ are congruent, we have
\begin{displaymath}
F_T\subseteq V_I.
\end{displaymath}

From the above discussion we see that
\begin{displaymath}
V_T\subseteq V_O,\:F_T\subseteq V_O,
\end{displaymath}
and
\begin{displaymath}
|V_T|=|F_T|=4,\:|V_O|=8,
\end{displaymath}
so we have
\begin{displaymath}
V_O=V_T\cup F_T.
\end{displaymath}

It is obvious that
\begin{displaymath}
F_O=E_T.
\end{displaymath}

Let $P$ denote the midpoint of $AB$. It is obvious that
\begin{displaymath}
P\notin V_T\cup F_T\cup E_T.
\end{displaymath}
Since
\begin{displaymath}
B_T(P)\subseteq E_O,
\end{displaymath}
and
\begin{displaymath}
|B_T(P)|=|E_O|=12,
\end{displaymath}
we have
\begin{displaymath}
E_O=B_T(P).
\end{displaymath}

For any $X\in S^2 \backslash (V_O\cup F_O\cup E_O)$, $f(X)\notin V_O\cup F_O\cup E_O$, $\forall f\in \mathcal A_\alpha$. Now let $g$ denote the rotation which fixes the square $ABCD$ and transforms $A$ to $B$. We have
\begin{displaymath}
X,\:g(X)\in S^2 \backslash (V_T\cup F_T\cup E_T).
\end{displaymath}
Next we shall prove that
\begin{displaymath}
B_T(X),\:B_T(g(X))
\end{displaymath}
are two different orbits of $A_4$.

Suppose that
\begin{displaymath}
B_T(X)=B_T(g(X)).
\end{displaymath}
Then there exists some $h\in G\simeq A_4$ such that
\begin{displaymath}
X=h(g(X)).
\end{displaymath}
Since $X\notin V_O\cup F_O\cup E_O$, it can not be fixed by any non-trivial element in $\mathcal A_\alpha$.
Thus we have
\begin{displaymath}
h\circ g=\mathrm{I},
\end{displaymath}
which contradicts the assumption that
\begin{displaymath}
h\in G\simeq A_4.
\end{displaymath}
Now we see that
\begin{displaymath}
B_T(X),\:B_T(g(X))
\end{displaymath}
are two different orbits of $A_4$.

As
\begin{displaymath}
B_T(X),\:B_T(g(X))\subseteq B_O(X)
\end{displaymath}
and
\begin{displaymath}
|B_T(X)|=|B_T(g(X))|=12,\:|B_O(X)|=24,
\end{displaymath}
we have
\begin{displaymath}
B_O(X)=B_T(X)\cup B_T(g(X)).
\end{displaymath}

In summary,
\begin{displaymath}
V_O=V_T\cup F_T,
\end{displaymath}
\begin{displaymath}
F_O=E_T,
\end{displaymath}
\begin{displaymath}
E_O=B_T(P),
\end{displaymath}
\begin{displaymath}
B_O(X)=B_T(X)\cup B_T(g(X)),
\end{displaymath}
$X\in S^2 \backslash (V_O\cup F_O\cup E_O)$.

As $\mathcal{A}_{\alpha}$ is $S_4$, $\alpha$ is a finite union of these above orbits.

%----------------------------------------------------------------
\subsubsection{Conclusion}

From the discussion above we conclude that
\begin{thm}\label{tA_4}
For any finite subset $\alpha$ of $S^2$, $\mathcal{A}_{\alpha}=G\simeq A_4$ if and only if all of the three claims are true:
\begin{enumerate}
\item $\alpha$ is a union of certain elements in
$$
\{V_T,\:F_T,\:E_T\}\cup\{B_T(X)|X\in S^2 \backslash (V_T\cup F_T\cup E_T)\};
$$
\item $\alpha$ is NOT a union of certain elements in
$$
\{V_I,\:F_I,\:E_I\}\cup\{B_I(X)|X\in S^2 \backslash (V_I\cup F_I\cup E_I)\};
$$
\item  $\alpha$ is NOT a union of certain elements in
$$
\{V_O,\:F_O,\:E_O\}\cup\{B_O(X)|X\in S^2 \backslash (V_O\cup F_O\cup E_O)\};
$$
\end{enumerate}
where (Figure \ref{A_5, A_4})
\begin{displaymath}
V_I=V_T\cup F_T\cup B_T(B);
\end{displaymath}
\begin{displaymath}
F_I=B_T(N);
\end{displaymath}
\begin{displaymath}
E_I=E_T\cup B_T(Q)\cup B_T(R);
\end{displaymath}
\begin{displaymath}
B_I(X)=B_T(X)\cup B_T(g(X))\cup B_T(g^2(X))\cup B_T(g^3(X))\cup B_T(g^4(X))
\end{displaymath}
and (Figure \ref{S_4, A_4})
\begin{displaymath}
V_O=V_T\cup F_T;
\end{displaymath}
\begin{displaymath}
F_O=E_T;
\end{displaymath}
\begin{displaymath}
E_O=B_T(P);
\end{displaymath}
\begin{displaymath}
B_O(X)=B_T(X)\cup B_T(g(X)).
\end{displaymath}
\end{thm}

%----------------------------------------------------------------
\subsection{$G$ is the Dihedral Group $D_n$}
\label{D_n}

In this section, we go back to $\widehat{\mathbb C}$. Assume that
\begin{displaymath}
f(z)=e^{\frac{1}{n}2\pi i}z,\:g(z)=\frac{1}{z},\:n\ge2
\end{displaymath}
and
\begin{displaymath}
G=\langle f, g\rangle\simeq D_n.
\end{displaymath}

There are three kinds of orbits. Set
\begin{displaymath}
V=\{0, \infty\},
\end{displaymath}
\begin{displaymath}
A_n=\{e^{\frac{k}{n}2\pi i}|k\in\mathbb{Z}\},
\end{displaymath}
\begin{displaymath}
B_n=\{e^{\frac{2k+1}{2n}2\pi i}|k\in\mathbb{Z}\}
\end{displaymath}
and
\begin{displaymath}
C_n(z)=\{ze^{\frac{k}{n}2\pi i}|k\in\mathbb{Z}\}\cup\{z^{-1}e^{\frac{k}{n}2\pi i}|k\in\mathbb{Z}\}
\end{displaymath}
if $z\in{\mathbb C}^*\backslash\{e^{\frac{l}{2n}2\pi i}|l\in\mathbb{Z}\}$.

If $\mathcal A_\alpha=G$, $\alpha$ is a finite union of the above orbits.

If $\alpha$ is a finite union of the above orbits, we have $G\subseteq\mathcal{A}_{\alpha}$. So we conclude that $\mathcal{A}_{\alpha}$ is isomorphic to $A_5,\:S_4,\:A_4$ or $D_{pn}$ for $p\in\mathbb Z^+$. The $A_5,\:S_4,\:A_4$ case have already been explored. Now we shall discuss the case that $\mathcal{A}_{\alpha}$ is isomorphic to $D_{pn}$, $p\in\mathbb Z,\:p\ge2$.

Set that
\begin{displaymath}
\mathcal{A}_{\alpha}=\langle \rho,\pi|(\rho)^{pn}=(\pi)^2=(\rho\pi)^2=e\rangle\simeq D_{pn}.
\end{displaymath}

\emph{Case 1: $f=\rho^q$ for some $q\in\mathbb Z$.}

As $f=\rho^q$, $\rho$ fixes $0$ and $\infty$ as $f$. As $\rho$ is of order $pn$, we see that
\begin{displaymath}
\rho(z)=e^{\frac{m}{pn}2\pi i}z
\end{displaymath}
for some $m$ prime to $pn$. Thus
\begin{displaymath}
\mathcal{A}_{\alpha}=\langle z\mapsto e^{\frac{1}{pn}2\pi i}z,\:z\mapsto\frac{1}{z} \rangle.
\end{displaymath}

If $p$ is odd, we have
\begin{displaymath}
A_{pn}=A_n\cup C_n({e^{\frac{1}{pn}2\pi i}})\cup C_n({e^{\frac{2}{pn}2\pi i}})\cup\cdots\cup C_n({e^{\frac{p-1}{2pn}2\pi i}}),
\end{displaymath}
\begin{displaymath}
B_{pn}=B_n\cup C_n({e^{\frac{1}{2pn}2\pi i}})\cup C_n({e^{\frac{3}{2pn}2\pi i}})\cup\cdots\cup C_n({e^{\frac{p-2}{2pn}2\pi i}}),
\end{displaymath}
\begin{displaymath}
C_{pn}(z)=C_n(z)\cup C_n({e^{\frac{1}{pn}2\pi i}}z)\cup C_n({e^{\frac{2}{pn}2\pi i}}z)\cup\cdots\cup C_n({e^{\frac{p-1}{pn}2\pi i}}z).
\end{displaymath}

If $p$ is even, we have
\begin{displaymath}
A_{pn}=A_n\cup B_n\cup C_n({e^{\frac{1}{pn}2\pi i}})\cup C_n({e^{\frac{2}{pn}2\pi i}})\cup\cdots\cup C_n({e^{\frac{p-2}{2pn}2\pi i}}),
\end{displaymath}
\begin{displaymath}
B_{pn}=C_n({e^{\frac{1}{2pn}2\pi i}})\cup C_n({e^{\frac{3}{2pn}2\pi i}})\cup\cdots\cup C_n({e^{\frac{p-1}{2pn}2\pi i}}),
\end{displaymath}
\begin{displaymath}
C_{pn}(z)=C_n(z)\cup C_n({e^{\frac{1}{pn}2\pi i}}z)\cup C_n({e^{\frac{2}{pn}2\pi i}}z)\cup\cdots\cup C_n({e^{\frac{p-1}{pn}2\pi i}}z).
\end{displaymath}

\emph{Case 2: $f=\pi$.}

Thus we see that $n=2$ and $g=\rho^p$. Define the linear fractional transformation $\phi$
\begin{displaymath}
\phi(z)=\frac{1+z}{1-z}.
\end{displaymath}
Now we have
\begin{displaymath}
\phi^{-1}\circ\rho^p\circ\phi(z)=-z,\:\phi^{-1}\circ\pi\circ\phi(z)=\frac{1}{z},
\end{displaymath}
and thus
\begin{displaymath}
\phi^{-1}\circ\rho\circ\phi(z)=e^{\frac{m}{2p}2\pi i}z,
\end{displaymath}
for some $m$ prime to $2p$.
So we have
\begin{displaymath}
\phi^{-1}\mathcal A_\alpha\phi=\langle z\mapsto e^{\frac{1}{2p}2\pi i}z,\:z\mapsto \frac{1}{z}\rangle\simeq D_{2p}.
\end{displaymath}

So the three kinds of orbits of $\mathcal A_\alpha$ are $\phi^{-1}(V)$, $\phi^{-1}(A_{2p})$,  $\phi^{-1}(B_{2p})$ and  $\phi^{-1}(C_{2p}(z))$ for $z\in\mathbb C^*\backslash\{e^{\frac{l}{4p}2\pi i}|l\in\mathbb{Z}\}$.

If $n$ is odd, we have
\begin{displaymath}
\phi^{-1}(V)=A_2,
\end{displaymath}
\begin{displaymath}
\phi^{-1}(A_{2p})=V\cup C_2(i\tan(\frac{1}{4p}2\pi))\cup C_2(i\tan(\frac{2}{4p}2\pi))\cup\cdots\cup C_2(i\tan(\frac{\frac{p-1}{2}}{4p}2\pi)),
\end{displaymath}
\begin{displaymath}
\phi^{-1}(B_{2p})=B_2\cup C_2(i\tan(\frac{1}{8p}2\pi))\cup C_2(i\tan(\frac{3}{8p}2\pi))\cup\cdots\cup C_2(i\tan(\frac{p-2}{8p}2\pi)),
\end{displaymath}
\begin{displaymath}
\phi^{-1}(C_{2p}(z))=C_2(\phi^{-1}(z))\cup C_2(\phi^{-1}(e^{\frac{1}{2p}2\pi i}z))\cup \cdots\cup C_2(\phi^{-1}(e^{\frac{p-1}{2p}2\pi i}z))
\end{displaymath}
for $z\in\mathbb C^*\backslash\{e^{\frac{l}{4p}2\pi i}|l\in\mathbb{Z}\}$.

If $n$ is even, we have
\begin{displaymath}
\phi^{-1}(V)=A_2,
\end{displaymath}
\begin{displaymath}
\phi^{-1}(A_{2p})=V\cup B_2\cup C_2(i\tan(\frac{1}{4p}2\pi))\cup C_2(i\tan(\frac{2}{4p}2\pi))\cup\cdots\cup C_2(i\tan(\frac{\frac{p-2}{2}}{4p}2\pi)),
\end{displaymath}
\begin{displaymath}
\phi^{-1}(B_{2p})=C_2(i\tan(\frac{1}{8p}2\pi))\cup C_2(i\tan(\frac{3}{8p}2\pi))\cup\cdots\cup C_2(i\tan(\frac{p-1}{8p}2\pi)),
\end{displaymath}
\begin{displaymath}
\phi^{-1}(C_{2p}(z))=C_2(\phi^{-1}(z))\cup C_2(\phi^{-1}(e^{\frac{1}{2p}2\pi i}z))\cup \cdots\cup C_2(\phi^{-1}(e^{\frac{p-1}{2p}2\pi i}z))
\end{displaymath}
for $z\in\mathbb C^*\backslash\{e^{\frac{l}{4p}2\pi i}|l\in\mathbb{Z}\}$.

In conclusion, we have
\begin{thm}\label{tD_n}
For any finite subset $\alpha$ of $S^2$, $|\alpha|\ge4$, $\mathcal{A}_{\alpha}=G\simeq D_n,\: n\ge2$ if and only if all of the four claims are true:
\begin{enumerate}
\item $\alpha$ is a union of certain elements in
$$
\{V,\:A_n,\:B_n\}\cup\{C_n(z)|z\in{\mathbb C}^*\backslash\{e^{\frac{l}{2n}2\pi i}|l\in\mathbb{Z}\}\};
$$
\item $\alpha$ is NOT a union of certain elements in
$$
\{V,\:A_{pn},\:B_{pn}\}\cup\{C_{pn}(z)|z\in{\mathbb C}^*\backslash\{e^{\frac{l}{2n}2\pi i}|l\in\mathbb{Z}\}\},\:p\ge2;
$$
\item when $n=2$, $\alpha$ is NOT a union of certain elements in
$$
\{A_2,\:\phi^{-1}(A_{2p}),\:\phi^{-1}(B_{2p})\}\cup\{\phi^{-1}(C_{2p}(z))|z\in\mathbb C^*\backslash\{e^{\frac{l}{4p}2\pi i}|l\in\mathbb{Z}\}\},\:p\ge2;
$$
\item $\alpha$ is NOT in the icosahedral, the octahedral or the tetrahedral case,
\end{enumerate}
where
\begin{displaymath}
A_{pn}=A_n\cup C_n({e^{\frac{1}{pn}2\pi i}})\cup C_n({e^{\frac{2}{pn}2\pi i}})\cup\cdots\cup C_n({e^{\frac{p-1}{2pn}2\pi i}}),
\end{displaymath}
\begin{displaymath}
B_{pn}=B_n\cup C_n({e^{\frac{1}{2pn}2\pi i}})\cup C_n({e^{\frac{3}{2pn}2\pi i}})\cup\cdots\cup C_n({e^{\frac{p-2}{2pn}2\pi i}}),
\end{displaymath}
\begin{displaymath}
C_{pn}(z)=C_n(z)\cup C_n({e^{\frac{1}{pn}2\pi i}}z)\cup C_n({e^{\frac{2}{pn}2\pi i}}z)\cup\cdots\cup C_n({e^{\frac{p-1}{pn}2\pi i}}z),
\end{displaymath}
\begin{displaymath}
\phi^{-1}(V)=A_2,
\end{displaymath}
\begin{displaymath}
\phi^{-1}(A_{2p})=V\cup C_2(i\tan(\frac{1}{4p}2\pi))\cup C_2(i\tan(\frac{2}{4p}2\pi))\cup\cdots\cup C_2(i\tan(\frac{\frac{p-1}{2}}{4p}2\pi)),
\end{displaymath}
\begin{displaymath}
\phi^{-1}(B_{2p})=B_2\cup C_2(i\tan(\frac{1}{8p}2\pi))\cup C_2(i\tan(\frac{3}{8p}2\pi))\cup\cdots\cup C_2(i\tan(\frac{p-2}{8p}2\pi)),
\end{displaymath}
\begin{displaymath}
\phi^{-1}(C_{2p}(z))=C_2(\phi^{-1}(z))\cup C_2(\phi^{-1}(e^{\frac{1}{2p}2\pi i}z))\cup \cdots\cup C_2(\phi^{-1}(e^{\frac{p-1}{2p}2\pi i}z))
\end{displaymath}
for $z\in\mathbb C^*\backslash\{e^{\frac{l}{4p}2\pi i}|l\in\mathbb{Z}\}$ if $p$ is odd, and
\begin{displaymath}
A_{pn}=A_n\cup B_n\cup C_n({e^{\frac{1}{pn}2\pi i}})\cup C_n({e^{\frac{2}{pn}2\pi i}})\cup\cdots\cup C_n({e^{\frac{p-2}{2pn}2\pi i}}),
\end{displaymath}
\begin{displaymath}
B_{pn}=C_n({e^{\frac{1}{2pn}2\pi i}})\cup C_n({e^{\frac{3}{2pn}2\pi i}})\cup\cdots\cup C_n({e^{\frac{p-1}{2pn}2\pi i}}),
\end{displaymath}
\begin{displaymath}
C_{pn}(z)=C_n(z)\cup C_n({e^{\frac{1}{pn}2\pi i}}z)\cup C_n({e^{\frac{2}{pn}2\pi i}}z)\cup\cdots\cup C_n({e^{\frac{p-1}{pn}2\pi i}}z),
\end{displaymath}
\begin{displaymath}
\phi^{-1}(V)=A_2,
\end{displaymath}
\begin{displaymath}
\phi^{-1}(A_{2p})=V\cup B_2\cup C_2(i\tan(\frac{1}{4p}2\pi))\cup C_2(i\tan(\frac{2}{4p}2\pi))\cup\cdots\cup C_2(i\tan(\frac{\frac{p-2}{2}}{4p}2\pi)),
\end{displaymath}
\begin{displaymath}
\phi^{-1}(B_{2p})=C_2(i\tan(\frac{1}{8p}2\pi))\cup C_2(i\tan(\frac{3}{8p}2\pi))\cup\cdots\cup C_2(i\tan(\frac{p-1}{8p}2\pi)),
\end{displaymath}
\begin{displaymath}
\phi^{-1}(C_{2p}(z))=C_2(\phi^{-1}(z))\cup C_2(\phi^{-1}(e^{\frac{1}{2p}2\pi i}z))\cup \cdots\cup C_2(\phi^{-1}(e^{\frac{p-1}{2p}2\pi i}z))
\end{displaymath}
for $z\in\mathbb C^*\backslash\{e^{\frac{l}{4p}2\pi i}|l\in\mathbb{Z}\}$ if $p$ is even.
\end{thm}

We also have two corollaries
\begin{cor}\label{D_n1}
For $n\ge5$, set $\alpha=\{e^{\frac{k}{n}2\pi i}|k\in\mathbb Z\}$, we have
$$
\mathcal A_\alpha=\langle z\mapsto e^{\frac{1}{n}2\pi i}z,\:z\mapsto \frac{1}{z} \rangle\simeq D_n.
$$
\end{cor}

\begin{cor}\label{D_n2}
For $n\ge5$, set $\alpha=\{0,\:\infty\}\cup\{e^{\frac{k}{n}2\pi i}|k\in\mathbb Z\}$, we have
$$
\mathcal A_\alpha=\langle z\mapsto e^{\frac{1}{n}2\pi i}z,\:z\mapsto \frac{1}{z} \rangle\simeq D_{n}.
$$
\end{cor}

%----------------------------------------------------------------
\subsection{$G$ is the Cyclic Group $\mathbb{Z}_n$}
\label{Z_n}

Things are the same as in the section above. We still work in $\widehat{\mathbb C}$. We may assume that
\begin{displaymath}
G=\langle z \mapsto e^{\frac{2\pi i}{n}}z\rangle\simeq\mathbb{Z}_n,\:n\ge 2.
\end{displaymath}
There are three kinds of orbits. Let $N$ be the orbit $\{\infty\}$, $S$ the orbit $\{0\}$ and $C_n(z)$ the orbit $\{e^{\frac{k}{n}2\pi i}z|k\in\mathbb{Z}\}$ for $z\in\mathbb C^*$.

If $\mathcal A_\alpha=G$, $\alpha$ is a finite union of the above orbits.

If $\alpha$ is a finite union of the above orbits, we have $G\subseteq\mathcal{A}_{\alpha}$. So we conclude that $\mathcal{A}_{\alpha}$ is isomorphic to $A_5,\:S_4,\:A_4,\:D_{pn}$ or $\mathbb{Z}_{pn}$ for $p\in{\mathbb Z}^+$. The $A_5,\:S_4,\:A_4,\:D_{pn}$ case have already been explored. Now we shall discuss the case that $\mathcal{A}_{\alpha}$ is isomorphic to ${\mathbb Z}_{pn}$, $p\in\mathbb Z,\:p\ge2$.

We have
\begin{displaymath}
C_{pn}(z)=C_n(z)\cup C_n(e^{\frac{1}{pn}2\pi i}z)\cup C_n(e^{\frac{2}{pn}2\pi i}z)\cup\cdots\cup C_n(e^{\frac{p-1}{pn}2\pi i}z)
\end{displaymath}
for $z\in\mathbb C^*$.

In conclusion, we have
\begin{thm}\label{tZ_n}
For any finite subset $\alpha$ of $S^2$, $|\alpha|\ge4$, $\mathcal{A}_{\alpha}=G\simeq \mathbb Z_n, n\ge2$ if and only if all of the three claims are true:
\begin{enumerate}
\item $\alpha$ is a union of certain elements in
$$
\{N,\:S\}\cup\{C_n(z)|z\in{\mathbb C}^*\};
$$
\item $\alpha$ is NOT a union of certain elements in
$$
\{N,\:S\}\cup\{C_{pn}(z)|z\in{\mathbb C}^*\},\:p\ge2;
$$
\item $\alpha$ is NOT in the icosahedral, the octahedral the tetrahedral or the dihedral case,
\end{enumerate}
where
\begin{displaymath}
C_{pn}(z)=C_n(z)\cup C_n(e^{\frac{1}{pn}2\pi i}z)\cup C_n(e^{\frac{2}{pn}2\pi i}z)\cup\cdots\cup C_n(e^{\frac{p-1}{pn}2\pi i}z)
\end{displaymath}
for $z\in\mathbb C^*$.
\end{thm}

We also have one corollary
\begin{cor}\label{cZ_n}
For $n\ge4$, set $\alpha=\{0\}\cup\{e^{\frac{k}{n}2\pi i}|k\in\mathbb Z\}$, we have
$$
\mathcal A_\alpha=\langle z\mapsto e^{\frac{1}{n}2\pi i}z \rangle\simeq \mathbb Z_n.
$$
\end{cor}

%----------------------------------------------------------------
\section{Representations of the Stabilizers of Singularities}
\label{rep}

\subsection{Definition and Some Prior Work}
For any $\boldsymbol {\lambda}\in K_n$ such that $\overline{[\boldsymbol {\lambda}]}$ is a singularity of $\mathfrak{M_{0,\:n}}$, each $g_\sigma \in G_{\boldsymbol {\lambda}}$ introduces a tangential mapping $g_\sigma^*$ of $\text{T}_{\boldsymbol {\lambda}}K_n$ such that
$$
g^*_\sigma(\frac{\partial}{\partial\lambda^i}\bigg|_{\boldsymbol{\lambda}})(\lambda^j)=\frac{\partial}{\partial\lambda^i}\bigg|_{\boldsymbol{\lambda}}(\lambda^j\circ g_\sigma)=\frac{\partial}{\partial\lambda^i}\bigg|_{\boldsymbol{\lambda}}(g^{(j)}_\sigma)=\frac{\partial g^{(j)}_\sigma}{\partial\lambda^i}\bigg|_{\boldsymbol{\lambda}}
$$
for $\:i,\:j=1,\:2,\:\cdots,\:n-3.$

\begin{defn}
Let $\boldsymbol{J_\sigma}$ denote the Jacobian matrix of $g_\sigma$. Define a representation of $G_{\boldsymbol{\lambda}}$
$$
X_{\boldsymbol{\lambda}}:\:G_{\boldsymbol{\lambda}}\to\text{GL}_{n-3},\:g_\sigma\mapsto \boldsymbol{J_\sigma}\bigg|_{\boldsymbol{\lambda}}.
$$
\end{defn}
It is obvious that $X_{\boldsymbol{\lambda}}$ is an representation of $G_{\boldsymbol {\lambda}}$ of degree $n-3$. In the rest of this section by calculating its character $\chi_{\boldsymbol{\lambda}}$, we shall explore the representation $X_{\boldsymbol{\lambda}}$ for each $\boldsymbol {\lambda}\in K_n$ such that $\overline{[\boldsymbol {\lambda}]}$ is a singularity of $\mathfrak{M_{0,\:n}}$. To make the calculation easier, we shall introduce two lemmas first.
\begin{lem}
\label{tec}
For any $\sigma\in S_n$, let $\boldsymbol{J_\sigma}$ denote the Jacobian matrix of $g_\sigma$. Then
$$
\mathbf{tr}\boldsymbol{J_\sigma}=\sum_{\substack{\sigma(k+3)\\=1,\:2,\:3\:\mathrm{or}\:k+3}}\frac{\partial g^{(k)}_\sigma(\boldsymbol {\lambda})}{\partial \lambda^{k}}.
$$
\end{lem}
\begin{proof}
Set
$$
\Psi_\sigma:\:\:K_n\times\widehat{\mathbb C}\to\widehat{\mathbb C},\:\:(\boldsymbol {\lambda},\:z)\mapsto f^{\boldsymbol {\lambda}}_\sigma(z).
$$
By definition $f^{\boldsymbol {\lambda}}_\sigma$ is the linear fractional transformation such that
$$
f^{\boldsymbol {\lambda}}_\sigma (z^{\boldsymbol{\lambda}}_{\sigma^{-1}(1)})=0,\:f^{\boldsymbol {\lambda}}_\sigma (z^{\boldsymbol{\lambda}}_{\sigma^{-1}(2)})=1,\:f^{\boldsymbol {\lambda}}_\sigma (z^{\boldsymbol{\lambda}}_{\sigma^{-1}(3)})=\infty.
$$
Thus $\Psi_\sigma(\boldsymbol {\lambda},\:z)$ is a fraction of $z$ and $\lambda^{j}$, where $j$ satisfies the condition $\sigma(j+3)=1,\:2$ or $3$. So we have
$$
\frac{\partial \Psi_\sigma(\boldsymbol {\lambda},\:z)}{\partial \lambda^{l}}=0
$$
if $\sigma(l+3)\ne1,\:2$ or $3$. As
$$
g^{(k)}_\sigma(\boldsymbol {\lambda})
=f^{\boldsymbol {\lambda}}_\sigma (z^{\boldsymbol{\lambda}}_{\sigma^{-1}(k+3)})
=\Psi_\sigma(\boldsymbol {\lambda},\:z^{\boldsymbol{\lambda}}_{\sigma^{-1}(k+3)})
$$
we have
$$
\frac{\partial g^{(k)}_\sigma(\boldsymbol {\lambda})}{\partial \lambda^{l}}=0
$$
if $\sigma(l+3)\ne1,\:2,\:3$ or $k+3$. Thus
$$
\mathbf{tr}\boldsymbol{J_\sigma}
=\sum_{k=1}^{n-3}\frac{\partial g^{(k)}_\sigma(\boldsymbol {\lambda})}{\partial \lambda^{k}}
=\sum_{\substack{\sigma(k+3)\\=1,\:2,\:3\:\mathrm{or}\:k+3}}\frac{\partial g^{(k)}_\sigma(\boldsymbol {\lambda})}{\partial \lambda^{k}}.
$$
\end{proof}

\begin{lem}
\label{relation}
If $\boldsymbol {\lambda},\:\boldsymbol {\mu}\in K_n$, and $[\boldsymbol {\lambda}]$, $[\boldsymbol {\mu}]$ are equivalent, then for any linear fractional transformation $\varphi$ such that
$$
\varphi ([\boldsymbol {\lambda}])=[\boldsymbol {\mu}],
$$
there exists a unique element $\pi\in S_n$ satisfying
$$
\varphi=f^{\boldsymbol {\lambda}}_\pi,\:
\boldsymbol {\mu}=g_\pi(\boldsymbol{\lambda}),\:
G_{\boldsymbol{\mu}}=g_\pi G_{\boldsymbol{\lambda}}g_{\pi^{-1}}
$$
and
$$
\Phi_{\boldsymbol\mu}(g_{\pi\cdot\sigma\cdot\pi^{-1}})
=f^{\boldsymbol {\lambda}}_\pi\circ\Phi_{\boldsymbol\lambda}(g_\sigma)\circ({f^{\boldsymbol {\lambda}}_\pi})^{-1},
\:\:\chi_{\boldsymbol{\mu}}(g_{\pi\cdot\sigma\cdot\pi^{-1}})=\chi_{\boldsymbol{\lambda}}(g_{\sigma})
$$
for each $g_\sigma\in G_{\boldsymbol{\lambda}}$, with $\Phi_{\boldsymbol\lambda}$, $\Phi_{\boldsymbol\mu}$ defined in \ref{def} and $\chi_{\boldsymbol{\lambda}}$, $\chi_{\boldsymbol{\mu}}$ denoting the characters for $X_{\boldsymbol{\lambda}}$ and $X_{\boldsymbol{\mu}}$ respectively.
\end{lem}
\begin{proof}
For any linear fractional transformation $\varphi$ such that
$$
\varphi ([\boldsymbol {\lambda}])=[\boldsymbol {\mu}],
$$
there exists a unique element $\pi\in S_n$ such that
$$
\varphi(z^{\boldsymbol{\lambda}}_{\pi^{-1}(k)})=z^{\boldsymbol{\mu}}_k,\:k=1,\:2,\:\cdots,\:n.
$$
Thus we have
$$
\varphi=f^{\boldsymbol {\lambda}}_\pi,\:\boldsymbol {\mu}=g_\pi(\boldsymbol{\lambda}).
$$

For any $g_\sigma\in G_{\boldsymbol{\lambda}}$,
$$
f^{\boldsymbol {\lambda}}_\pi\circ f^{\boldsymbol {\lambda}}_\sigma\circ (f^{\boldsymbol {\lambda}}_\pi)^{-1}(z^{\boldsymbol{\mu}}_{\pi\cdot\sigma^{-1}\cdot\pi^{-1}(k)})
=f^{\boldsymbol {\lambda}}_\pi\circ f^{\boldsymbol {\lambda}}_\sigma(z^{\boldsymbol{\lambda}}_{\sigma^{-1}\cdot\pi^{-1}(k)})
=f^{\boldsymbol {\lambda}}_\pi(z^{\boldsymbol{\lambda}}_{\pi^{-1}(k)})
=z^{\boldsymbol{\mu}}_k
$$
for $k=1,\:2,\:\cdots,\:n$. Thus
$$
f^{\boldsymbol {\lambda}}_\pi\circ f^{\boldsymbol {\lambda}}_\sigma\circ(f^{\boldsymbol {\lambda}}_\pi)^{-1}=f^{\boldsymbol {\mu}}_{\pi\cdot\sigma\cdot\pi^{-1}},\:
g_{\pi\cdot\sigma\cdot\pi^{-1}}\in G_{\boldsymbol{\mu}}.
$$
We conclude that
$$
g_\pi G_{\boldsymbol{\lambda}}g_{\pi^{-1}}=G_{\boldsymbol{\mu}}
$$
and
$$
\Phi_{\boldsymbol\mu}(g_{\pi\cdot\sigma\cdot\pi^{-1}})
=f^{\boldsymbol {\mu}}_{\pi\cdot\sigma\cdot\pi^{-1}}
=f^{\boldsymbol {\lambda}}_\pi\circ f^{\boldsymbol {\lambda}}_\sigma\circ(f^{\boldsymbol {\lambda}}_\pi)^{-1}
=f^{\boldsymbol {\lambda}}_\pi\circ \Phi_{\boldsymbol\lambda}(g_\sigma)\circ(f^{\boldsymbol {\lambda}}_\pi)^{-1}.
$$

Notice that
$$
X_{\boldsymbol{\mu}}(g_{\pi\cdot\sigma\cdot\pi^{-1}})
=\boldsymbol{J_{\pi\cdot\sigma\cdot\pi^{-1}}}\bigg|_{\boldsymbol{\mu}}
=\boldsymbol{J_\pi}\bigg|_{g_\sigma(g_{\pi^{-1}}({\boldsymbol{\mu}}))}
\boldsymbol{J_\sigma}\bigg|_{g_{\pi^{-1}}({\boldsymbol{\mu}})}
\boldsymbol{J_{\pi^{-1}}}\bigg|_{\boldsymbol{\mu}}
=\boldsymbol{J_\pi}\bigg|_{\boldsymbol{\lambda}}
\boldsymbol{J_\sigma}\bigg|_{\boldsymbol{\lambda}}
\boldsymbol{J_{\pi^{-1}}}\bigg|_{\boldsymbol{\mu}}.
$$
As $g_\pi\circ g_{\pi^{-1}}=g_\varepsilon$, we have
$$
\mathrm{I}_{n-3}
=X_{\boldsymbol{\mu}}(g_\varepsilon)
=\boldsymbol{J_{\varepsilon}}\bigg|_{\boldsymbol{\mu}}
%=\boldsymbol{J_{\pi\cdot\pi^{-1}}}\bigg|_{\boldsymbol{\mu}}
=\boldsymbol{J_\pi}\bigg|_{g_{\pi^{-1}}({\boldsymbol{\mu}})}
\boldsymbol{J_{\pi^{-1}}}\bigg|_{\boldsymbol{\mu}}
=\boldsymbol{J_\pi}\bigg|_{\boldsymbol{\lambda}}
\boldsymbol{J_{\pi^{-1}}}\bigg|_{\boldsymbol{\mu}}.
$$
Thus
$$
X_{\boldsymbol{\mu}}(g_{\pi\cdot\sigma\cdot\pi^{-1}})
=\boldsymbol{J_\pi}\bigg|_{\boldsymbol{\lambda}}
X_{\boldsymbol{\lambda}}(g_\sigma)(\boldsymbol{J_{\pi}}\bigg|_{\boldsymbol{\lambda}})^{-1}
$$
and
$$
\chi_{\boldsymbol{\mu}}(g_{\pi\cdot\sigma\cdot\pi^{-1}})=\chi_{\boldsymbol{\lambda}}(g_{\sigma}).
$$
\end{proof}

%----------------------------------------------------------------
\subsection{Calculation of the Characters}
\label{cha}
Before the calculation, notice that
\begin{rem}
For each $g_\sigma\in G_{\boldsymbol{\lambda}}$,
$$
f^{\boldsymbol{\lambda}}_\sigma(z^{\boldsymbol{\lambda}}_k)=z^{\boldsymbol{\lambda}}_{\sigma(k)}
$$
for $k=1,\:2,\:\cdots,\:n$. So if $\sigma$ is of order $m$, $f^{\boldsymbol{\lambda}}_\sigma$ is of order $m$, too. Also note that $\sigma$ has at most two fixed points if $g_\sigma$ is not the identity element.
\end{rem}

%----------------------------------------------------------------
\subsubsection{Characters of Elements with Two Fixed Points}
\begin{thm}
For any $\boldsymbol{\lambda_0}\in K_n$ and $g_\sigma\in G_{\boldsymbol{\lambda_0}}$ ($g_\sigma$ is not the identity element), if there are two fixed points of $\sigma$, then $\chi_{\boldsymbol{\lambda_0}}(g_{\sigma})=-1$.
\end{thm}
\begin{proof}
Suppose that
$$
\sigma(a)=a,\:\sigma(d)=b,\:\sigma(c)=c
$$
where $a,\:b,\:c$ and $d$ are four different integers. Let $\varphi$ be the linear fractional transformation such that
$$
\varphi(z^{\boldsymbol{\lambda_0}}_a)=0,\:\varphi(z^{\boldsymbol{\lambda_0}}_b)=1,\:\varphi(z^{\boldsymbol{\lambda_0}}_c)=\infty
$$
and $\boldsymbol{\mu_0}$ an element of $K_n$ such that
$$
[\boldsymbol{\mu_0}]=\varphi([\boldsymbol{\lambda_0}]).
$$
From Lemma \ref{relation} we know that there exists a unique element $\pi\in S_n$ such that $\varphi=f^{\boldsymbol {\lambda_0}}_\pi$ and
$$
G_{\boldsymbol{\mu_0}}=g_\pi G_{\boldsymbol{\lambda_0}}g_{\pi^{-1}},
\:\:\chi_{\boldsymbol{\mu_0}}(g_{\pi\cdot\sigma\cdot\pi^{-1}})=\chi_{\boldsymbol{\lambda_0}}(g_{\sigma}).
$$

Note that
$$
\pi(a)=1,\:\pi(b)=2,\:\pi(c)=3,\:\pi(d)>3.
$$
Thus
$$
\pi\cdot\sigma^{-1}\cdot\pi^{-1}(1)=1,\:
\pi\cdot\sigma^{-1}\cdot\pi^{-1}(2)=\pi(d),\:
\pi\cdot\sigma^{-1}\cdot\pi^{-1}(3)=3
$$
and
$$
\{k\in\mathbb{N}^*\big|\pi\cdot\sigma\cdot\pi^{-1}(k+3)=1,\:2,\:3\:\:\mathrm{or}\:\:k+3\}=\{\pi(d)-3\}.
$$

For any $\boldsymbol{\mu}\in K_n$, by the definition of $f^{\boldsymbol{\mu}}_{\pi\cdot\sigma\cdot\pi^{-1}}$ we have
$$
0=f^{\boldsymbol{\mu}}_{\pi\cdot\sigma\cdot\pi^{-1}}(z^{\boldsymbol{\mu}}_{\pi\cdot\sigma^{-1}\cdot\pi^{-1}(1)})
=f^{\boldsymbol{\mu}}_{\pi\cdot\sigma\cdot\pi^{-1}}(z^{\boldsymbol{\mu}}_1)
=f^{\boldsymbol{\mu}}_{\pi\cdot\sigma\cdot\pi^{-1}}(0),
$$
$$
1=f^{\boldsymbol{\mu}}_{\pi\cdot\sigma\cdot\pi^{-1}}(z^{\boldsymbol{\mu}}_{\pi\cdot\sigma^{-1}\cdot\pi^{-1}(2)})
=f^{\boldsymbol{\mu}}_{\pi\cdot\sigma\cdot\pi^{-1}}(z^{\boldsymbol{\mu}}_{\pi(d)})
=f^{\boldsymbol{\mu}}_{\pi\cdot\sigma\cdot\pi^{-1}}(\mu^{\pi(d)-3}),
$$
$$
\infty=f^{\boldsymbol{\mu}}_{\pi\cdot\sigma\cdot\pi^{-1}}(z^{\boldsymbol{\mu}}_{\pi\cdot\sigma^{-1}\cdot\pi^{-1}(3)})
=f^{\boldsymbol{\mu}}_{\pi\cdot\sigma\cdot\pi^{-1}}(z^{\boldsymbol{\mu}}_3)
=f^{\boldsymbol{\mu}}_{\pi\cdot\sigma\cdot\pi^{-1}}(\infty).
$$
Thus
$$
f^{\boldsymbol{\mu}}_{\pi\cdot\sigma\cdot\pi^{-1}}(z)
=\frac{z}{\mu^{\pi(d)-3}}.
$$
Particularly
$$
f^{\boldsymbol{\mu_0}}_{\pi\cdot\sigma\cdot\pi^{-1}}(z)
=\frac{z}{\mu_0^{\pi(d)-3}}.
$$

Suppose $\sigma$ is of order $m$ ($m\ge2$). As $g_{\pi\cdot\sigma\cdot\pi^{-1}}\in G_{\boldsymbol{\mu_0}}$, $f^{\boldsymbol{\mu_0}}_{\pi\cdot\sigma\cdot\pi^{-1}}$ is of order $m$, too. Thus $\mu_0^{\pi(d)-3}$ is the primitive $m$th root of unity. As
$$
\mu_0^{\pi(d)-3}
=z^{\boldsymbol{\mu_0}}_{\pi(d)}
=f^{\boldsymbol{\mu_0}}_{\pi\cdot\sigma\cdot\pi^{-1}}(z^{\boldsymbol{\mu_0}}_{\pi\cdot\sigma^{-1}(d)})
=\frac{z^{\boldsymbol{\mu_0}}_{\pi\cdot\sigma^{-1}(d)}}{\mu_0^{\pi(d)-3}}
$$
we have
$$
z^{\boldsymbol{\mu_0}}_{\pi\cdot\sigma^{-1}(d)}=(\mu_0^{\pi(d)-3})^2.
$$

\begin{align*}
\chi_{\boldsymbol{\lambda_0}}(g_\sigma)
=&\chi_{\boldsymbol{\mu_0}}(g_{\pi\cdot\sigma\cdot\pi^{-1}})
=\mathbf{tr}\boldsymbol{J_{\pi\cdot\sigma\cdot\pi^{-1}}}\bigg|_{\boldsymbol{\mu_0}}\\
=&\sum_{\substack{\pi\cdot\sigma\cdot\pi^{-1}(k+3)\\=1,\:2,\:3\:\mathrm{or}\:k+3}}
\frac{\partial g^{(k)}_{\pi\cdot\sigma\cdot\pi^{-1}}({\boldsymbol{\mu}})}{\partial \mu^{k}}
\bigg|_{\boldsymbol{\mu_0}}
=\frac{\partial g^{(\pi(d)-3)}_{\pi\cdot\sigma\cdot\pi^{-1}}({\boldsymbol{\mu}})}
{\partial \mu^{\pi(d)-3}}\bigg|_{\boldsymbol{\mu_0}}\\
=&\frac{\partial f^{\boldsymbol{\mu}}_{\pi\cdot\sigma\cdot\pi^{-1}}(z^{\boldsymbol{\mu}}_{\pi\cdot\sigma^{-1}\cdot\pi^{-1}(\pi(d))})}
{\partial \mu^{\pi(d)-3}}\bigg|_{\boldsymbol{\mu_0}}
=\frac{\partial f^{\boldsymbol{\mu}}_{\pi\cdot\sigma\cdot\pi^{-1}}(z^{\boldsymbol{\mu}}_{\pi\cdot\sigma^{-1}(d)})}
{\partial \mu^{\pi(d)-3}}\bigg|_{\boldsymbol{\mu_0}}\\
=&\bigg(\frac{\partial}{\partial \mu^{\pi(d)-3}}
\bigg(\frac{z^{\boldsymbol{\mu}}_{\pi\cdot\sigma^{-1}(d)}}{\mu^{\pi(d)-3}}\bigg)\bigg)\bigg|_{\boldsymbol{\mu_0}}.
\end{align*}
As $\pi\cdot\sigma^{-1}(d)-3\ne\pi(d)-3$, we know that
$$
\frac{\partial z^{\boldsymbol{\mu}}_{\pi\cdot\sigma^{-1}(d)}}{\partial \mu^{\pi(d)-3}}=0.
$$
Thus
$$
\chi_{\boldsymbol{\lambda_0}}(g_\sigma)
=\bigg(\frac{\partial}{\partial \mu^{\pi(d)-3}}
\big(\frac{z^{\boldsymbol{\mu}}_{\pi\cdot\sigma^{-1}(d)}}{\mu^{\pi(d)-3}}\big)\bigg)\bigg|_{\boldsymbol{\mu_0}}
=\frac{-z^{\boldsymbol{\mu}}_{\pi\cdot\sigma^{-1}(d)}}{(\mu^{\pi(d)-3})^2}\bigg|_{\boldsymbol{\mu_0}}
=\frac{-z^{\boldsymbol{\mu_0}}_{\pi\cdot\sigma^{-1}(d)}}{(\mu_0^{\pi(d)-3})^2}
=-1.
$$
\end{proof}

%----------------------------------------------------------------
\subsubsection{Characters of Elements with One Fixed Point}
\begin{lem}
For any $\boldsymbol{\lambda_0}\in K_n$ and $g_\sigma\in G_{\boldsymbol{\lambda_0}}$, if $\sigma$ has only one fixed point and is of order two, then $\chi_{\boldsymbol{\lambda_0}}(g_{\sigma})=0$.
\end{lem}
\begin{proof}
Suppose that
$$
\sigma(a)=a,\:\sigma(b)=c,\:\sigma(c)=b
$$
where $a,\:b$ and $c$ are three different integers. Let $\varphi$ be the linear fractional transformation such that
$$
\varphi(z^{\boldsymbol{\lambda_0}}_a)=0,\:\varphi(z^{\boldsymbol{\lambda_0}}_b)=1,\:\varphi(z^{\boldsymbol{\lambda_0}}_c)=\infty
$$
and $\boldsymbol{\mu_0}$ an element of $K_n$ such that
$$
[\boldsymbol{\mu_0}]=\varphi([\boldsymbol{\lambda_0}]).
$$
From Lemma \ref{relation} we know that there exists a unique element $\pi\in S_n$ such that $\varphi=f^{\boldsymbol {\lambda_0}}_\pi$ and
$$
\chi_{\boldsymbol{\mu_0}}(g_{\pi\cdot\sigma\cdot\pi^{-1}})=\chi_{\boldsymbol{\lambda_0}}(g_{\sigma}).
$$

Note that
$$
\pi(a)=1,\:\pi(b)=2,\:\pi(c)=3
$$
and
$$
\pi\cdot\sigma^{-1}\cdot\pi^{-1}(1)=1,\:
\pi\cdot\sigma^{-1}\cdot\pi^{-1}(2)=3,\:
\pi\cdot\sigma^{-1}\cdot\pi^{-1}(3)=2.
$$
Thus
$$
\{k\in\mathbb{N}^*\big|\pi\cdot\sigma\cdot\pi^{-1}(k+3)=1,\:2,\:3\:\:\mathrm{or}\:\:k+3\}=\emptyset
$$
and
$$
\chi_{\boldsymbol{\lambda_0}}(g_\sigma)
=\chi_{\boldsymbol{\mu_0}}(g_{\pi\cdot\sigma\cdot\pi^{-1}})
=\mathbf{tr}\boldsymbol{J_{\pi\cdot\sigma\cdot\pi^{-1}}}\bigg|_{\boldsymbol{\mu_0}}
=\sum_{\substack{\pi\cdot\sigma\cdot\pi^{-1}(k+3)\\=1,\:2,\:3\:\mathrm{or}\:k+3}}
\frac{\partial g^{(k)}_{\pi\cdot\sigma\cdot\pi^{-1}}({\boldsymbol{\mu}})}{\partial \mu^{k}}
\bigg|_{\boldsymbol{\mu_0}}
=0.
$$
\end{proof}

\begin{thm}
For any $\boldsymbol{\lambda_0}\in K_n$ and $g_\sigma\in G_{\boldsymbol{\lambda_0}}$, if $\sigma$ has only one fixed point $a_1$, and there exists some linear fractional transformation $\psi$ such that
$$
\psi\circ f^{\boldsymbol{\lambda_0}}_\sigma\circ\psi^{-1}(z)
=e^{\frac{2q\pi}{p}i}z
$$
where $p,\:q$ are co-prime positive integers, and
$$
\psi(z^{\boldsymbol{\lambda_0}}_{a_1})=0,
$$
then
$$
\chi_{\boldsymbol{\lambda_0}}(g_{\sigma})=-1-e^{\frac{-2q\pi}{p}i}.
$$
\end{thm}
\begin{proof}
Notice that $\sigma$ is of order $p$ as $f^{\boldsymbol{\lambda_0}}_\sigma$ is of order $p$. Thus from the previous lemma it is obvious that the result holds when $p=2$. Next assume that $p\ge3$. Suppose that
$$
\sigma(a_2)=a_3,\:\sigma(a_3)=a_4,\:\cdots,\:\sigma(a_{p+1})=a_2
$$
where $a_1,\:a_2,\:\cdots,\:a_{p+1}$ are $p+1$ different integers. Let $\varphi$ be the linear fractional transformation such that
$$
\varphi\circ\psi(z^{\boldsymbol{\lambda_0}}_{a_1})=0,\:\varphi\circ\psi(z^{\boldsymbol{\lambda_0}}_{a_2})=1,\:\varphi\circ\psi(z^{\boldsymbol{\lambda_0}}_{a_3})=\infty
$$
and $\boldsymbol{\mu_0}$ an element of $K_n$ such that
$$
[\boldsymbol{\mu_0}]=\varphi\circ\psi([\boldsymbol{\lambda_0}]).
$$
From Lemma \ref{relation} we know that there exists a unique element $\pi\in S_n$ such that $\varphi\circ\psi=f^{\boldsymbol {\lambda_0}}_\pi$ and
$$
\boldsymbol {\mu_0}=g_\pi(\boldsymbol{\lambda_0}),
\:\:G_{\boldsymbol{\mu_0}}=g_\pi G_{\boldsymbol{\lambda_0}}g_{\pi^{-1}},
\:\:\chi_{\boldsymbol{\mu_0}}(g_{\pi\cdot\sigma\cdot\pi^{-1}})=\chi_{\boldsymbol{\lambda_0}}(g_{\sigma}).
$$

Note that
$$
\pi(a_1)=1,\:\pi(a_2)=2,\:\pi(a_3)=3,\:\pi(a_{p+1})>3.
$$
Thus
$$
\pi\cdot\sigma^{-1}\cdot\pi^{-1}(1)=1,\:
\pi\cdot\sigma^{-1}\cdot\pi^{-1}(2)=\pi(a_{p+1}),\:
\pi\cdot\sigma^{-1}\cdot\pi^{-1}(3)=2
$$
and
$$
\{k\in\mathbb{N}^*\big|\pi\cdot\sigma\cdot\pi^{-1}(k+3)=1,\:2,\:3\:\:\mathrm{or}\:\:k+3\}=\{\pi(a_{p+1})-3\}.
$$

For any $\boldsymbol{\mu}\in K_n$, by the definition of $f^{\boldsymbol{\mu}}_{\pi\cdot\sigma\cdot\pi^{-1}}$ we have
$$
0=f^{\boldsymbol{\mu}}_{\pi\cdot\sigma\cdot\pi^{-1}}(z^{\boldsymbol{\mu}}_{\pi\cdot\sigma^{-1}\cdot\pi^{-1}(1)})
=f^{\boldsymbol{\mu}}_{\pi\cdot\sigma\cdot\pi^{-1}}(z^{\boldsymbol{\mu}}_1)
=f^{\boldsymbol{\mu}}_{\pi\cdot\sigma\cdot\pi^{-1}}(0),
$$
$$
1=f^{\boldsymbol{\mu}}_{\pi\cdot\sigma\cdot\pi^{-1}}(z^{\boldsymbol{\mu}}_{\pi\cdot\sigma^{-1}\cdot\pi^{-1}(2)})
=f^{\boldsymbol{\mu}}_{\pi\cdot\sigma\cdot\pi^{-1}}(z^{\boldsymbol{\mu}}_{\pi(a_{p+1})})
=f^{\boldsymbol{\mu}}_{\pi\cdot\sigma\cdot\pi^{-1}}(\mu^{\pi(a_{p+1})-3}),
$$
$$
\infty=f^{\boldsymbol{\mu}}_{\pi\cdot\sigma\cdot\pi^{-1}}(z^{\boldsymbol{\mu}}_{\pi\cdot\sigma^{-1}\cdot\pi^{-1}(3)})
=f^{\boldsymbol{\mu}}_{\pi\cdot\sigma\cdot\pi^{-1}}(z^{\boldsymbol{\mu}}_2)
=f^{\boldsymbol{\mu}}_{\pi\cdot\sigma\cdot\pi^{-1}}(1).
$$
Thus
$$
f^{\boldsymbol{\mu}}_{\pi\cdot\sigma\cdot\pi^{-1}}(z)
=\frac{(\mu^{\pi(a_{p+1})-3}-1)z}{\mu^{\pi(a_{p+1})-3}(z-1)}.
$$

Now we have
\begin{align*}
\chi_{\boldsymbol{\lambda_0}}(g_\sigma)
=&\chi_{\boldsymbol{\mu_0}}(g_{\pi\cdot\sigma\cdot\pi^{-1}})
=\mathbf{tr}\boldsymbol{J_{\pi\cdot\sigma\cdot\pi^{-1}}}\bigg|_{\boldsymbol{\mu_0}}\\
=&\sum_{\substack{\pi\cdot\sigma\cdot\pi^{-1}(k+3)\\=1,\:2,\:3\:\mathrm{or}\:k+3}}
\frac{\partial g^{(k)}_{\pi\cdot\sigma\cdot\pi^{-1}}({\boldsymbol{\mu}})}{\partial \mu^{k}}
\bigg|_{\boldsymbol{\mu_0}}
=\frac{\partial g^{(\pi(a_{p+1})-3)}_{\pi\cdot\sigma\cdot\pi^{-1}}({\boldsymbol{\mu}})}
{\partial \mu^{\pi(a_{p+1})-3}}\bigg|_{\boldsymbol{\mu_0}}\\
=&\frac{\partial f^{\boldsymbol{\mu}}_{\pi\cdot\sigma\cdot\pi^{-1}}(z^{\boldsymbol{\mu}}_{\pi\cdot\sigma^{-1}\cdot\pi^{-1}(\pi(a_{p+1}))})}
{\partial \mu^{\pi(a_{p+1})-3}}\bigg|_{\boldsymbol{\mu_0}}
=\frac{\partial f^{\boldsymbol{\mu}}_{\pi\cdot\sigma\cdot\pi^{-1}}(z^{\boldsymbol{\mu}}_{\pi(a_p)})}
{\partial \mu^{\pi(a_{p+1})-3}}\bigg|_{\boldsymbol{\mu_0}}.
\end{align*}

Set $w=e^{\frac{2\pi}{p}i}$. Thus $\psi\circ f^{\boldsymbol{\lambda_0}}_\sigma\circ\psi^{-1}(z)=w^qz$.
\begin{enumerate}
\item If $p=3$, we have
$$
\psi(z^{\boldsymbol{\lambda_0}}_{a_1})=0,
$$
$$
\psi(z^{\boldsymbol{\lambda_0}}_{a_3})
=\psi\circ f^{\boldsymbol{\lambda_0}}_\sigma\circ\psi^{-1}
(\psi(z^{\boldsymbol{\lambda_0}}_{a_2}))
=w^q\psi(z^{\boldsymbol{\lambda_0}}_{a_2}),
$$
$$
\psi(z^{\boldsymbol{\lambda_0}}_{a_4})
=\psi\circ f^{\boldsymbol{\lambda_0}}_\sigma\circ\psi^{-1}
(\psi(z^{\boldsymbol{\lambda_0}}_{a_3}))
=w^{2q}\psi(z^{\boldsymbol{\lambda_0}}_{a_2}).
$$
Thus
\begin{align*}
\mu_0^{\pi(a_4)-3}
=&[0,\:1,\:\mu_0^{\pi(a_4)-3},\:\infty]\\
=&[\varphi\circ\psi(z^{\boldsymbol{\lambda_0}}_{a_1}),\:\varphi\circ\psi(z^{\boldsymbol{\lambda_0}}_{a_2}),\:\varphi\circ\psi(z^{\boldsymbol{\lambda_0}}_{a_4}),\:\varphi\circ\psi(z^{\boldsymbol{\lambda_0}}_{a_3})]\\
=&[\psi(z^{\boldsymbol{\lambda_0}}_{a_1}),\:\psi(z^{\boldsymbol{\lambda_0}}_{a_2}),\:\psi(z^{\boldsymbol{\lambda_0}}_{a_4}),\:\psi(z^{\boldsymbol{\lambda_0}}_{a_3})]\\
=&[0,\:1,\:w^{2q},\:w^{q}]\\
=&-w^q.
\end{align*}
\begin{align*}
\chi_{\boldsymbol{\lambda_0}}(g_\sigma)
=&\frac{\partial f^{\boldsymbol{\mu}}_{\pi\cdot\sigma\cdot\pi^{-1}}(z^{\boldsymbol{\mu}}_{\pi(a_3)})}
{\partial \mu^{\pi(a_4)-3}}\bigg|_{\boldsymbol{\mu_0}}
=\frac{\partial f^{\boldsymbol{\mu}}_{\pi\cdot\sigma\cdot\pi^{-1}}(z^{\boldsymbol{\mu}}_3)}
{\partial \mu^{\pi(a_4)-3}}\bigg|_{\boldsymbol{\mu_0}}\\
=&\bigg(\frac{\partial}{\partial \mu^{\pi(a_4)-3}}\bigg(
\frac{\mu^{\pi(a_4)-3}-1}{\mu_0^{\pi(a_4)-3}}
\bigg)\bigg)\bigg|_{\boldsymbol{\mu_0}}
=\frac{1}{(\mu^{\pi(a_4)-3})^2}\bigg|_{\boldsymbol{\mu_0}}
=w^{-2q}.
\end{align*}
As $p=3$, we have
$$
w^{-2q}=-1-w^{-q}
$$
for $q=1$ and $2$.
\item If $p\ge4$, we have
$$
\psi(z^{\boldsymbol{\lambda_0}}_{a_1})=0,
$$
$$
\psi(z^{\boldsymbol{\lambda_0}}_{a_{p+1}})
=\psi\circ f^{\boldsymbol{\lambda_0}}_\sigma\circ\psi^{-1}
(\psi(z^{\boldsymbol{\lambda_0}}_{a_p}))
=w^q\psi(z^{\boldsymbol{\lambda_0}}_{a_p}),
$$
$$
\psi(z^{\boldsymbol{\lambda_0}}_{a_2})
=\psi\circ f^{\boldsymbol{\lambda_0}}_\sigma\circ\psi^{-1}
(\psi(z^{\boldsymbol{\lambda_0}}_{a_{p+1}}))
=w^{2q}\psi(z^{\boldsymbol{\lambda_0}}_{a_p}),
$$
$$
\psi(z^{\boldsymbol{\lambda_0}}_{a_3})
=\psi\circ f^{\boldsymbol{\lambda_0}}_\sigma\circ\psi^{-1}
(\psi(z^{\boldsymbol{\lambda_0}}_{a_2}))
=w^{3q}\psi(z^{\boldsymbol{\lambda_0}}_{a_p}).
$$
Thus
\begin{align*}
\mu_0^{\pi(a_{p+1})-3}
=&[0,\:1,\:\mu_0^{\pi(a_{p+1})-3},\:\infty]\\
=&[\varphi\circ\psi(z^{\boldsymbol{\lambda_0}}_{a_1}),\:\varphi\circ\psi(z^{\boldsymbol{\lambda_0}}_{a_2}),\:\varphi\circ\psi(z^{\boldsymbol{\lambda_0}}_{a_{p+1}}),\:\varphi\circ\psi(z^{\boldsymbol{\lambda_0}}_{a_3})]\\
=&[\psi(z^{\boldsymbol{\lambda_0}}_{a_1}),\:\psi(z^{\boldsymbol{\lambda_0}}_{a_2}),\:\psi(z^{\boldsymbol{\lambda_0}}_{a_{p+1}}),\:\psi(z^{\boldsymbol{\lambda_0}}_{a_3})]\\
=&[0,\:w^{2q},\:w^q,\:w^{3q}]\\
=&\frac{1}{1+w^q}
\end{align*}
and
\begin{align*}
\mu_0^{\pi(a_p)-3}
=&[0,\:1,\:\mu_0^{\pi(a_p)-3},\:\infty]\\
=&[\varphi\circ\psi(z^{\boldsymbol{\lambda_0}}_{a_1}),\:\varphi\circ\psi(z^{\boldsymbol{\lambda_0}}_{a_2}),\:\varphi\circ\psi(z^{\boldsymbol{\lambda_0}}_{a_p}),\:\varphi\circ\psi(z^{\boldsymbol{\lambda_0}}_{a_3})]\\
=&[\psi(z^{\boldsymbol{\lambda_0}}_{a_1}),\:\psi(z^{\boldsymbol{\lambda_0}}_{a_2}),\:\psi(z^{\boldsymbol{\lambda_0}}_{a_p}),\:\psi(z^{\boldsymbol{\lambda_0}}_{a_3})]\\
=&[0,\:w^{2q},\:1,\:w^{3q}]\\
=&\frac{1}{1+w^q+w^{2q}}.
\end{align*}

\begin{align*}
\chi_{\boldsymbol{\lambda_0}}(g_\sigma)
=&\frac{\partial f^{\boldsymbol{\mu}}_{\pi\cdot\sigma\cdot\pi^{-1}}(z^{\boldsymbol{\mu}}_{\pi(a_p)})}
{\partial \mu^{\pi(a_{p+1})-3}}\bigg|_{\boldsymbol{\mu_0}}
=\frac{\partial f^{\boldsymbol{\mu}}_{\pi\cdot\sigma\cdot\pi^{-1}}(\mu^{\pi(a_p)-3})}
{\partial \mu^{\pi(a_{p+1})-3}}\bigg|_{\boldsymbol{\mu_0}}\\
=&\bigg(\frac{\partial}{\partial \mu^{\pi(a_{p+1})-3}}\bigg(
\frac{(\mu^{\pi(a_{p+1})-3}-1)\mu^{\pi(a_p)-3}}
{\mu^{\pi(a_{p+1})-3}(\mu^{\pi(a_p)-3}-1)}
\bigg)\bigg)\bigg|_{\boldsymbol{\mu_0}}\\
=&\frac{\mu^{\pi(a_p)-3}}
{(\mu^{\pi(a_{p+1})-3})^2(\mu^{\pi(a_p)-3}-1)}
\bigg|_{\boldsymbol{\mu_0}}
=-1-w^{-q}.
\end{align*}
\end{enumerate}
\end{proof}

%----------------------------------------------------------------
\subsubsection{Characters of Elements with No Fixed Point}
\begin{lem}
For any $\boldsymbol{\lambda_0}\in K_n$ and $g_\sigma\in G_{\boldsymbol{\lambda_0}}$, if $\sigma$ has no fixed point and is of order two, then $\chi_{\boldsymbol{\lambda_0}}(g_{\sigma})=1$.
\end{lem}
\begin{proof}
Suppose that
$$
\sigma(a)=c,\:\sigma(c)=a,\:\sigma(d)=b,\:\sigma(b)=d
$$
where $a,\:b,\:c$ and $d$ are four different integers. Let $\varphi$ be the linear fractional transformation such that
$$
\varphi(z^{\boldsymbol{\lambda_0}}_a)=0,\:\varphi(z^{\boldsymbol{\lambda_0}}_b)=1,\:\varphi(z^{\boldsymbol{\lambda_0}}_c)=\infty
$$
and $\boldsymbol{\mu_0}$ an element of $K_n$ such that
$$
[\boldsymbol{\mu_0}]=\varphi([\boldsymbol{\lambda_0}]).
$$
From Lemma \ref{relation} we know that there exists a unique element $\pi\in S_n$ such that $\varphi=f^{\boldsymbol {\lambda_0}}_\pi$ and
$$
G_{\boldsymbol{\mu_0}}=g_\pi G_{\boldsymbol{\lambda_0}}g_{\pi^{-1}},
\:\:\chi_{\boldsymbol{\mu_0}}(g_{\pi\cdot\sigma\cdot\pi^{-1}})=\chi_{\boldsymbol{\lambda_0}}(g_{\sigma}).
$$

Note that
$$
\pi(a)=1,\:\pi(b)=2,\:\pi(c)=3,\:\pi(d)>3.
$$
Thus
$$
\pi\cdot\sigma^{-1}\cdot\pi^{-1}(1)=3,\:
\pi\cdot\sigma^{-1}\cdot\pi^{-1}(2)=\pi(d),\:
\pi\cdot\sigma^{-1}\cdot\pi^{-1}(3)=1
$$
and
$$
\{k\in\mathbb{N}^*\big|\pi\cdot\sigma\cdot\pi^{-1}(k+3)=1,\:2,\:3\:\:\mathrm{or}\:\:k+3\}=\{\pi(d)-3\}.
$$

For any $\boldsymbol{\mu}\in K_n$, by the definition of $f^{\boldsymbol{\mu}}_{\pi\cdot\sigma\cdot\pi^{-1}}$ we have
$$
0=f^{\boldsymbol{\mu}}_{\pi\cdot\sigma\cdot\pi^{-1}}(z^{\boldsymbol{\mu}}_{\pi\cdot\sigma^{-1}\cdot\pi^{-1}(1)})
=f^{\boldsymbol{\mu}}_{\pi\cdot\sigma\cdot\pi^{-1}}(z^{\boldsymbol{\mu}}_3)
=f^{\boldsymbol{\mu}}_{\pi\cdot\sigma\cdot\pi^{-1}}(\infty),
$$
$$
1=f^{\boldsymbol{\mu}}_{\pi\cdot\sigma\cdot\pi^{-1}}(z^{\boldsymbol{\mu}}_{\pi\cdot\sigma^{-1}\cdot\pi^{-1}(2)})
=f^{\boldsymbol{\mu}}_{\pi\cdot\sigma\cdot\pi^{-1}}(z^{\boldsymbol{\mu}}_{\pi(d)})
=f^{\boldsymbol{\mu}}_{\pi\cdot\sigma\cdot\pi^{-1}}(\mu^{\pi(d)-3}),
$$
$$
\infty=f^{\boldsymbol{\mu}}_{\pi\cdot\sigma\cdot\pi^{-1}}(z^{\boldsymbol{\mu}}_{\pi\cdot\sigma^{-1}\cdot\pi^{-1}(3)})
=f^{\boldsymbol{\mu}}_{\pi\cdot\sigma\cdot\pi^{-1}}(z^{\boldsymbol{\mu}}_1)
=f^{\boldsymbol{\mu}}_{\pi\cdot\sigma\cdot\pi^{-1}}(0).
$$
Thus
$$
f^{\boldsymbol{\mu}}_{\pi\cdot\sigma\cdot\pi^{-1}}(z)
=\frac{\mu^{\pi(d)-3}}{z}.
$$

\begin{align*}
\chi_{\boldsymbol{\lambda_0}}(g_\sigma)
=&\chi_{\boldsymbol{\mu_0}}(g_{\pi\cdot\sigma\cdot\pi^{-1}})
=\mathbf{tr}\boldsymbol{J_{\pi\cdot\sigma\cdot\pi^{-1}}}\bigg|_{\boldsymbol{\mu_0}}\\
=&\sum_{\substack{\pi\cdot\sigma\cdot\pi^{-1}(k+3)\\=1,\:2,\:3\:\mathrm{or}\:k+3}}
\frac{\partial g^{(k)}_{\pi\cdot\sigma\cdot\pi^{-1}}({\boldsymbol{\mu}})}{\partial \mu^{k}}
\bigg|_{\boldsymbol{\mu_0}}
=\frac{\partial g^{(\pi(d)-3)}_{\pi\cdot\sigma\cdot\pi^{-1}}({\boldsymbol{\mu}})}
{\partial \mu^{\pi(d)-3}}\bigg|_{\boldsymbol{\mu_0}}\\
=&\frac{\partial f^{\boldsymbol{\mu}}_{\pi\cdot\sigma\cdot\pi^{-1}}(z^{\boldsymbol{\mu}}_{\pi\cdot\sigma^{-1}\cdot\pi^{-1}(\pi(d))})}
{\partial \mu^{\pi(d)-3}}\bigg|_{\boldsymbol{\mu_0}}
=\frac{\partial f^{\boldsymbol{\mu}}_{\pi\cdot\sigma\cdot\pi^{-1}}(z^{\boldsymbol{\mu}}_2)}
{\partial \mu^{\pi(d)-3}}\bigg|_{\boldsymbol{\mu_0}}\\
=&\frac{\partial \mu^{\pi(d)-3}}{\partial \mu^{\pi(d)-3}}\bigg|_{\boldsymbol{\mu_0}}
=1.
\end{align*}
\end{proof}

\begin{lem}
For any $\boldsymbol{\lambda_0}\in K_n$ and $g_\sigma\in G_{\boldsymbol{\lambda_0}}$, if $\sigma$ has no fixed point and is of order three, then $\chi_{\boldsymbol{\lambda_0}}(g_{\sigma})=0$.
\end{lem}
\begin{proof}
Suppose that
$$
\sigma(a)=b,\:\sigma(b)=c,\:\sigma(c)=a
$$
where $a,\:b$ and $c$ are three different integers. Let $\varphi$ be the linear fractional transformation such that
$$
\varphi(z^{\boldsymbol{\lambda_0}}_a)=0,\:\varphi(z^{\boldsymbol{\lambda_0}}_b)=1,\:\varphi(z^{\boldsymbol{\lambda_0}}_c)=\infty
$$
and $\boldsymbol{\mu_0}$ an element of $K_n$ such that
$$
[\boldsymbol{\mu_0}]=\varphi([\boldsymbol{\lambda_0}]).
$$
From Lemma \ref{relation} we know that there exists a unique element $\pi\in S_n$ such that $\varphi=f^{\boldsymbol {\lambda_0}}_\pi$ and
$$
\chi_{\boldsymbol{\mu_0}}(g_{\pi\cdot\sigma\cdot\pi^{-1}})=\chi_{\boldsymbol{\lambda_0}}(g_{\sigma}).
$$

Note that
$$
\pi(a)=1,\:\pi(b)=2,\:\pi(c)=3
$$
and
$$
\pi\cdot\sigma^{-1}\cdot\pi^{-1}(1)=3,\:
\pi\cdot\sigma^{-1}\cdot\pi^{-1}(2)=1,\:
\pi\cdot\sigma^{-1}\cdot\pi^{-1}(3)=2.
$$
Thus
$$
\{k\in\mathbb{N}^*\big|\pi\cdot\sigma\cdot\pi^{-1}(k+3)=1,\:2,\:3\:\:\mathrm{or}\:\:k+3\}=\emptyset
$$
and
$$
\chi_{\boldsymbol{\lambda_0}}(g_\sigma)
=\chi_{\boldsymbol{\mu_0}}(g_{\pi\cdot\sigma\cdot\pi^{-1}})
=\mathbf{tr}\boldsymbol{J_{\pi\cdot\sigma\cdot\pi^{-1}}}\bigg|_{\boldsymbol{\mu_0}}
=\sum_{\substack{\pi\cdot\sigma\cdot\pi^{-1}(k+3)\\=1,\:2,\:3\:\mathrm{or}\:k+3}}
\frac{\partial g^{(k)}_{\pi\cdot\sigma\cdot\pi^{-1}}({\boldsymbol{\mu}})}{\partial \mu^{k}}
\bigg|_{\boldsymbol{\mu_0}}
=0.
$$
\end{proof}

\begin{thm}
For any $\boldsymbol{\lambda_0}\in K_n$ and $g_\sigma\in G_{\boldsymbol{\lambda_0}}$, if $\sigma$ has no fixed point and $f^{\boldsymbol{\lambda_0}}_\sigma$ is conjugate to a rotation of $\frac{2q\pi}{p}$ ($p,\:q$ are co-prime positive integers), then
$$
\chi_{\boldsymbol{\lambda_0}}(g_{\sigma})=-1-2\cos{\frac{2q\pi}{p}}.
$$
\end{thm}
\begin{proof}
Notice that $\sigma$ is of order $p$ as $f^{\boldsymbol{\lambda_0}}_\sigma$ is of order $p$. Thus from the previous lemmas it is obvious that the result holds when $p=2$ and $3$. Next assume that $p\ge4$. Suppose that
$$
\sigma(a_1)=a_2,\:\sigma(a_2)=a_3,\:\cdots,\:\sigma(a_p)=a_1
$$
where $a_1,\:a_2,\:\cdots,\:a_p$ are $p$ different integers. Let $\varphi$ be the linear fractional transformation such that
$$
\varphi(z^{\boldsymbol{\lambda_0}}_{a_1})=0,\:\varphi(z^{\boldsymbol{\lambda_0}}_{a_2})=1,\:\varphi(z^{\boldsymbol{\lambda_0}}_{a_3})=\infty
$$
and $\boldsymbol{\mu_0}$ an element of $K_n$ such that
$$
[\boldsymbol{\mu_0}]=\varphi([\boldsymbol{\lambda_0}].
$$
From Lemma \ref{relation} we know that there exists a unique element $\pi\in S_n$ such that $\varphi=f^{\boldsymbol {\lambda_0}}_\pi$ and
$$
\boldsymbol {\mu_0}=g_\pi(\boldsymbol{\lambda_0}),
\:\:G_{\boldsymbol{\mu_0}}=g_\pi G_{\boldsymbol{\lambda_0}}g_{\pi^{-1}},
\:\:\chi_{\boldsymbol{\mu_0}}(g_{\pi\cdot\sigma\cdot\pi^{-1}})=\chi_{\boldsymbol{\lambda_0}}(g_{\sigma}).
$$

Note that
$$
\pi(a_1)=1,\:\pi(a_2)=2,\:\pi(a_3)=3,\:\pi(a_p)>3.
$$
Thus
$$
\pi\cdot\sigma^{-1}\cdot\pi^{-1}(1)=\pi(a_p),\:
\pi\cdot\sigma^{-1}\cdot\pi^{-1}(2)=1,\:
\pi\cdot\sigma^{-1}\cdot\pi^{-1}(3)=2
$$
and
$$
\{k\in\mathbb{N}^*\big|\pi\cdot\sigma\cdot\pi^{-1}(k+3)=1,\:2,\:3\:\:\mathrm{or}\:\:k+3\}=\{\pi(a_p)-3\}.
$$

For any $\boldsymbol{\mu}\in K_n$, by the definition of $f^{\boldsymbol{\mu}}_{\pi\cdot\sigma\cdot\pi^{-1}}$ we have
$$
0=f^{\boldsymbol{\mu}}_{\pi\cdot\sigma\cdot\pi^{-1}}(z^{\boldsymbol{\mu}}_{\pi\cdot\sigma^{-1}\cdot\pi^{-1}(1)})
=f^{\boldsymbol{\mu}}_{\pi\cdot\sigma\cdot\pi^{-1}}(z^{\boldsymbol{\mu}}_{\pi(a_p)})
=f^{\boldsymbol{\mu}}_{\pi\cdot\sigma\cdot\pi^{-1}}(\mu^{\pi(a_p)-3}),
$$
$$
1=f^{\boldsymbol{\mu}}_{\pi\cdot\sigma\cdot\pi^{-1}}(z^{\boldsymbol{\mu}}_{\pi\cdot\sigma^{-1}\cdot\pi^{-1}(2)})
=f^{\boldsymbol{\mu}}_{\pi\cdot\sigma\cdot\pi^{-1}}(z^{\boldsymbol{\mu}}_1)
=f^{\boldsymbol{\mu}}_{\pi\cdot\sigma\cdot\pi^{-1}}(0),
$$
$$
\infty=f^{\boldsymbol{\mu}}_{\pi\cdot\sigma\cdot\pi^{-1}}(z^{\boldsymbol{\mu}}_{\pi\cdot\sigma^{-1}\cdot\pi^{-1}(3)})
=f^{\boldsymbol{\mu}}_{\pi\cdot\sigma\cdot\pi^{-1}}(z^{\boldsymbol{\mu}}_2)
=f^{\boldsymbol{\mu}}_{\pi\cdot\sigma\cdot\pi^{-1}}(1).
$$
Thus
$$
f^{\boldsymbol{\mu}}_{\pi\cdot\sigma\cdot\pi^{-1}}(z)
=\frac{z-\mu^{\pi(a_p)-3}}{\mu^{\pi(a_p)-3}(z-1)}.
$$

\begin{align*}
\chi_{\boldsymbol{\lambda_0}}(g_\sigma)
=&\chi_{\boldsymbol{\mu_0}}(g_{\pi\cdot\sigma\cdot\pi^{-1}})
=\mathbf{tr}\boldsymbol{J_{\pi\cdot\sigma\cdot\pi^{-1}}}\bigg|_{\boldsymbol{\mu_0}}\\
=&\sum_{\substack{\pi\cdot\sigma\cdot\pi^{-1}(k+3)\\=1,\:2,\:3\:\mathrm{or}\:k+3}}
\frac{\partial g^{(k)}_{\pi\cdot\sigma\cdot\pi^{-1}}({\boldsymbol{\mu}})}{\partial \mu^{k}}
\bigg|_{\boldsymbol{\mu_0}}
=\frac{\partial g^{(\pi(a_p)-3)}_{\pi\cdot\sigma\cdot\pi^{-1}}({\boldsymbol{\mu}})}
{\partial \mu^{\pi(a_p)-3}}\bigg|_{\boldsymbol{\mu_0}}\\
=&\frac{\partial f^{\boldsymbol{\mu}}_{\pi\cdot\sigma\cdot\pi^{-1}}(z^{\boldsymbol{\mu}}_{\pi\cdot\sigma^{-1}\cdot\pi^{-1}(\pi(a_p))})}
{\partial \mu^{\pi(a_p)-3}}\bigg|_{\boldsymbol{\mu_0}}
=\frac{\partial f^{\boldsymbol{\mu}}_{\pi\cdot\sigma\cdot\pi^{-1}}(z^{\boldsymbol{\mu}}_{\pi(a_{p-1})})}
{\partial \mu^{\pi(a_p)-3}}\bigg|_{\boldsymbol{\mu_0}}.
\end{align*}

Set
$$
w=e^{\frac{2\pi}{p}i}.
$$
As $f^{\boldsymbol{\lambda_0}}_\sigma$ is conjugate to a rotation of $\frac{2q\pi}{p}$, there exists a linear fractional transformation $\psi$ such that
$$
\psi(z^{\boldsymbol{\lambda_0}}_{a_k})=w^{kq}
$$
for $k=1,\:2,\:\cdots,\:p$.

\begin{enumerate}
\item If $p=4$, we have
\begin{align*}
\mu_0^{\pi(a_4)-3}
=&[0,\:1,\:\mu_0^{\pi(a_4)-3},\:\infty]\\
=&[\varphi(z^{\boldsymbol{\lambda_0}}_{a_1}),\:\varphi(z^{\boldsymbol{\lambda_0}}_{a_2}),\:\varphi(z^{\boldsymbol{\lambda_0}}_{a_4}),\:\varphi(z^{\boldsymbol{\lambda_0}}_{a_3})]\\
=&[\psi(z^{\boldsymbol{\lambda_0}}_{a_1}),\:\psi(z^{\boldsymbol{\lambda_0}}_{a_2}),\:\psi(z^{\boldsymbol{\lambda_0}}_{a_4}),\:\psi(z^{\boldsymbol{\lambda_0}}_{a_3})]\\
=&[1,\:w^q,\:w^{-q},\:w^{2q}]\\
=&\frac{-w^q}{1+w^q+w^{2q}}\\
=&-1.
\end{align*}

\begin{align*}
\chi_{\boldsymbol{\lambda_0}}(g_\sigma)
=&\frac{\partial f^{\boldsymbol{\mu}}_{\pi\cdot\sigma\cdot\pi^{-1}}(z^{\boldsymbol{\mu}}_{\pi(a_3)})}
{\partial \mu^{\pi(a_4)-3}}\bigg|_{\boldsymbol{\mu_0}}
=\frac{\partial f^{\boldsymbol{\mu}}_{\pi\cdot\sigma\cdot\pi^{-1}}(z^{\boldsymbol{\mu}}_3)}
{\partial \mu^{\pi(a_4)-3}}\bigg|_{\boldsymbol{\mu_0}}\\
=&\bigg(\frac{\partial}{\partial \mu^{\pi(a_4)-3}}\bigg(
\frac{1}{\mu_0^{\pi(a_4)-3}}
\bigg)\bigg)\bigg|_{\boldsymbol{\mu_0}}
=\frac{-1}{(\mu_0^{\pi(a_4)-3})^2}\bigg|_{\boldsymbol{\mu_0}}=-1.
\end{align*}
As $p=4$, we have
$$
-1=-1-2\cos{\frac{2q\pi}{p}}
$$
when $p,\:q$ are co-prime positive integers.

\item If $p\ge4$, we have
\begin{align*}
\mu_0^{\pi(a_p)-3}
=&[0,\:1,\:\mu_0^{\pi(a_p)-3},\:\infty]\\
=&[\varphi(z^{\boldsymbol{\lambda_0}}_{a_1}),\:\varphi(z^{\boldsymbol{\lambda_0}}_{a_2}),\:\varphi(z^{\boldsymbol{\lambda_0}}_{a_p}),\:\varphi(z^{\boldsymbol{\lambda_0}}_{a_3})]\\
=&[\psi(z^{\boldsymbol{\lambda_0}}_{a_1}),\:\psi(z^{\boldsymbol{\lambda_0}}_{a_2}),\:\psi(z^{\boldsymbol{\lambda_0}}_{a_p}),\:\psi(z^{\boldsymbol{\lambda_0}}_{a_3})]\\
=&[1,\:w^q,\:w^{-q},\:w^{2q}]\\
=&\frac{-w^q}{1+w^q+w^{2q}}.
\end{align*}
and
\begin{align*}
\mu_0^{\pi(a_{p-1})-3}
=&[0,\:1,\:\mu_0^{\pi(a_{p-1})-3},\:\infty]\\
=&[\varphi(z^{\boldsymbol{\lambda_0}}_{a_1}),\:\varphi(z^{\boldsymbol{\lambda_0}}_{a_2}),\:\varphi(z^{\boldsymbol{\lambda_0}}_{a_{p-1}}),\:\varphi(z^{\boldsymbol{\lambda_0}}_{a_3})]\\
=&[\psi(z^{\boldsymbol{\lambda_0}}_{a_1}),\:\psi(z^{\boldsymbol{\lambda_0}}_{a_2}),\:\psi(z^{\boldsymbol{\lambda_0}}_{a_{p-1}}),\:\psi(z^{\boldsymbol{\lambda_0}}_{a_3})]\\
=&[1,\:w^q,\:w^{-2q},\:w^{2q}]\\
=&\frac{-w^q}{1+w^{2q}}.
\end{align*}
\begin{align*}
\chi_{\boldsymbol{\lambda_0}}(g_\sigma)
=&\frac{\partial f^{\boldsymbol{\mu}}_{\pi\cdot\sigma\cdot\pi^{-1}}(z^{\boldsymbol{\mu}}_{\pi(a_{p-1})})}
{\partial \mu^{\pi(a_p)-3}}\bigg|_{\boldsymbol{\mu_0}}
=\frac{\partial f^{\boldsymbol{\mu}}_{\pi\cdot\sigma\cdot\pi^{-1}}(\mu^{\pi(a_{p-1})-3})}
{\partial \mu^{\pi(a_p)-3}}\bigg|_{\boldsymbol{\mu_0}}\\
=&\bigg(\frac{\partial}{\partial \mu^{\pi(a_p)-3}}\bigg(
\frac{\mu^{\pi(a_{p-1})-3}-\mu^{\pi(a_p)-3}}{\mu^{\pi(a_p)-3}(\mu^{\pi(a_{p-1})-3}-1)}
\bigg)\bigg)\bigg|_{\boldsymbol{\mu_0}}\\
=&\frac{-\mu^{\pi(a_{p-1})-3}}{(\mu^{\pi(a_p)-3})^2(\mu^{\pi(a_{p-1})-3}-1)}
\bigg|_{\boldsymbol{\mu_0}}\\
=&-(1+2\cos{\frac{2q\pi}{p}}).
\end{align*}
\end{enumerate}
\end{proof}

%----------------------------------------------------------------
\subsection{Representations of the Icosahedral Group $I$}
\label{A_5prep}

For any $\boldsymbol{\lambda}\in K_n$ such that $G_{\boldsymbol{\lambda}}$ is isomorphic to the icosahedral group $I$, recall from Section \ref{sing} that
$$
\Phi_{\boldsymbol{\lambda}}:\:G_{\boldsymbol{\lambda}}\to\mathcal{A}_{[\boldsymbol{\lambda}]},\:g_\sigma\mapsto f^{\boldsymbol{\lambda}}_\sigma
$$
is a group isomorphism. There are five conjugacy classes of $G_{\boldsymbol{\lambda}}$:
\begin{itemize}
\item $K^{(1)}_{\boldsymbol{\lambda}}$: the identity element;
\item $K^{(2)}_{\boldsymbol{\lambda}}$: elements whose images under $\Phi_{\boldsymbol{\lambda}}$ are conjugate to a rotation of $\frac{2\pi}{5}$;
\item $K^{(3)}_{\boldsymbol{\lambda}}$: elements whose images under $\Phi_{\boldsymbol{\lambda}}$ are conjugate to a rotation of $\frac{4\pi}{5}$;
\item $K^{(4)}_{\boldsymbol{\lambda}}$: elements whose images under $\Phi_{\boldsymbol{\lambda}}$ are conjugate to a rotation of $\frac{2\pi}{3}$;
\item $K^{(5)}_{\boldsymbol{\lambda}}$: elements whose images under $\Phi_{\boldsymbol{\lambda}}$ are conjugate to a rotation of $\pi$.
\end{itemize}
Table \ref{chaA_5} is the character table of $G_{\boldsymbol{\lambda}}$, with $X^{(1)}_{\boldsymbol{\lambda}},\: X^{(2)}_{\boldsymbol{\lambda}},\:\cdots,\:X^{(5)}_{\boldsymbol{\lambda}}$ representing the five different irreducible representations of $G_{\boldsymbol{\lambda}}$ and $\chi^{(1)}_{\boldsymbol{\lambda}},\: \chi^{(2)}_{\boldsymbol{\lambda}},\:\cdots,\:\chi^{(5)}_{\boldsymbol{\lambda}}$ their characters.
\begin{table}[h]
\centering
$$
\begin{array}{c|ccccc}
& K^{(1)}_{\boldsymbol{\lambda}} & K^{(2)}_{\boldsymbol{\lambda}} & K^{(3)}_{\boldsymbol{\lambda}} & K^{(4)}_{\boldsymbol{\lambda}} & K^{(5)}_{\boldsymbol{\lambda}} \\
\hline
\chi^{(1)}_{\boldsymbol{\lambda}}& 1 & 1 & 1 & 1 & 1 \\
\chi^{(2)}_{\boldsymbol{\lambda}}& 4 & -1 & -1 & 1 & 0 \\\chi^{(3)}_{\boldsymbol{\lambda}}& 5 & 0 & 0 & -1 & 1 \\
\chi^{(4)}_{\boldsymbol{\lambda}}& 3 & \frac{1+\sqrt5}{2} & \frac{1-\sqrt5}{2} & 0 & -1 \\
\chi^{(5)}_{\boldsymbol{\lambda}}& 3 & \frac{1-\sqrt5}{2} & \frac{1+\sqrt5}{2} & 0 & -1 \\
\end{array}
$$
\caption{Character table of $G_{\boldsymbol{\lambda}}$, with $G_{\boldsymbol{\lambda}}$ isomorphic to $A_5$.}
\label{chaA_5}
\end{table}

\begin{defn}
For any $\boldsymbol{\lambda}\in K_n$ such that $G_{\boldsymbol{\lambda}}$ is isomorphic to the icosahedral group $I$, assume that  $X_{\boldsymbol{\lambda}}=p_1X^{(1)}_{\boldsymbol{\lambda}}\oplus\cdots\oplus p_5X^{(5)}_{\boldsymbol{\lambda}}$. Then call $(p_1,\:p_2,\:\cdots,\:p_5)$ the \textbf{multiplicity vector} of $\boldsymbol{\lambda}$.
\end{defn}

Now all we have to do is to find the multiplicity vector for any $\boldsymbol{\lambda}\in K_n$ such that $G_{\boldsymbol{\lambda}}$ is isomorphic to the icosahedral group $I$. From Section \ref{A_5} we know that $[\boldsymbol{\lambda}]$ is equivalent to one of the following sets (which are categorized into eight types):
\begin{enumerate}
\item \textbf{F+mB}: $F_I\cup B_I(a_1)\cup B_I(a_2)\cup\cdots\cup B_I(a_m)$, where $B_I(a_1),\:B_I(a_2),\:\cdots,\:B_I(a_m)$ are different orbits of order 60, $m\in\mathbb {N}$;
\item \textbf{V+mB}: $V_I\cup B_I(a_1)\cup B_I(a_2)\cup\cdots\cup B_I(a_m)$, where $B_I(a_1),\:B_I(a_2),\:\cdots,\:B_I(a_m)$ are different orbits of order 60, $m\in\mathbb {N}$;
\item \textbf{E+mB}: $E_I\cup B_I(a_1)\cup B_I(a_2)\cup\cdots\cup B_I(a_m)$, where $B_I(a_1),\:B_I(a_2),\:\cdots,\:B_I(a_m)$ are different orbits of order 60, $m\in\mathbb {N}$;
\item \textbf{FV+mB}: $F_I\cup V_I\cup B_I(a_1)\cup B_I(a_2)\cup\cdots\cup B_I(a_m)$, where $B_I(a_1),\:B_I(a_2),\:\cdots,\:B_I(a_m)$ are different orbits of order 60, $m\in\mathbb {N}$;
\item \textbf{VE+mB}: $V_I\cup E_I\cup B_I(a_1)\cup B_I(a_2)\cup\cdots\cup B_I(a)_m)$, where $B_I(a_1),\:B_I(a_2),\:\cdots,\:B_I(a_m)$ are different orbits of order 60, $m\in\mathbb {N}$;
\item \textbf{EF+mB}: $E_I\cup F_I\cup B_I(a_1)\cup B_I(a_2)\cup\cdots\cup B_I(a_m)$, where $B_I(a_1),\:B_I(a_2),\:\cdots,\:B_I(a_m)$ are different orbits of order 60, $m\in\mathbb {N}$;
\item \textbf{FVE+mB}: $F_I\cup V_I\cup E_I\cup B_I(a_1)\cup B_I(a_2)\cup\cdots\cup B_I(a_m)$, where $B_I(a_1),\:B_I(a_2),\:\cdots,\:B_I(a_m)$ are different orbits of order 60, $m\in\mathbb {N}$;
\item \textbf{(1+m)B}: $B_I(a_0)\cup B_I(a_1)\cup\cdots\cup B_I(a_m)$, where $B_I(a_0),\:B_I(a_1),\:\cdots,\:B_I(a_m)$ are different orbits of order 60, $m\in\mathbb {N}$.
\end{enumerate}

From the theorems in Section \ref{cha} we have already known the character of $X_{\boldsymbol{\lambda}}$. The results are shown in Table \ref{chaA5}.
\begin{table}[h]
\centering
$$
\begin{array}{c|ccccc}
& K^{(1)}_{\boldsymbol{\lambda}} & K^{(2)}_{\boldsymbol{\lambda}} & K^{(3)}_{\boldsymbol{\lambda}} & K^{(4)}_{\boldsymbol{\lambda}} & K^{(5)}_{\boldsymbol{\lambda}} \\
\hline
F+mB
& 9+60m  & -1 & -1 & 0 & 1 \\
V+mB
& 17+60m & \frac{-\sqrt5-1}{2} & \frac{\sqrt5-1}{2} & -1 & 1 \\
E+mB
& 27+60m & \frac{-\sqrt5-1}{2} & \frac{\sqrt5-1}{2} & 0 & -1 \\
FV+mB
& 29+60m & -1 & -1 & -1 & 1 \\
VE+mB
& 47+60m & \frac{-\sqrt5-1}{2} & \frac{\sqrt5-1}{2} & -1 & -1 \\
EF+mB
& 39+60m & -1 & -1 & 0 & -1 \\
FVE+mB
& 59+60m & -1 & -1 & -1 & -1 \\
(1+m)B
& 57+60m & \frac{-\sqrt5-1}{2} & \frac{\sqrt5-1}{2} & 0 & 1 \\
\end{array}
$$
\caption{Character of $X_{\boldsymbol{\lambda}}$, with $G_{\boldsymbol{\lambda}}$ isomorphic to $A_5$.}
\label{chaA5}
\end{table}

Thus we have the conclusion
\begin{thm}
For any $\boldsymbol{\lambda}\in K_n$ such that $G_{\boldsymbol{\lambda}}$ is isomorphic to the icosahedral group $I$, we have found its multiplicity vector
\begin{enumerate}
\item \textbf{F+mB}: $(m,\:1+4m,\:1+5m,\:3m,\:3m)$;
\item \textbf{V+mB}: $(m,\:1+4m,\:2+5m,\:3m,\:1+3m)$;
\item \textbf{E+mB}: $(m,\:2+4m,\:2+5m,\:1+3m,\:2+3m)$;
\item \textbf{FV+mB}: $(m,\:2+4m,\:3+5m,\:1+3m,\:1+3m)$;
\item \textbf{VE+mB}: $(m,\:3+4m,\:4+5m,\:2+3m,\:3+3m)$;
\item \textbf{EF+mB}: $(m,\:3+4m,\:3+5m,\:2+3m,\:2+3m)$;
\item \textbf{FVE+mB}: $(m,\:4+4m,\:5+5m,\:3+3m,\:3+3m)$;
\item \textbf{(1+m)B}: $(1+m,\:4+4m,\:5+5m,\:2+3m,\:3+3m)$.
\end{enumerate}
\end{thm}

There is an interesting pattern of the result.
\begin{cor}
For $\boldsymbol {\lambda_1}\in K_{n_1},\:\boldsymbol {\lambda_2}\in K_{n_2}$ with $(a_1,\:a_2,\:\cdots,\:a_5)$ and $(b_1,\:b_2,\:\cdots,\:b_5)$ their multiplicity vectors respectively, if there exist some linear fractional transformations $\phi_1$ and $\phi_2$ such that
$$
\mathcal{A}_{\phi_1([\boldsymbol {\lambda_1}])}=\mathcal{A}_{\phi_2([\boldsymbol {\lambda_2}])}\simeq A_5
$$
with
$\phi_1([\boldsymbol{\lambda_1}])\ne\phi_2([\boldsymbol {\lambda_2}])$, then choose $\boldsymbol {\mu}\in K_{n_1+n_2+3}$ such that there exists some linear fractional transformation $\phi$ satisfying
$$
\phi([\boldsymbol {\mu}])=\phi_1([\boldsymbol{\lambda_1}])\cup\phi_2([\boldsymbol {\lambda_2}]),
$$
then $G_{\boldsymbol {\mu}}\simeq A_5$ and the multiplicity vector of $\boldsymbol {\mu}$ is
$$(a_1+b_1,\:a_2+b_2,\:a_3+b_3,\:a_4+b_4+1,\:a_5+b_5).$$
\end{cor}

%----------------------------------------------------------------
\subsection{Representations of the Octahedral Group $O$}

For any $\boldsymbol{\lambda}\in K_n$ such that $G_{\boldsymbol{\lambda}}$ is isomorphic to the octahedral group $O$, let $\Psi_{\boldsymbol{\lambda}}$ be the isomorphism which maps  $G_{\boldsymbol{\lambda}}$ to the symmetric group $S_4$. There are five conjugacy classes of $G_{\boldsymbol{\lambda}}$:
\begin{itemize}
\item $K^{(1)}_{\boldsymbol{\lambda}}$: the identity element;
\item $K^{(2)}_{\boldsymbol{\lambda}}$: elements whose images under $\Psi_{\boldsymbol{\lambda}}$ are conjugate to $(1234)$;
\item $K^{(3)}_{\boldsymbol{\lambda}}$: elements whose images under $\Psi_{\boldsymbol{\lambda}}$ are conjugate to $(123)$;
\item $K^{(4)}_{\boldsymbol{\lambda}}$: elements whose images under $\Psi_{\boldsymbol{\lambda}}$ are conjugate to $(12)(34)$;
\item $K^{(5)}_{\boldsymbol{\lambda}}$: elements whose images under $\Psi_{\boldsymbol{\lambda}}$ are conjugate to $(12)$.
\end{itemize}
Table \ref{chaS_4} is the character table of $G_{\boldsymbol{\lambda}}$, with $X^{(1)}_{\boldsymbol{\lambda}},\: X^{(2)}_{\boldsymbol{\lambda}},\:\cdots,\:X^{(5)}_{\boldsymbol{\lambda}}$ representing the five different irreducible representations of $G_{\boldsymbol{\lambda}}$ and $\chi^{(1)}_{\boldsymbol{\lambda}},\: \chi^{(2)}_{\boldsymbol{\lambda}},\:\cdots,\:\chi^{(5)}_{\boldsymbol{\lambda}}$ their characters.
\begin{table}[h!]
\centering
$$
\begin{array}{c|ccccc}
& K^{(1)}_{\boldsymbol{\lambda}} & K^{(2)}_{\boldsymbol{\lambda}} & K^{(3)}_{\boldsymbol{\lambda}} & K^{(4)}_{\boldsymbol{\lambda}} & K^{(5)}_{\boldsymbol{\lambda}} \\
\hline
\chi^{(1)}_{\boldsymbol{\lambda}}& 1 & 1  & 1  & 1  & 1  \\
\chi^{(2)}_{\boldsymbol{\lambda}}& 1 & -1 & 1  & 1  & -1 \\\chi^{(3)}_{\boldsymbol{\lambda}}& 3 & -1 & 0  & -1 & 1  \\
\chi^{(4)}_{\boldsymbol{\lambda}}& 3 & 1  & 0  & -1 & -1 \\
\chi^{(5)}_{\boldsymbol{\lambda}}& 2 & 0  & -1 & 2  & 0  \\
\end{array}
$$
\caption{Character table of $G_{\boldsymbol{\lambda}}$, with $G_{\boldsymbol{\lambda}}$ isomorphic to $S_4$.}
\label{chaS_4}
\end{table}

\begin{defn}
For any $\boldsymbol{\lambda}\in K_n$ such that $G_{\boldsymbol{\lambda}}$ is isomorphic to the octahedral group $O$, assume that  $X_{\boldsymbol{\lambda}}=p_1X^{(1)}_{\boldsymbol{\lambda}}\oplus\cdots\oplus p_5X^{(5)}_{\boldsymbol{\lambda}}$. Then call $(p_1,\:p_2,\:\cdots,\:p_5)$ the \textbf{multiplicity vector} of $\boldsymbol{\lambda}$.
\end{defn}

Now all we have to do is to find the multiplicity vector for any $\boldsymbol{\lambda}\in K_n$ such that $G_{\boldsymbol{\lambda}}$ is isomorphic to the octahedral group $O$. From Section \ref{S_4} we know that $[\boldsymbol{\lambda}]$ is equivalent to one of the following sets (which are categorized into eight types):
\begin{enumerate}
\item \textbf{F+mB}: $F_O\cup B_O(a_1)\cup B_O(a_2)\cup\cdots\cup B_O(a_m)$, where $B_O(a_1),\:B_O(a_2),\:\cdots,\:B_O(a_m)$ are different orbits of order 24, $m\in\mathbb {N}$;
\item \textbf{V+mB}: $V_O\cup B_O(a_1)\cup B_O(a_2)\cup\cdots\cup B_O(a_m)$, where $B_O(a_1),\:B_O(a_2),\:\cdots,\:B_O(a_m)$ are different orbits of order 24, $m\in\mathbb {N}$;
\item \textbf{E+mB}: $E_O\cup B_O(a_1)\cup B_O(a_2)\cup\cdots\cup B_O(a_m)$, where $B_O(a_1),\:B_O(a_2),\:\cdots,\:B_O(a_m)$ are different orbits of order 24, $m\in\mathbb {N}$;
\item \textbf{FV+mB}: $F_O\cup V_O\cup B_O(a_1)\cup B_O(a_2)\cup\cdots\cup B_O(a_m)$, where $B_O(a_1),\:B_O(a_2),\:\cdots,\:B_O(a_m)$ are different orbits of order 24, $m\in\mathbb {N}$;
\item \textbf{VE+mB}: $V_O\cup E_O\cup B_O(a_1)\cup B_O(a_2)\cup\cdots\cup B_O(a_m)$, where $B_O(a_1),\:B_O(a_2),\:\cdots,\:B_O(a_m)$ are different orbits of order 24, $m\in\mathbb {N}$;
\item \textbf{EF+mB}: $E_O\cup F_O\cup B_O(a_1)\cup B_O(a_2)\cup\cdots\cup B_O(a_m)$, where $B_O(a_1),\:B_O(a_2),\:\cdots,\:B_O(a_m)$ are different orbits of order 24, $m\in\mathbb {N}$;
\item \textbf{FVE+mB}: $F_O\cup V_O\cup E_O\cup B_O(a_1)\cup B_O(a_2)\cup\cdots\cup B_O(a_m)$, where $B_O(a_1),\:B_O(a_2),\:\cdots,\:B_O(a_m)$ are different orbits of order 24, $m\in\mathbb {N}$;
\item \textbf{(1+m)B}: $B_O(a_0)\cup B_O(a_1)\cup B_O(a_2)\cup\cdots\cup B_O(a_m)$, where $B_O(a_1),\:B_O(a_2),\:\cdots,\:B_O(a_m)$ are different orbits of order 24, $m\in\mathbb {N}$.
\end{enumerate}

From the theorems in Section \ref{cha}, we have already found the character of $X_{\boldsymbol{\lambda}}$. The results are shown in Table \ref{chaS4}.
\begin{table}[h!]
\centering
$$
\begin{array}{c|ccccc}
& K^{(1)}_{\boldsymbol{\lambda}} & K^{(2)}_{\boldsymbol{\lambda}} & K^{(3)}_{\boldsymbol{\lambda}} & K^{(4)}_{\boldsymbol{\lambda}} & K^{(5)}_{\boldsymbol{\lambda}} \\
\hline
F+mB   &  3+24m & -1 & 0 & -1 & 1 \\
V+mB   &  5+24m & -1 & -1 & 1 & 1 \\
E+mB   &  9+24m & -1 & 0 & 1 & -1 \\
FV+mB  & 11+24m & -1 & -1 & -1 & 1 \\
VE+mB  & 17+24m & -1 & -1 & 1 & -1 \\
EF+mB  & 15+24m & -1 & 0 & -1 & -1 \\
FVE+mB & 23+24m & -1 & -1 & -1 & -1 \\
(1+m)B & 21+24m & -1 & 0 & 1 & 1 \\
\end{array}
$$
\caption{Character of $X_{\boldsymbol{\lambda}}$, with $G_{\boldsymbol{\lambda}}$ isomorphic to $S_4$.}
\label{chaS4}
\end{table}
\begin{thm}
For any $\boldsymbol{\lambda}\in K_n$ such that $G_{\boldsymbol{\lambda}}$ is isomorphic to the octahedral group $O$, we have found its multiplicity vector
\begin{enumerate}
\item \textbf{F+mB}: $(m,\:m,\:1+3m,\:3m,\:2m)$;
\item \textbf{V+mB}: $(m,\:m,\:1+3m,\:3m,\:1+2m)$;
\item \textbf{E+mB}: $(m,\:1+m,\:1+3m,\:1+3m,\:1+2m)$;
\item \textbf{FV+mB}: $(m,\:m,\:2+3m,\:1+3m,\:1+2m)$;
\item \textbf{VE+mB}: $(m,\:1+m,\:2+3m,\:2+3m,\:2+2m)$;
\item \textbf{EF+mB}: $(m,\:1+m,\:2+3m,\:2+3m,\:1+2m)$;
\item \textbf{FVE+mB}: $(m,\:1+m,\:3+3m,\:3+3m,\:2+2m)$;
\item \textbf{(1+m)B}: $(1+m,\:1+m,\:3+3m,\:2+3m,\:2+2m)$.
\end{enumerate}
\end{thm}

There is an interesting pattern of the result.
\begin{cor}
For $\boldsymbol {\lambda_1}\in K_{n_1},\:\boldsymbol {\lambda_2}\in K_{n_2}$ with $(a_1,\:a_2,\:\cdots,\:a_5)$ and $(b_1,\:b_2,\:\cdots,\:b_5)$ their multiplicity vectors respectively, if there exist some linear fractional transformations $\phi_1$ and $\phi_2$ such that
$$
\mathcal{A}_{\phi_1([\boldsymbol {\lambda_1}])}=\mathcal{A}_{\phi_2([\boldsymbol {\lambda_2}])}\simeq S_4
$$
with
$\phi_1([\boldsymbol{\lambda_1}])\ne\phi_2([\boldsymbol {\lambda_2}])$, then choose $\boldsymbol {\mu}\in K_{n_1+n_2+3}$ such that there exists some linear fractional transformation $\phi$ satisfying
$$
\phi([\boldsymbol {\mu}])=\phi_1([\boldsymbol{\lambda_1}])\cup\phi_2([\boldsymbol {\lambda_2}]),
$$
then $G_{\boldsymbol {\mu}}\simeq S_4$ and the multiplicity vector of $\boldsymbol {\mu}$ is
$$(a_1+b_1,\:a_2+b_2,\:a_3+b_3,\:a_4+b_4+1,\:a_5+b_5).$$
\end{cor}

%----------------------------------------------------------------
\subsection{Representations of the Tetrahedral Group $T$}

Things are a little complicated in the tetrahedral case, as all rotations of order three are conjugate in the liner fractional transformation group. Let $F_T$ denote the projections of the four central points of the faces of the tetrahedron on the Riemann sphere. That means
$$
F_T=\{A,\:B,\:C,\:D\}
=\{0,\:\sqrt2\:,\sqrt2e^{\frac{2\pi}{3}i},\:\sqrt2e^{\frac{4\pi}{3}i}\}.
$$
Let $V_T$, $E_T$ denote the four vertices and the projections of the six middle points of the edges of the tetrahedron on the Riemann sphere. That means
$$
V_T=\{E,\:F,\:G,\:H\}
=\{\infty,\:\frac{-1}{\sqrt2}\:,\frac{-1}{\sqrt2}e^{\frac{2\pi}{3}i},\:\frac{-1}{\sqrt2}e^{\frac{4\pi}{3}i}\},
$$
$$
E_T=\{I,\:J,\:K,\:L,\:M,\:N\}
=\{\frac{\sqrt2}{1-\sqrt3},\:\frac{\sqrt2}{1-\sqrt3}e^{\frac{2\pi}{3}i},\:\frac{\sqrt2}{1-\sqrt3}e^{\frac{4\pi}{3}i},\:\frac{\sqrt2}{1+\sqrt3},\:\frac{\sqrt2}{1+\sqrt3}e^{\frac{2\pi}{3}i},\:\frac{\sqrt2}{1+\sqrt3}e^{\frac{4\pi}{3}i}\}.
$$
Set
$$
G_0=\langle z\mapsto e^{\frac{2\pi}{3}i}z,\:
z\mapsto \frac{\sqrt2-z}{\sqrt2z+1}\rangle\simeq A_4.
$$
Thus $G_0$ fixes $F_T$, $V_T$ and $E_T$.
\begin{figure}[h!]
\centering
\includegraphics[width=2.3in]{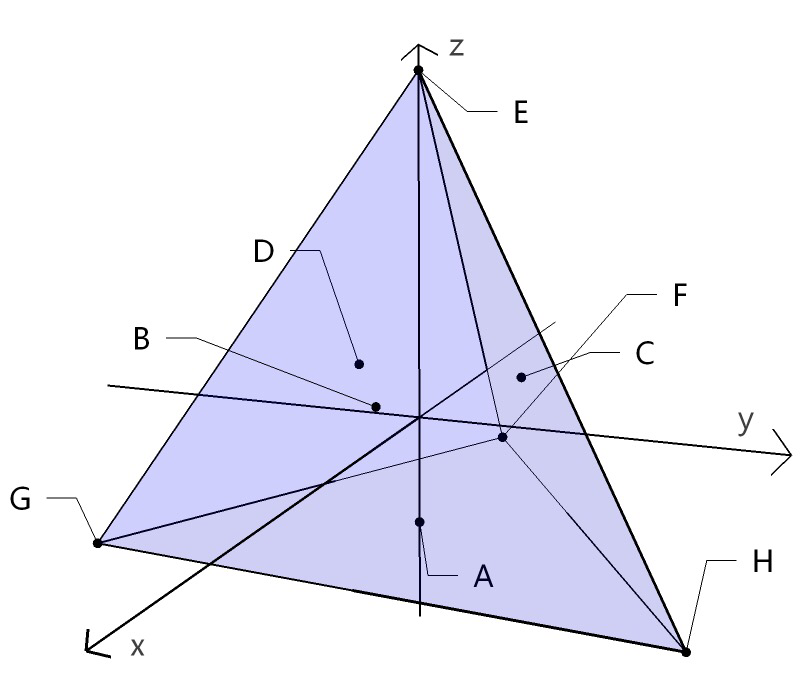}
\includegraphics[width=2.1in]{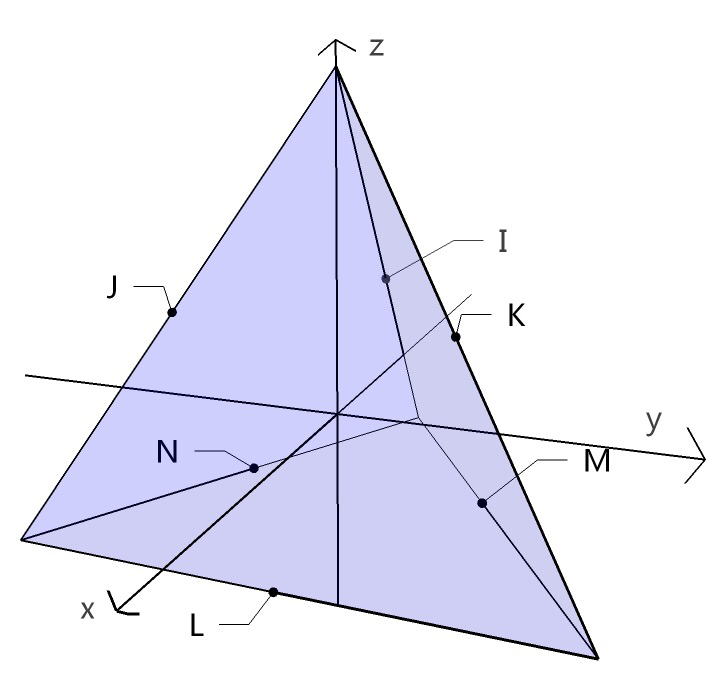}
\caption{Center points on the faces, vertices and middle points on the edges of a tetrahedron.}
\label{A4}
\end{figure}

From Section \ref{A_4} we know that if $G_{\boldsymbol{\lambda}}$ is isomorphic to the Tetrahedral Group $T$, then $[\boldsymbol{\lambda}]$ is equivalent to one of the following sets (which are categorized into six types):
\begin{enumerate}
\item \textbf{F+mB}: $F_T\cup B_T(a_1)\cup B_T(a_2)\cup\cdots\cup B_T(a_m)$, where $B_T(a_1),\:B_T(a_2),\:\cdots,\:B_T(a_m)$ are different orbits of order 12, $m\in\mathbb {N^*}$;
\item \textbf{E+mB}: $E_T\cup B_T(a_1)\cup B_T(a_2)\cup\cdots\cup B_T(a_m)$, where $B_T(a_1),\:B_T(a_2),\:\cdots,\:B_T(a_m)$ are different orbits of order 12, $m\in\mathbb {N^*}$;
\item \textbf{FV+mB}: $F_T\cup V_T\cup B_T(a_1)\cup B_T(a_2)\cup\cdots\cup B_T(a_m)$, where $B_T(a_1),\:B_T(a_2),\:\cdots,\:B_T(a_m)$ are different orbits of order 12, $m\in\mathbb {N^*}$;
\item \textbf{FE+mB}: $F_T\cup E_T\cup B_T(a_1)\cup B_T(a_2)\cup\cdots\cup B_T(a_m)$, where $B_T(a_1),\:B_T(a_2),\:\cdots,\:B_T(a_m)$ are different orbits of order 12, $m\in\mathbb {N}$;
\item \textbf{FVE+mB}: $F_T\cup V_T\cup E_T\cup B_I(a_1)\cup B_I(a_2)\cup\cdots\cup B_I(a_m)$, where $B_I(a_1),\:B_I(a_2),\:\cdots,\:B_I(a_m)$ are different orbits of order 12, $m\in\mathbb {N^*}$;
\item \textbf{(1+m)B}: $B_T(a_1)\cup B_T(a_2)\cup\cdots\cup B_T(a_m)$, where $B_T(a_1),\:B_T(a_2),\:\cdots,\:B_T(a_m)$ are different orbits of order 12, $m\in\mathbb {N}$.
\end{enumerate}
Set $$w=e^{\frac{2\pi}{3}i}.$$

%----------------------------------------------------------------
\subsubsection{Representations of $T$ when
$\boldsymbol{\lambda}$ is of type E+mB, FV+mB, EFV+mB or (1+m)B}

There are four conjugacy classes of $G_{\boldsymbol{\lambda}}$:
\begin{itemize}
\item $K^{(1)}_{\boldsymbol{\lambda}}$: the identity element;
\item $K^{(2)}_{\boldsymbol{\lambda}}$: elements of order two;
\item $K^{(3)}_{\boldsymbol{\lambda}}$: one of the two left classes;
\item $K^{(4)}_{\boldsymbol{\lambda}}$: the other of the two left classes.
\end{itemize}

Now fix $K^{(1)}_{\boldsymbol{\lambda}}$, $K^{(2)}_{\boldsymbol{\lambda}}$, $K^{(3)}_{\boldsymbol{\lambda}}$, $K^{(4)}_{\boldsymbol{\lambda}}$. Table \ref{chaA_41} is the character table of $G_{\boldsymbol{\lambda}}$, with $X^{(1)}_{\boldsymbol{\lambda}},\:X^{(2)}_{\boldsymbol{\lambda}},\:X^{(3)}_{\boldsymbol{\lambda}}$ and $X^{(4)}_{\boldsymbol{\lambda}}$ representing the four different irreducible representations of $G_{\boldsymbol{\lambda}}$ and $\chi^{(1)}_{\boldsymbol{\lambda}},\: \chi^{(2)}_{\boldsymbol{\lambda}},\:\chi^{(3)}_{\boldsymbol{\lambda}}$ and $\chi^{(4)}_{\boldsymbol{\lambda}}$ their characters.
\begin{table}[h]
\centering
$$
\begin{array}{c|cccc}
& K^{(1)}_{\boldsymbol{\lambda}} & K^{(2)}_{\boldsymbol{\lambda}} & K^{(3)}_{\boldsymbol{\lambda}} & K^{(4)}_{\boldsymbol{\lambda}}\\
\hline
\chi^{(1)}_{\boldsymbol{\lambda}}& 1 & 1 & 1   & 1   \\
\chi^{(2)}_{\boldsymbol{\lambda}}& 1 & 1 & w   & w^2 \\
\chi^{(3)}_{\boldsymbol{\lambda}}& 1 & 1 & w^2 & w   \\
\chi^{(4)}_{\boldsymbol{\lambda}}& 3 &-1 & 0   & 0   \\
\end{array}
$$
\caption{Character table of $G_{\boldsymbol{\lambda}}$, with $G_{\boldsymbol{\lambda}}$ isomorphic to $A_4$.}
\label{chaA_41}
\end{table}

\begin{defn}
For any $\boldsymbol{\lambda}\in K_n$ such that $G_{\boldsymbol{\lambda}}$ is isomorphic to the tetrahedral group $T$ and $\boldsymbol{\lambda}$ is of type E+mB, FV+mB, EFV+mB or (1+m)B, assume that $X_{\boldsymbol{\lambda}}=p_1X^{(1)}_{\boldsymbol{\lambda}}\oplus\cdots\oplus p_4X^{(4)}_{\boldsymbol{\lambda}}$. Then call $(p_1,\:p_2,\:p_3,\:p_4)$ the \textbf{multiplicity vector} of $\boldsymbol{\lambda}$.
\end{defn}

\begin{rem}
The definition seems ambiguous as $K^{(3)}_{\boldsymbol{\lambda}}$ and $K^{(4)}_{\boldsymbol{\lambda}}$ are chosen randomly. But we shall prove that $(p_1,\:p_2,\:p_3,\:p_4)$ remains the same however we choose $K^{(3)}_{\boldsymbol{\lambda}}$ and $K^{(4)}_{\boldsymbol{\lambda}}$, which makes the notion multiplicity vector well-defined.
\end{rem}

From the theorems in Section \ref{cha}, we have already found the character of $X_{\boldsymbol{\lambda}}$. The results are shown in Table \ref{chaA41}.

\begin{table}[h]
\centering
$$
\begin{array}{c|cccc}
& K^{(1)}_{\boldsymbol{\lambda}} & K^{(2)}_{\boldsymbol{\lambda}} & K^{(3)}_{\boldsymbol{\lambda}} & K^{(4)}_{\boldsymbol{\lambda}} \\
\hline
E+mB   & 3+12m &-1 & 0 & 0 \\
FV+mB  & 5+12m & 1 & -1 & -1 \\
FVE+mB & 11+12m&-1 & -1 & -1 \\
(1+m)B & 9+12m & 1 & 0 & 0 \\
\end{array}
$$
\caption{Character of $X_{\boldsymbol{\lambda}}$, with $G_{\boldsymbol{\lambda}}$ isomorphic to $A_4$.}
\label{chaA41}
\end{table}

Thus we have the conclusion that the multiplicity vector of $\boldsymbol{\lambda}$ is
\begin{enumerate}
\item \textbf{E+mB}: $(m,\:m,\:m,\:1+3m)$;
\item \textbf{FV+mB}: $(m,\:1+m,\:1+m,\:1+3m)$;
\item \textbf{FVE+mB}: $(m,\:1+m,\:1+m,\:3+3m)$;
\item \textbf{(1+m)B}: $(m,\:1+m,\:1+m,\:2+3m)$.
\end{enumerate}

%----------------------------------------------------------------
\subsubsection{Representations of $T$ when
$\boldsymbol{\lambda}$ is of type F+mB or FE+mB}

If $\boldsymbol{\lambda}$ is of type F+mB or FE+mB, let $\psi$ be a linear fractional transformation such that
$$
\psi([\boldsymbol{\lambda}])=F_T\cup B_T(a_1)\cup B_T(a_2)\cup\cdots\cup B_T(a_m)
$$
or
$$
\psi([\boldsymbol{\lambda}])=F_T\cup E_T\cup B_T(a_1)\cup B_T(a_2)\cup\cdots\cup B_T(a_m).
$$
Let $\rho$ denote the element in $G_{\boldsymbol{\lambda}}$ such that
$$
\psi\circ\Phi_{\boldsymbol{\lambda}}(\rho)\circ\psi^{-1}(z)=wz.
$$
There are four conjugacy classes of $G_{\boldsymbol{\lambda}}$:
\begin{itemize}
\item $K^{(1)}_{\boldsymbol{\lambda}}$: the identity element;
\item $K^{(2)}_{\boldsymbol{\lambda}}$: elements of order two;
\item $K^{(3)}_{\boldsymbol{\lambda}}$: the conjugate class of $\rho$;
\item $K^{(4)}_{\boldsymbol{\lambda}}$: the conjugate class of $\rho^2$.
\end{itemize}

Now fix $K^{(1)}_{\boldsymbol{\lambda}}$, $K^{(2)}_{\boldsymbol{\lambda}}$, $K^{(3)}_{\boldsymbol{\lambda}}$, $K^{(4)}_{\boldsymbol{\lambda}}$. Table \ref{chaA_42} is the character table of $G_{\boldsymbol{\lambda}}$, with $X^{(1)}_{\boldsymbol{\lambda}},\:X^{(2)}_{\boldsymbol{\lambda}},\:X^{(3)}_{\boldsymbol{\lambda}}$ and $X^{(4)}_{\boldsymbol{\lambda}}$ representing the four different irreducible representations of $G_{\boldsymbol{\lambda}}$ and $\chi^{(1)}_{\boldsymbol{\lambda}},\: \chi^{(2)}_{\boldsymbol{\lambda}},\:\chi^{(3)}_{\boldsymbol{\lambda}}$ and $\chi^{(4)}_{\boldsymbol{\lambda}}$ their characters.
\begin{table}[h]
\centering
$$
\begin{array}{c|cccc}
& K^{(1)}_{\boldsymbol{\lambda}} & K^{(2)}_{\boldsymbol{\lambda}} & K^{(3)}_{\boldsymbol{\lambda}} & K^{(4)}_{\boldsymbol{\lambda}}\\
\hline
\chi^{(1)}_{\boldsymbol{\lambda}}& 1 & 1 & 1   & 1   \\
\chi^{(2)}_{\boldsymbol{\lambda}}& 1 & 1 & w   & w^2 \\
\chi^{(3)}_{\boldsymbol{\lambda}}& 1 & 1 & w^2 & w   \\
\chi^{(4)}_{\boldsymbol{\lambda}}& 3 &-1 & 0   & 0   \\
\end{array}
$$
\caption{Character table of $G_{\boldsymbol{\lambda}}$, with $G_{\boldsymbol{\lambda}}$ isomorphic to $A_4$.}
\label{chaA_42}
\end{table}

\begin{defn}
For any $\boldsymbol{\lambda}\in K_n$ such that $G_{\boldsymbol{\lambda}}$ is isomorphic to the tetrahedral group $T$ and $\boldsymbol{\lambda}$ is of type F+mB or FE+mB, assume that $X_{\boldsymbol{\lambda}}=p_1X^{(1)}_{\boldsymbol{\lambda}}\oplus\cdots\oplus p_4X^{(4)}_{\boldsymbol{\lambda}}$. Then call $(p_1,\:p_2,\:p_3,\:p_4)$ the \textbf{multiplicity vector} of $\boldsymbol{\lambda}$.
\end{defn}

\begin{rem}
The definition seems ambiguous as $\psi$ is chosen randomly. But we shall prove that $(p_1,\:p_2,\:p_3,\:p_4)$ remains the same however we choose $\psi$, which makes the notion multiplicity vector well-defined.
\end{rem}

From the theorems in Section \ref{cha}, we have already found the character of $X_{\boldsymbol{\lambda}}$. The results are shown in Table \ref{chaA42}.

\begin{table}[h]
\centering
$$
\begin{array}{c|cccc}
& K^{(1)}_{\boldsymbol{\lambda}} & K^{(2)}_{\boldsymbol{\lambda}} & K^{(3)}_{\boldsymbol{\lambda}} & K^{(4)}_{\boldsymbol{\lambda}} \\
\hline
F+mB  & 1+12m & 1 & w & w^2 \\
FE+mB & 7+12m &-1 & w & w^2 \\
\end{array}
$$
\caption{Character of $X_{\boldsymbol{\lambda}}$, with $G_{\boldsymbol{\lambda}}$ isomorphic to $A_4$.}
\label{chaA42}
\end{table}

Thus we have the conclusion that the multiplicity vector of $\boldsymbol{\lambda}$ is
\begin{enumerate}
\item \textbf{F+mB}: $(m,\:1+m,\:m,\:3m)$;
\item \textbf{FE+mB}: $(m,\:1+m,\:m,\:2+3m)$.
\end{enumerate}

\begin{thm}
For each singularity $[\boldsymbol{\lambda}]$ such that $G_{\boldsymbol{\lambda}}$ is isomorphic to the tetrahedral group $T$, we have found its multiplicity vector
\begin{enumerate}
\item \textbf{F+mB}: $(m,\:1+m,\:m,\:3m)$;
\item \textbf{E+mB}: $(m,\:m,\:m,\:1+3m)$;
\item \textbf{FV+mB}: $(m,\:1+m,\:1+m,\:1+3m)$;
\item \textbf{FE+mB}: $(m,\:1+m,\:m,\:2+3m)$;
\item \textbf{FVE+mB}: $(m,\:1+m,\:1+m,\:3+3m)$;
\item \textbf{(1+m)B}: $(m,\:1+m,\:1+m,\:2+3m)$.
\end{enumerate}
\end{thm}

%----------------------------------------------------------------
\subsection{Representations of the Dihedral Group $D_p$}
\label{D_pprep}

For any $\boldsymbol{\lambda}\in K_n$ such that $G_{\boldsymbol{\lambda}}$ is isomorphic to the dihedral group $D_p$, recall from Section \ref{sing} that
$$
\Phi_{\boldsymbol{\lambda}}:\:G_{\boldsymbol{\lambda}}\to\mathcal{A}_{[\boldsymbol{\lambda}]},\:g_\sigma\mapsto f^{\boldsymbol{\lambda}}_\sigma
$$
is a group isomorphism. From Section \ref{D_n} we know that $[\boldsymbol{\lambda}]$ is equivalent to one of the following sets (which are categorized into six types):
\begin{enumerate}
\item \textbf{mC}: $C_p(a_1)\cup C_p(a_2)\cup\cdots\cup C_p(a_m)$, where $C_p(a_1),\:C_p(a_2),\:\cdots,\:C_p(a_m)$ are different orbits of order $2p$, $m\in\mathbb {N^*}$;
\item \textbf{A+mC}: $A_p\cup C_p(a_1)\cup C_p(a_2)\cup\cdots\cup C_p(a_m)$, where $C_p(a_1),\:C_p(a_2),\:\cdots,\:C_p(a_m)$ are different orbits of order $2p$, $m\in\mathbb {N}$;
\item \textbf{AB+mC}: $A_p\cup B_p\cup C_p(a_1)\cup C_p(a_2)\cup\cdots\cup C_p(a_m)$, where $C_p(a_1),\:C_p(a_2),\:\cdots,\:C_p(a_m)$ are different orbits of order $2p$, $m\in\mathbb {N^*}$;
\item \textbf{2+mC}: $\{0,\:\infty\}\cup C_p(a_1)\cup C_p(a_2)\cup\cdots\cup C_p(a_m)$, where $C_p(a_1),\:C_p(a_2),\:\cdots,\:C_p(a_m)$ are different orbits of order $2p$, $m\in\mathbb {N^*}$;
\item \textbf{A+2+mC}: $A_p\cup\{0,\:\infty\}\cup C_p(a_1)\cup C_p(a_2)\cup\cdots\cup C_p(a_m)$, where $C_p(a_1),\:C_p(a_2),\:\cdots,\:C_p(a_m)$ are different orbits of order $2p$, $m\in\mathbb {N}$;
\item \textbf{AB+2+mC}: $A_p\cup B_p\cup\{0,\:\infty\}\cup C_p(a_1)\cup C_p(a_2)\cup\cdots\cup C_p(a_m)$, where $C_p(a_1),\:C_p(a_2),\:\cdots,\:C_p(a_m)$ are different orbits of order $2p$, $m\in\mathbb {N^*}$.
\end{enumerate}

%----------------------------------------------------------------
\subsubsection{Representations of $D_p$ when $p$ is odd}
\label{D_poddprep}

When $G_{\boldsymbol{\lambda}}$ is isomorphic to the dihedral group $D_p$ ($p$ is odd), there are $\frac{p+3}{2}$ conjugacy classes of $G_{\boldsymbol{\lambda}}$:
\begin{itemize}
\item $K^{(l)}_{\boldsymbol{\lambda}}$: elements whose images under $\Phi_{\boldsymbol{\lambda}}$ are conjugate to a rotation of $\frac{2l\pi}{p}$, where $l=1,\:2,:\cdots,\:\frac{p-1}{2}$;
\item $K^{(\frac{p+1}{2})}_{\boldsymbol{\lambda}}$: the identity element;
\item $K^{(\frac{p+3}{2})}_{\boldsymbol{\lambda}}$: elements of order two.
\end{itemize}
Table \ref{chaD_podd} is the character table of $G_{\boldsymbol{\lambda}}$, with $X^{(1)}_{\boldsymbol{\lambda}},\: X^{(2)}_{\boldsymbol{\lambda}},\:\cdots,\:X^{(\frac{p+3}{2})}_{\boldsymbol{\lambda}}$ representing the $\frac{p+3}{2}$ different irreducible representations of $G_{\boldsymbol{\lambda}}$ and $\chi^{(1)}_{\boldsymbol{\lambda}},\: \chi^{(2)}_{\boldsymbol{\lambda}},\:\cdots,\:\chi^{(\frac{p+3}{2})}_{\boldsymbol{\lambda}}$ their characters.
\begin{table}[h]
\centering
$$
\begin{array}{c|ccccccc}
& K^{(1)}_{\boldsymbol{\lambda}} & \cdots
& K^{(k)}_{\boldsymbol{\lambda}} & \cdots
& K^{(\frac{p-1}{2})}_{\boldsymbol{\lambda}}
& K^{(\frac{p+1}{2})}_{\boldsymbol{\lambda}}
& K^{(\frac{p+3}{2})}_{\boldsymbol{\lambda}} \\
\hline
\chi^{(1)}_{\boldsymbol{\lambda}}
& 2\cos{\frac{2\pi}{p}} & \cdots
& 2\cos{\frac{2k\pi}{p}} & \cdots
& 2\cos{\frac{(p-1)\pi}{p}} & 2 & 0 \\
\vdots & \vdots & & \vdots & & \vdots & \vdots & \vdots \\
\chi^{(l)}_{\boldsymbol{\lambda}}
& 2\cos{\frac{2l\pi}{p}} & \cdots
& 2\cos{\frac{2kl\pi}{p}} & \cdots
& 2\cos{\frac{(p-1)l\pi}{p}} & 2 & 0 \\
\vdots & \vdots & & \vdots & & \vdots & \vdots & \vdots \\
\chi^{(\frac{p-1}{2})}_{\boldsymbol{\lambda}}
& 2\cos{\frac{(p-1)\pi}{p}} & \cdots
& 2\cos{\frac{(p-1)k\pi}{p}} & \cdots
& 2\cos{\frac{(p-1)^2\pi}{2p}} & 2 & 0 \\
\chi^{(\frac{p+1}{2})}_{\boldsymbol{\lambda}}
& 1 & \cdots & 1 & \cdots & 1 & 1 & 1 \\
\chi^{(\frac{p+3}{2})}_{\boldsymbol{\lambda}}
& 1 & \cdots & 1 & \cdots & 1 & 1 & -1 \\
\end{array}
$$
\caption{Character table of $G_{\boldsymbol{\lambda}}$, with $G_{\boldsymbol{\lambda}}$ isomorphic to $D_p$ and $p$ odd.}
\label{chaD_podd}
\end{table}

\begin{defn}
For any $\boldsymbol{\lambda}\in K_n$ such that $G_{\boldsymbol{\lambda}}$ is isomorphic to the dihedral group $D_p$ ($p$ is odd), assume that  $X_{\boldsymbol{\lambda}}=a_1X^{(1)}_{\boldsymbol{\lambda}}\oplus\cdots\oplus a_{\frac{p+3}{2}}X^{(\frac{p+3}{2})}_{\boldsymbol{\lambda}}$. Then call $(a_1,\:a_2,\:\cdots,\:a_{\frac{p+3}{2}})$ the \textbf{multiplicity vector} of $\boldsymbol{\lambda}$.
\end{defn}

From the theorems in Section \ref{cha} we have already known the character of $X_{\boldsymbol{\lambda}}$. The results are shown in Table \ref{chaDpodd}.
\begin{table}[h]
\centering
$$
\begin{array}{c|ccccccc}
& K^{(1)}_{\boldsymbol{\lambda}} & \cdots
& K^{(k)}_{\boldsymbol{\lambda}} & \cdots
& K^{(\frac{p-1}{2})}_{\boldsymbol{\lambda}}
& K^{(\frac{p+1}{2})}_{\boldsymbol{\lambda}}
& K^{(\frac{p+3}{2})}_{\boldsymbol{\lambda}}
\\\hline
mC
& -1-2\cos{\frac{2\pi}{p}} & \cdots
& -1-2\cos{\frac{2k\pi}{p}} & \cdots
& -1-2\cos{\frac{(p-1)\pi}{p}} & 2mp-3 & 1
\\
A+mC
& -1-2\cos{\frac{2\pi}{p}} & \cdots
& -1-2\cos{\frac{2k\pi}{p}} & \cdots
& -1-2\cos{\frac{(p-1)\pi}{p}} & (2m+1)p-3 & 0
\\
AB+mC
& -1-2\cos{\frac{2\pi}{p}} & \cdots
& -1-2\cos{\frac{2k\pi}{p}} & \cdots
& -1-2\cos{\frac{(p-1)\pi}{p}} & (2m+2)p-3 & -1
\\
2+mC
& -1 & \cdots & -1 & \cdots & -1 & 2mp-1 & 1
\\
A+2+mC
& -1 & \cdots & -1 & \cdots & -1 & (2m+1)p-1 & 0
\\
AB+2+mC
& -1 & \cdots & -1 & \cdots & -1 & (2m+2)p-1 & -1
\\
\end{array}
$$
\caption{Character of $X_{\boldsymbol{\lambda}}$, with $G_{\boldsymbol{\lambda}}$ isomorphic to $D_p$ and $p$ odd.}
\label{chaDpodd}
\end{table}

Thus we have the conclusion
\begin{thm}
For any $\boldsymbol{\lambda}\in K_n$ such that $G_{\boldsymbol{\lambda}}$ is isomorphic to the dihedral group $D_p$ ($p$ is odd), we have found its multiplicity vector
\begin{enumerate}
  \item \textbf{mC}:
    \begin{itemize}
      \item $(2m-1,\:m,\:m-1)$ for $p=3$;
      \item $(2m-1,\:2m,\:\cdots,\:2m,\:m,\:m-1)$ for $p\ge5$;
    \end{itemize}
  \item \textbf{A+mC}:
    \begin{itemize}
      \item $(2m,\:m,\:m)$ for $p=3$;
      \item $(2m,\:2m+1,\:\cdots,\:2m+1,\:m,\:m)$ for $p\ge5$;
    \end{itemize}
  \item \textbf{AB+mC}:
    \begin{itemize}
      \item $(2m+1,\:m,\:m+1)$ for $p=3$;
      \item $(2m+1,\:2m+2,\:\cdots,\:2m+2,\:m,\:m+1)$ for $p\ge5$;
    \end{itemize}
  \item \textbf{2+mC}: $(2m,\:\cdots,\:2m,\:m,\:m-1)$;
  \item \textbf{A+2+mC}: $(2m+1,\:\cdots,\:2m+1,\:m,\:m)$;
  \item \textbf{AB+2+mC}: $(2m+2,\:\cdots,\:2m+2,\:m,\:m+1)$.
\end{enumerate}
\end{thm}

%----------------------------------------------------------------
\subsubsection{Representations of $D_p$ when $p$ is even}
\label{D_pevenprep}

When $G_{\boldsymbol{\lambda}}$ is isomorphic to the dihedral group $D_p$ ($p$ is even), there are $\frac{p+6}{2}$ conjugacy classes of $G_{\boldsymbol{\lambda}}$:
\begin{itemize}
\item $K^{(l)}_{\boldsymbol{\lambda}}$: elements whose images under $\Phi_{\boldsymbol{\lambda}}$ are conjugate to a rotation of $\frac{2l\pi}{p}$, where $l=1,\:2,:\cdots,\:\frac{p}{2}$;
\item $K^{(\frac{p+2}{2})}_{\boldsymbol{\lambda}}$: the identity element.
\end{itemize}
There are only two conjugacy classes left. If $\boldsymbol{\lambda}$ is of type A+mC or A+2+mC, set
\begin{itemize}
\item $K^{(\frac{p+4}{2})}_{\boldsymbol{\lambda}}$: elements which have two fixed points in $[\boldsymbol{\lambda}]$ and are not in $K^{(l)}_{\boldsymbol{\lambda}}$, $l=1,\:2,:\cdots,\:\frac{p+2}{2}$;
\item $K^{(\frac{p+6}{2})}_{\boldsymbol{\lambda}}$: elements which have no fixed point in $[\boldsymbol{\lambda}]$ and are not in $K^{(l)}_{\boldsymbol{\lambda}}$, $l=1,\:2,:\cdots,\:\frac{p+2}{2}$.
\end{itemize}
If $\boldsymbol{\lambda}$ is of type mC, AB+mC, 2+mC or AB+2+mC, just set $K^{(\frac{p+4}{2})}_{\boldsymbol{\lambda}}$ to be any of the left two conjugacy classes and $K^{(\frac{p+6}{2})}_{\boldsymbol{\lambda}}$ the other one.

Now fix $K^{(1)}_{\boldsymbol{\lambda}},\:K^{(2)}_{\boldsymbol{\lambda}},\cdots, K^{(\frac{p+6}{2})}_{\boldsymbol{\lambda}}$. Table \ref{chaD_peven} is the character table of $G_{\boldsymbol{\lambda}}$, with $X^{(1)}_{\boldsymbol{\lambda}},\: X^{(2)}_{\boldsymbol{\lambda}},\:\cdots,\:X^{(\frac{p+6}{2})}_{\boldsymbol{\lambda}}$ representing the $\frac{p+6}{2}$ different irreducible representations of $G_{\boldsymbol{\lambda}}$ and $\chi^{(1)}_{\boldsymbol{\lambda}},\: \chi^{(2)}_{\boldsymbol{\lambda}},\:\cdots,\:\chi^{(\frac{p+6}{2})}_{\boldsymbol{\lambda}}$ their characters.
\begin{table}[h]
\centering
$$
\begin{array}{c|cccccccc}
& K^{(1)}_{\boldsymbol{\lambda}} & \cdots
& K^{(k)}_{\boldsymbol{\lambda}} & \cdots
& K^{(\frac{p}{2})}_{\boldsymbol{\lambda}}
& K^{(\frac{p+2}{2})}_{\boldsymbol{\lambda}}
& K^{(\frac{p+4}{2})}_{\boldsymbol{\lambda}}
& K^{(\frac{p+6}{2})}_{\boldsymbol{\lambda}}  \\
\hline
\chi^{(1)}_{\boldsymbol{\lambda}}
& 2\cos{\frac{2\pi}{p}} & \cdots
& 2\cos{\frac{2k\pi}{p}} & \cdots
& -2 & 2 & 0 & 0 \\
\vdots & \vdots & & \vdots & & \vdots & \vdots & \vdots &\vdots\\
\chi^{(l)}_{\boldsymbol{\lambda}}
& 2\cos{\frac{2l\pi}{p}} & \cdots
& 2\cos{\frac{2kl\pi}{p}} & \cdots
& 2(-1)^l & 2 & 0 & 0 \\
\vdots & \vdots & & \vdots & & \vdots & \vdots & \vdots &\vdots\\
\chi^{(\frac{p-2}{2})}_{\boldsymbol{\lambda}}
& 2\cos{\frac{(p-2)\pi}{p}} & \cdots
& 2\cos{\frac{(p-2)k\pi}{p}} & \cdots
& 2(-1)^{\frac{p-2}{2}} & 2 & 0 & 0\\
\chi^{(\frac{p}{2})}_{\boldsymbol{\lambda}}
& 1 & \cdots & 1 & \cdots & 1 & 1 & 1 & 1 \\
\chi^{(\frac{p+2}{2})}_{\boldsymbol{\lambda}}
& 1 & \cdots & 1 & \cdots & 1 & 1 & -1 & -1\\
\chi^{(\frac{p+4}{2})}_{\boldsymbol{\lambda}}
& -1 & \cdots & (-1)^k & \cdots & (-1)^{\frac{p}{2}} & 1 & 1&-1\\
\chi^{(\frac{p+6}{2})}_{\boldsymbol{\lambda}}
& -1 & \cdots & (-1)^k & \cdots & (-1)^{\frac{p}{2}} & 1 & -1&1\\
\end{array}
$$
\caption{Character table of $G_{\boldsymbol{\lambda}}$, with $G_{\boldsymbol{\lambda}}$ isomorphic to $D_p$ and $p$ even.}
\label{chaD_peven}
\end{table}

\begin{defn}
For any $\boldsymbol{\lambda}\in K_n$ such that $G_{\boldsymbol{\lambda}}$ is isomorphic to the dihedral group $D_p$ ($p$ is even), assume that  $X_{\boldsymbol{\lambda}}=a_1X^{(1)}_{\boldsymbol{\lambda}}\oplus\cdots\oplus a_{\frac{p+6}{2}}X^{(\frac{p+6}{2})}_{\boldsymbol{\lambda}}$. Then call $(a_1,\:a_2,\:\cdots,\:a_{\frac{p+6}{2}})$ the \textbf{multiplicity vector} of $\boldsymbol{\lambda}$.
\end{defn}

\begin{rem}
The definition seems ambiguous when $\boldsymbol{\lambda}$ is of type mC, AB+mC, 2+mC or AB+2+mC as $K^{(\frac{p+4}{2})}_{\boldsymbol{\lambda}}$ and $K^{(\frac{p+6}{2})}_{\boldsymbol{\lambda}}$ are chosen randomly. But we shall prove that $(a_1,\:a_2,\:\cdots,\:a_{\frac{p+6}{2}})$ remains the same however we choose $K^{(\frac{p+4}{2})}_{\boldsymbol{\lambda}}$ and $K^{(\frac{p+6}{2})}_{\boldsymbol{\lambda}}$, which makes the notion multiplicity vector well-defined.
\end{rem}

From the theorems in Section \ref{cha} we have already known the character of $X_{\boldsymbol{\lambda}}$. The results are shown in Table \ref{chaDpeven}.
\begin{table}[h]
\centering
$$
\begin{array}{c|cccccccc}
\centering
& K^{(1)}_{\boldsymbol{\lambda}} & \cdots
& K^{(k)}_{\boldsymbol{\lambda}} & \cdots
& K^{(\frac{p}{2})}_{\boldsymbol{\lambda}}
& K^{(\frac{p+2}{2})}_{\boldsymbol{\lambda}}
& K^{(\frac{p+4}{2})}_{\boldsymbol{\lambda}}
& K^{(\frac{p+6}{2})}_{\boldsymbol{\lambda}}  \\
\\\hline
\mathrm{mC}
& -1-2\cos{\frac{2\pi}{p}} & \cdots
& -1-2\cos{\frac{2k\pi}{p}} & \cdots
& -1-2\cos{\frac{(p-1)\pi}{p}} & 2mp-3 & 1 & 1
\\
\mathrm{A+mC}
& -1-2\cos{\frac{2\pi}{p}} & \cdots
& -1-2\cos{\frac{2k\pi}{p}} & \cdots
& -1-2\cos{\frac{(p-1)\pi}{p}} & (2m+1)p-3 & -1 & 1
\\
\mathrm{AB+mC}
& -1-2\cos{\frac{2\pi}{p}} & \cdots
& -1-2\cos{\frac{2k\pi}{p}} & \cdots
& -1-2\cos{\frac{(p-1)\pi}{p}} & (2m+2)p-3 & -1 & -1
\\
\mathrm{2+mC}
& -1 & \cdots & -1 & \cdots & -1 & 2mp-1 & 1 & 1
\\
\mathrm{A+2+mC}
& -1 & \cdots & -1 & \cdots & -1 & (2m+1)p-1 & -1 & 1
\\
\mathrm{AB+2+mC}
& -1 & \cdots & -1 & \cdots & -1 & (2m+2)p-1 & -1 & -1
\\
\end{array}
$$
\caption{Character of $X_{\boldsymbol{\lambda}}$, with $G_{\boldsymbol{\lambda}}$ isomorphic to $D_p$ and $p$ even.}
\label{chaDpeven}
\end{table}

Thus we have the conclusion
\begin{thm}
For any $\boldsymbol{\lambda}\in K_n$ such that $G_{\boldsymbol{\lambda}}$ is isomorphic to the dihedral group $D_p$ ($p$ is even), we have found its multiplicity vector
\begin{enumerate}
  \item \textbf{mC}:
    \begin{itemize}
      \item $(m,\:m-1,\:m-1,\:m-1)$ for $p=2$;
      \item $(2m-1,\:m,\:m-1,\:m,\:m)$ for $p=4$;
      \item $(2m-1,\:2m,\:\cdots,\:2m,\:m,\:m-1,\:m,\:m)$ for $p\ge6$;
    \end{itemize}
  \item \textbf{A+mC}:
    \begin{itemize}
      \item $(m,\:m,\:m-1,\:m)$ for $p=2$;
      \item $(2m,\:m,\:m,\:m,\:m+1)$ for $p=4$;
      \item $(2m,\:2m+1,\:\cdots,\:2m+1,\:m,\:m,\:m,\:m+1)$ for $p\ge6$;
    \end{itemize}
  \item \textbf{AB+mC}:
    \begin{itemize}
      \item $(m,\:m+1,\:m,\:m)$ for $p=2$;
      \item $(2m+1,\:m,\:m+1,\:m+1,\:m+1)$ for $p=4$;
      \item $(2m+1,\:2m+2,\:\cdots,\:2m+2,\:m,\:m+1,\:m+1,\:m+1)$ for $p\ge6$;
    \end{itemize}
  \item \textbf{2+mC}:
    \begin{itemize}
      \item $(m,\:m-1,\:m,\:m)$ for $p=2$;
      \item $(2m,\:\cdots,\:2m,\:m,\:m-1,\:m,\:m)$ for $p\ge4$;
    \end{itemize}
  \item \textbf{A+2+mC}:
    \begin{itemize}
      \item $(m,\:m,\:m,\:m+1)$ for $p=2$;
      \item $(2m+1,\:\cdots,\:2m+1,\:m,\:m,\:m,\:m+1)$ for $p\ge4$;
    \end{itemize}
  \item \textbf{AB+2+mC}:
    \begin{itemize}
      \item $(m,\:m+1,\:m+1,\:m+1)$ for $p=2$;
      \item $(2m+2,\:\cdots,\:2m+2,\:m,\:m+1,\:m+1,\:m+1)$ for $p\ge4$.
    \end{itemize}
\end{enumerate}
\end{thm}

\begin{rem}
When $G_{\boldsymbol{\lambda}}$ is isomorphic to $D_2$, $\boldsymbol{\lambda}$ is of type 2+mC if and only if $\boldsymbol{\lambda}$ is of type A+mC, and
$\boldsymbol{\lambda}$ is of type A+2+mC if and only if $\boldsymbol{\lambda}$ is of type AB+mC. Notice that the multiplicity vectors are different when $\boldsymbol{\lambda}$ is viewed as an element of different types.
\end{rem}

%----------------------------------------------------------------
\subsection{Representations of the Cyclic Group $\mathbb{Z}_p$}
\label{Z_pprep}

For any $\boldsymbol{\lambda}\in K_n$ such that $G_{\boldsymbol{\lambda}}$ is isomorphic to the cyclic group $\mathbb{Z}_p$, recall from Section \ref{sing} that
$$
\Phi_{\boldsymbol{\lambda}}:\:G_{\boldsymbol{\lambda}}\to\mathcal{A}_{[\boldsymbol{\lambda}]},\:g_\sigma\mapsto f^{\boldsymbol{\lambda}}_\sigma
$$
is a group isomorphism. From Section \ref{Z_n} we know that $[\boldsymbol{\lambda}]$ is equivalent to one of the following sets (which are categorized into three types):
\begin{enumerate}
\item \textbf{mC}: $C_p(a_1)\cup C_p(a_2)\cup\cdots\cup C_p(a_m)$, where $C_p(a_1),\:C_p(a_2),\:\cdots,\:C_p(a_m)$ are different orbits of order $p$, $m\in\mathbb {N^*}$;
\item \textbf{1+mC}: $\{0\}\cup C_p(a_1)\cup C_p(a_2)\cup\cdots\cup C_p(a_m)$, where $C_p(a_1),\:C_p(a_2),\:\cdots,\:C_p(a_m)$ are different orbits of order $p$, $m\in\mathbb {N^*}$;
\item \textbf{2+mC}: $\{0,\:\infty\}\cup C_p(a_1)\cup C_p(a_2)\cup\cdots\cup C_p(a_m)$, where $C_p(a_1),\:C_p(a_2),\:\cdots,\:C_p(a_m)$ are different orbits of order $p$, $m\in\mathbb {N^*}$.
\end{enumerate}
Set $$w=e^{\frac{2\pi}{p}i}.$$

%----------------------------------------------------------------
\subsubsection{Representations of $\mathbb{Z}_p$ when
$\boldsymbol{\lambda}$ is of type mC or 2+mC}

When $\boldsymbol{\lambda}$ is of type mC or 2+mC, let $\rho$ denote an element in $G_{\boldsymbol{\lambda}}$ such that $\Phi_{\boldsymbol{\lambda}}(\rho)$ is a rotation of $\frac{2\pi}{p}$. Set
$$
K^{(k)}_{\boldsymbol{\lambda}}=\{\rho^k\}
$$
for $k=1,\:2,\:\cdots,\:p$, and we have got the $p$ conjugate classes of $G_{\boldsymbol{\lambda}}$.

Now fix $K^{(1)}_{\boldsymbol{\lambda}},\:K^{(2)}_{\boldsymbol{\lambda}},\cdots, K^{(p)}_{\boldsymbol{\lambda}}$. Table \ref{chaZ_p1} is the character table of $G_{\boldsymbol{\lambda}}$, with $X^{(1)}_{\boldsymbol{\lambda}},\: X^{(2)}_{\boldsymbol{\lambda}},\:\cdots,\:X^{(p)}_{\boldsymbol{\lambda}}$ representing the $p$ different irreducible representations of $G_{\boldsymbol{\lambda}}$ and $\chi^{(1)}_{\boldsymbol{\lambda}},\: \chi^{(2)}_{\boldsymbol{\lambda}},\:\cdots,\:\chi^{(p)}_{\boldsymbol{\lambda}}$ their characters.
\begin{table}[h]
\centering
$$
\begin{array}{c|ccccccc}
& K^{(1)}_{\boldsymbol{\lambda}}
& K^{(2)}_{\boldsymbol{\lambda}} & \cdots
& K^{(k)}_{\boldsymbol{\lambda}} & \cdots
& K^{(p)}_{\boldsymbol{\lambda}} \\
\hline
\chi^{(1)}_{\boldsymbol{\lambda}}
       & w      & w^2    & \cdots & w^k    & \cdots & 1 \\
\chi^{(2)}_{\boldsymbol{\lambda}}
       & w^2    & w^4    & \cdots & w^{2k} & \cdots & 1 \\
\vdots & \vdots & \vdots &        & \vdots &        & \vdots \\
\chi^{(l)}_{\boldsymbol{\lambda}}
       & w^l    & w^{2l} & \cdots & w^{kl} & \cdots & 1 \\
\vdots & \vdots & \vdots &        & \vdots &        & \vdots \\
\chi^{(p)}_{\boldsymbol{\lambda}}
       & 1      & 1      & \cdots & 1      & \cdots & 1 \\
\end{array}
$$
\caption{Character table of $G_{\boldsymbol{\lambda}}$, with $G_{\boldsymbol{\lambda}}$ isomorphic to $\mathbb{Z}_p$.}
\label{chaZ_p1}
\end{table}

\begin{defn}
For any $\boldsymbol{\lambda}\in K_n$ such that $G_{\boldsymbol{\lambda}}$ is isomorphic to the cyclic group $\mathbb{Z}_p$ and $\boldsymbol{\lambda}$ is of type mC or 2+mC, assume that $X_{\boldsymbol{\lambda}}=a_1X^{(1)}_{\boldsymbol{\lambda}}\oplus\cdots\oplus a_{p}X^{(p)}_{\boldsymbol{\lambda}}$. Then call $(a_1,\:a_2,\:\cdots,\:a_p)$ the \textbf{multiplicity vector} of $\boldsymbol{\lambda}$.
\end{defn}

\begin{rem}
The definition seems ambiguous as $\rho$ is chosen randomly. But we shall prove that $(a_1,\:a_2,\:\cdots,\:a_p)$ remains the same however we choose $\rho$, which makes the notion multiplicity vector well-defined.
\end{rem}

From the theorems in Section \ref{cha} we have already known the character of $X_{\boldsymbol{\lambda}}$. The results are shown in Table \ref{chaZp1}.
\begin{table}[h]
\centering
$$
\begin{array}{c|ccccccc}
& K^{(1)}_{\boldsymbol{\lambda}}
& K^{(2)}_{\boldsymbol{\lambda}} & \cdots
& K^{(k)}_{\boldsymbol{\lambda}} & \cdots
& K^{(p-1)}_{\boldsymbol{\lambda}}
& K^{(p)}_{\boldsymbol{\lambda}} \\
\hline
mC
& -1-2\cos{\frac{2\pi}{p}}  & -1-2\cos{\frac{4\pi}{p}} & \cdots
& -1-2\cos{\frac{2k\pi}{p}} & \cdots
& -1-2\cos{\frac{2(p-1)\pi}{p}} & mp-3 \\
2+mC & -1 & -1 & \cdots & -1 & \cdots & -1 & mp-1 \\
\end{array}
$$
\caption{Character of $X_{\boldsymbol{\lambda}}$, with $G_{\boldsymbol{\lambda}}$ isomorphic to $\mathbb{Z}_p$.}
\label{chaZp1}
\end{table}

Thus we have the conclusion that the multiplicity vector of $\boldsymbol{\lambda}$ is
\begin{enumerate}
\item \textbf{mC}:
\begin{itemize}
\item $(m-2,\:m-1)$ for $p=2$;
\item $(m-1,\:m-1,\:m-1)$ for $p=3$;
\item $(m-1,\:m,\:\cdots,\:m,\:m-1,\:m-1)$ for $p\ge4$;
\end{itemize}
\item \textbf{2+mC}: $(m,\:\cdots,\:m,\:m-1)$.
\end{enumerate}

%----------------------------------------------------------------
\subsubsection{Representations of $\mathbb{Z}_p$ when
$\boldsymbol{\lambda}$ is of type 1+mC}

When $\boldsymbol{\lambda}$ is of type 1+mC, suppose that $\sigma(a)=a$ for each $g_\sigma\in G_{\boldsymbol{\lambda}}$. Let $\psi$ be any linear fractional transformation such that
$$
\psi(z^{\boldsymbol{\lambda}}_a)=0
$$
and
$$
\psi\circ\Phi_{\boldsymbol{\lambda}}(G_{\boldsymbol{\lambda}})\circ\psi^{-1}
=\langle z\mapsto wz\rangle\simeq \mathbb{Z}_p.
$$

Let $\rho$ denote the element in $G_{\boldsymbol{\lambda}}$ such that
$$
\psi\circ \Phi_{\boldsymbol{\lambda}}(\rho)\circ\psi^{-1}(z)=wz.
$$
Set
$$
K^{(k)}_{\boldsymbol{\lambda}}=\{\rho^k\}
$$
for $k=1,\:2,\:\cdots,\:p$, and we have got the $p$ conjugate classes of $G_{\boldsymbol{\lambda}}$.

Now fix $K^{(1)}_{\boldsymbol{\lambda}},\:K^{(2)}_{\boldsymbol{\lambda}},\cdots, K^{(p)}_{\boldsymbol{\lambda}}$. Table \ref{chaZ_p2} is the character table of $G_{\boldsymbol{\lambda}}$, with $X^{(1)}_{\boldsymbol{\lambda}},\: X^{(2)}_{\boldsymbol{\lambda}},\:\cdots,\:X^{(p)}_{\boldsymbol{\lambda}}$ representing the $p$ different irreducible representations of $G_{\boldsymbol{\lambda}}$ and $\chi^{(1)}_{\boldsymbol{\lambda}},\: \chi^{(2)}_{\boldsymbol{\lambda}},\:\cdots,\:\chi^{(p)}_{\boldsymbol{\lambda}}$ their characters.
\begin{table}[h]
\centering
$$
\begin{array}{c|ccccccc}
& K^{(1)}_{\boldsymbol{\lambda}}
& K^{(2)}_{\boldsymbol{\lambda}} & \cdots
& K^{(k)}_{\boldsymbol{\lambda}} & \cdots
& K^{(p)}_{\boldsymbol{\lambda}} \\
\hline
\chi^{(1)}_{\boldsymbol{\lambda}}
       & w      & w^2    & \cdots & w^k    & \cdots & 1 \\
\chi^{(2)}_{\boldsymbol{\lambda}}
       & w^2    & w^4    & \cdots & w^{2k} & \cdots & 1 \\
\vdots & \vdots & \vdots &        & \vdots &        & \vdots \\
\chi^{(l)}_{\boldsymbol{\lambda}}
       & w^l    & w^{2l} & \cdots & w^{kl} & \cdots & 1 \\
\vdots & \vdots & \vdots &        & \vdots &        & \vdots \\
\chi^{(p)}_{\boldsymbol{\lambda}}
       & 1      & 1      & \cdots & 1      & \cdots & 1 \\
\end{array}
$$
\caption{Character table of $G_{\boldsymbol{\lambda}}$, with $G_{\boldsymbol{\lambda}}$ isomorphic to $\mathbb{Z}_p$.}
\label{chaZ_p2}
\end{table}

\begin{defn}
For any $\boldsymbol{\lambda}\in K_n$ such that $G_{\boldsymbol{\lambda}}$ is isomorphic to the cyclic group $\mathbb{Z}_p$ and $\boldsymbol{\lambda}$ is of type 1+mC, assume that $X_{\boldsymbol{\lambda}}=a_1X^{(1)}_{\boldsymbol{\lambda}}\oplus\cdots\oplus a_{p}X^{(p)}_{\boldsymbol{\lambda}}$. Then call $(a_1,\:a_2,\:\cdots,\:a_p)$ the \textbf{multiplicity vector} of $\boldsymbol{\lambda}$.
\end{defn}

\begin{rem}
The definition seems ambiguous as $\psi$ is chosen randomly. But we shall prove that $(a_1,\:a_2,\:\cdots,\:a_p)$ remains the same however we choose $\psi$, which makes the notion multiplicity vector well-defined.
\end{rem}

From the theorems in Section \ref{cha} we have already known the character of $X_{\boldsymbol{\lambda}}$. The results are shown in Table \ref{chaZp2}.
\begin{table}[h]
\centering
$$
\begin{array}{c|ccccccc}
& K^{(1)}_{\boldsymbol{\lambda}}
& K^{(2)}_{\boldsymbol{\lambda}} & \cdots
& K^{(k)}_{\boldsymbol{\lambda}} & \cdots
& K^{(p-1)}_{\boldsymbol{\lambda}}
& K^{(p)}_{\boldsymbol{\lambda}} \\
\hline
1+mC
& -1-w^{-1}  & -1-w^{-2} & \cdots & -1-w^{-k} & \cdots
& -1-w^{-(p-1)} & mp-2 \\
\end{array}
$$
\caption{Character of $X_{\boldsymbol{\lambda}}$, with $G_{\boldsymbol{\lambda}}$ isomorphic to $\mathbb{Z}_p$.}
\label{chaZp2}
\end{table}

Thus we have the conclusion that the multiplicity vector of $\boldsymbol{\lambda}$ is
\begin{enumerate}
\item \textbf{1+mC}:
\begin{itemize}
\item $(m-1,\:m-1)$ for $p=2$;
\item $(m,\:\cdots,\:m,\:m-1,\:m-1)$ for $p\ge3$.
\end{itemize}
\end{enumerate}

\begin{thm}
For any $\boldsymbol{\lambda}\in K_n$ such that $G_{\boldsymbol{\lambda}}$ is isomorphic to the cyclic group $\mathbb{Z}_p$, we have found its multiplicity vector
\begin{enumerate}
\item \textbf{mC}:
\begin{itemize}
\item $(m-2,\:m-1)$ for $p=2$;
\item $(m-1,\:m-1,\:m-1)$ for $p=3$;
\item $(m-1,\:m,\:\cdots,\:m,\:m-1,\:m-1)$ for $p\ge4$;
\end{itemize}
\item \textbf{1+mC}:
\begin{itemize}
\item $(m-1,\:m-1)$ for $p=2$;
\item $(m,\:\cdots,\:m,\:m-1,\:m-1)$ for $p\ge3$.
\end{itemize}
\item \textbf{2+mC}: $(m,\:\cdots,\:m,\:m-1)$.
\end{enumerate}
\end{thm}

%----------------------------------------------------------------
\section{The Group that Fixes Four Points}
\label{n=4}

Let $\alpha=\{z_1, z_2, z_3, z_4\} \subseteq \widehat {\mathbb C}$, we define $\lambda$ to be the cross-ratio $[{z}_1, {z}_2, {z}_3, {z}_4]$.

Define a function $\pi: {S}_4 \to \mathbb C \;\; \sigma \mapsto [{z}_{\sigma(1)}, {z}_{\sigma(2)}, {z}_{\sigma(3)}, {z}_{\sigma(4)}$].

Then we have

\begin{displaymath}
\pi((1,2))=[{z}_2, {z}_1, {z}_3, {z}_4]=\frac{({z}_2-{z}_3)({z}_1-{z}_4)}{({z}_2-{z}_1)({z}_3-{z}_4)}=1-\lambda;
\end{displaymath}

\begin{displaymath}
\pi((1,3))=[{z}_3, {z}_2, {z}_1, {z}_4]=\frac{({z}_3-{z}_1)({z}_2-{z}_4)}{({z}_3-{z}_2)({z}_1-{z}_4)}=\frac{\lambda}{\lambda-1};
\end{displaymath}

\begin{displaymath}
\pi((1,4))=[{z}_4, {z}_2, {z}_3, {z}_1]=\frac{({z}_4-{z}_3)({z}_2-{z}_1)}{({z}_4-{z}_2)({z}_3-{z}_1)}=\frac{1}{\lambda}.
\end{displaymath}

By calculation we have

\begin{displaymath}
\pi((1,2))=\pi((3,4))=\pi((1,3,2,4))=\pi((1,4,2,3))=1-\lambda;
\end{displaymath}

\begin{displaymath}
\pi((1,3))=\pi((2,4))=\pi((1,2,3,4))=\pi((1,4,3,2))=\frac{\lambda}{\lambda-1};
\end{displaymath}

\begin{displaymath}
\pi((1,4))=\pi((2,3))=\pi((1,2,4,3))=\pi((1,3,4,2))=\frac{1}{\lambda};
\end{displaymath}

\begin{displaymath}
\pi((1,2,3))=\pi((2,4,3))=\pi((3,4,1))=\pi((4,2,1))=\frac{\lambda-1}{\lambda};
\end{displaymath}

\begin{displaymath}
\pi((1,3,2))=\pi((2,3,4))=\pi((3,1,4))=\pi((4,1,2))=\frac{1}{1-\lambda};
\end{displaymath}

\begin{displaymath}
\pi((1,2)(3,4))=\pi((1,3)(2,4))=\pi((1,4)(3,2))=\lambda.
\end{displaymath}

On one hand, for any $f\in\mathcal{A_\alpha}$, we have
\begin{displaymath}
\pi(\sigma_f)=[z_{\sigma_f(1)}, z_{\sigma_f(2)}, z_{\sigma_f(3)}, z_{\sigma_f(4)}]=[f(z_1), f(z_2), f(z_3), f(z_4)]=[z_1, z_2, z_3, z_4]=\lambda.
\end{displaymath}
On the other hand, if there exists some $\sigma \in {S}_4$ s.t.~
\begin{displaymath}
\pi(\sigma)=[z_{\sigma(1)}, z_{\sigma(2)}, z_{\sigma(3)}, z_{\sigma(4)}]=[z_1, z_2, z_3, z_4]=\lambda,
\end{displaymath}
there must be some $f\in\mathcal{A_\alpha}$ s.t.~$f(z_i)=z_{{\sigma}(i)}, i=1, 2, 3, 4$.

Now, all we have to do is to find $\sigma \in {S}_4 \textrm{ s.t. } \lambda = \pi (\sigma)$.

If $\lambda=1/2$, we have $1-\lambda=\lambda, \frac{\lambda}{\lambda-1}\ne\lambda, \frac{1}{\lambda}\ne\lambda, \frac{\lambda-1}{\lambda}\ne\lambda, \frac{1}{1-\lambda}\ne\lambda$.

\begin{displaymath}
\{\sigma_f\in S_4|f\in\mathcal{A_\alpha}\}=\{ e,(1,2),(3,4),(1,2)(3,4),(1,3)(2,4),(1,4)(2,3),(1,3,2,4),(1,4,2,3)\}\cong{D}_4.
\end{displaymath}

If $\lambda=2$, we have $\frac{\lambda}{\lambda-1}=\lambda, 1-\lambda\ne\lambda, \frac{1}{\lambda}\ne\lambda, \frac{\lambda-1}{\lambda}\ne\lambda, \frac{1}{1-\lambda}\ne\lambda$.

\begin{displaymath}
\{\sigma_f\in S_4|f\in\mathcal{A_\alpha}\}=\{e,(1,3),(2,4),(1,2)(3,4),(1,3)(2,4),(1,4)(2,3),(1,2,3,4),(1,4,3,2)\}\cong {D}_4
\end{displaymath}

If $\lambda=-1$, we have $\frac{1}{\lambda}=\lambda, 1-\lambda\ne\lambda, \frac{\lambda}{\lambda-1}\ne\lambda, \frac{\lambda-1}{\lambda}\ne\lambda, \frac{1}{1-\lambda}\ne\lambda$.

\begin{displaymath}
\{\sigma_f\in S_4|f\in\mathcal{A_\alpha}\}=\{e,(1,4),(2,3),(1,2)(3,4),(1,3)(2,4),(1,4)(2,3),(1,2,4,3),(1,3,4,2)\}\cong {D}_4
\end{displaymath}

If $\lambda = \frac{1\pm \sqrt 3 i}{2}$, we have $ \frac{\lambda-1}{\lambda}=\frac{1}{1-\lambda}=\lambda, 1-\lambda\ne\lambda, \frac{\lambda}{\lambda-1}\ne\lambda, \frac{1}{\lambda}\ne\lambda$.

\begin{eqnarray*}
\{\sigma_f\in S_4|f\in\mathcal{A_\alpha}\} &=&
\{e,(1,2)(3,4),(1,3)(2,4),(1,4)(2,3),(1,2,3),(1,3,2),(2,3,4),(2,4,3),\\
&& \:\:\:(3,4,1),(3,1,4),(4,1,2),(4,2,1)\}={A}_4.
\end{eqnarray*}

If $\lambda \ne 1/2, 2, -1, \frac{1\pm \sqrt 3 i}{2}$, now $1-\lambda\ne\lambda, \frac{\lambda}{\lambda-1}\ne\lambda, \frac{1}{\lambda}\ne\lambda, \frac{\lambda-1}{\lambda}\ne\lambda, \frac{1}{1-\lambda}\ne\lambda$.

\begin{displaymath}
\{\sigma_f\in S_4|f\in\mathcal{A_\alpha}\}=\{e,(1,2)(3,4),(1,3)(2,4),(1,4)(2,3)\}\cong K_4.
\end{displaymath}

%----------------------------------------------------------------
\section{The Group that Fixes Five Points}
\label{n=5}

Let ${ z }_{ 1 },{ z }_{ 2 },{ z }_{ 3 },{ z }_{ 4 }$ be $0, 1, \lambda, \infty $, respectively, $\lambda \in \mathbb{C} \backslash \{0, 1\}$. For $\sigma \in S_4$, let $f_\sigma$ denote the linear fractional transformation which maps $z_i$ to $z_{\sigma (i)}, i=1, 2, 3, 4$ (if $f_\sigma$ exists). From the previous section we know that ${ f }_{ (1,2)(3,4) },{ f }_{ (1,3)(2,4) },{ f }_{ (1,4)(2,3) }$ always exist regardless of the choice of $\lambda$.

By calculation we know that
\begin{displaymath}
f_{(1,2)(3,4)}(z)=\frac{\lambda z - \lambda}{z-\lambda},{ f }_{ (1,3)(2,4) }(z)=\frac{z-\lambda}{z-1},{ f }_{ (1,4)(2,3)}(z)=\frac{\lambda}{z},
\end{displaymath}
which fix $\lambda \pm \sqrt { { \lambda  }^{ 2 }-\lambda  }, 1\pm \sqrt { 1-\lambda  } ,\pm \sqrt { \lambda  } $ respectively.

When $n=5$, $\alpha=\{z_1, z_2, z_3, z_4, z_5\}\subseteq\widehat{\mathbb{C}}$. Since any non-trivial $f\in\mathcal{A_\alpha}$ is elliptic, $\sigma_f$ must be of Type $(5),(4,1),(3,1,1)$ or $(2,2,1)$. There are five possibilities:
\begin{enumerate}
\item There exists some $f\in\mathcal{A_\alpha}$ s.t.~$\sigma_f$ is of Type $(5)$;
\item For any $f\in\mathcal{A_\alpha}, \sigma_f$ is not of Type $(5)$, but there exists some $f\in\mathcal{A_\alpha}$ s.t.~$\sigma_f$ is of Type $(4, 1)$;
\item For any $f\in\mathcal{A_\alpha}, \sigma_f$ is not of Type $(5)$ or $(4,1)$, but there exists some $f\in\mathcal{A_\alpha}$ s.t.~$\sigma_f$ is of Type $(3,1,1)$;
\item For any $f\in\mathcal{A_\alpha}, \sigma_f$ is not of Type $(5),(4,1)$ or $(3,1,1)$, but there exists some $f\in\mathcal{A_\alpha}$ s.t.~$\sigma_f$ is of Type $(2,2,1)$;
\item $\mathcal{A_\alpha}$ is the trivial group.
\end{enumerate}

In the rest of this section we shall discuss the five possibilities one by one, and prove the following theorem
\begin{thm}
For $\alpha=\{z_1,\: z_2, \:z_3,\: z_4, \:z_5\}\subseteq\widehat{\mathbb{C}}$, $\mathcal{A_\alpha}$ is isomorphic to $D_5$, $\mathbb{Z}_4$, $D_3$, $\mathbb{Z}_2$ or the trivial group $\{\mathrm{Id}\}$.
\end{thm}

%----------------------------------------------------------------
\subsection{$\mathcal{A_\alpha}$ is Isomorphic to $D_5$}
\label{n=5, D_5}

In the first case assume that there exists some $f\in\mathcal{A_\alpha}$ s.t.~$\sigma_f$ is of Type $(5)$. This assumption amounts to the existence of some linear fractional transformation $\psi$ s.t.
\begin{displaymath}
\psi(\alpha)=\{1, w, w^2, w^3, w^4\},\: w=e^{\frac{2\pi}{5}i}.
\end{displaymath}

Without lose of generality we assume that
\begin{displaymath}
\alpha=\left\{ 1,w,{ w }^{ 2 },{ w }^{ 3 },{ w }^{ 4 } \right\},\: f(z)=e^{\frac{2\pi}{5}i}z.
\end{displaymath}

Define another linear fractional transformation $g$
\begin{displaymath}
g(z)=\frac{1}{z}.
\end{displaymath}
It is obvious that $f,g \in \mathcal{A_\alpha}$. Note that $\langle f,g\rangle\simeq D_5$ acts transitively on $\alpha$.

For any $k \in \mathcal{A_\alpha}$, there exists some $l\in  \langle f,g\rangle$ s.t.~$h=l\circ k$ fixes $w^4$. We have
\begin{displaymath}
h(\{1, w, w^2, w^3\})=\{1, w, w^2, w^3\},\: h(w^4)=w^4.
\end{displaymath}

Let $\phi$ be the linear fractional transformation which maps $1, w, w^3$ to $0, 1, \infty$ respectively. Define $\lambda$ to be the image of $w^2$, so we have
\begin{displaymath}
\lambda=\phi(w^2)=[0,1,\phi(w^2),\infty]=[1, w, w^2, w^3]=w^4+w+2,
\end{displaymath}
and
\begin{displaymath}
\phi(w^4)=[0,1,\phi(w^4),\infty]=[1, w, w^4, w^3]=w^3+w^2.
\end{displaymath}
Thus
\begin{displaymath}
\phi \circ h \circ \phi^{-1}(\{0,1,\lambda,\infty\})=\{0,1,\lambda,\infty\},\: \phi \circ h \circ \phi^{-1}(w^3+w^2)=w^3+w^2.
\end{displaymath}
As a result, $\phi \circ h \circ \phi^{-1} \in \mathcal{A}_{\{0, 1, \lambda, \infty\}}$.

As $[0, 1, \lambda, \infty]=\lambda=w^4+w+2=\frac{\sqrt 5 +3}{2} \ne 2, \frac{1}{2}, -1, \frac{1\pm \sqrt3 i}{2}$, we know from \ref{n=4} that $\mathcal{A}_{\{0, 1, \lambda, \infty\}}=\{I, { f }_{ (1,2)(3,4) },{ f }_{ (1,3)(2,4) },{ f }_{ (1,4)(2,3) }\}.$

However $f_{(1,2)(3,4)}, f_{(1,3)(2,4)}, f_{(1,4)(2,3)}$ fix $\lambda \pm \sqrt { { \lambda  }^{ 2 }-\lambda  }, 1\pm \sqrt { 1-\lambda  } ,\pm \sqrt { \lambda  }$ respectively, and
\begin{displaymath}
w^3+w^2 \ne \lambda \pm \sqrt { { \lambda  }^{ 2 }-\lambda  }, 1\pm \sqrt { 1-\lambda  } ,\pm \sqrt { \lambda  }.
\end{displaymath}
So we have $\phi \circ h \circ \phi^{-1} = I$, and thus $k = l^{-1}\in \langle f,g\rangle$.

In conclusion we have
\begin{displaymath}
\mathcal{A_\alpha} = \langle z \mapsto e^{\frac{2\pi}{5}i}z,z \mapsto \frac{1}{z}\rangle \simeq D_5.
\end{displaymath}

%----------------------------------------------------------------
\subsection{$\mathcal{A_\alpha}$ is Isomorphic to $\mathbb{Z}_4$}

In this case assume that for any $h\in\mathcal{A_\alpha}$, $\sigma_h$ is not of Type $(5)$, but there exists some $f\in\mathcal{A_\alpha}$ s.t.~$\sigma_f$ is of Type $(4, 1)$. Under this assumption, there exists some linear fractional transformation $\psi$ s.t.
\begin{displaymath}
\psi(\alpha) =\{0, 1, i, -1, -i\}.
\end{displaymath}

On the other hand, for any $\alpha=\{z_1, z_2, z_3, z_4, z_5\}\subseteq \widehat{\mathbb C}$, if there exists some linear fractional transformation $\psi$ s.t.~$\psi(\alpha)=\{0, 1, i, -1, -i\}$, we see that ${\psi}^{-1} \circ f \circ \psi$ fixes $\alpha$ and is of Type $(4,1)$, where $f(z)=iz$. However, since the five points in $\alpha$ are not concyclic, then for any $h\in\mathcal{A_\alpha}$, $\sigma_h$ is not of Type $(5)$. Thus we see that the assumption amounts to the existence of some linear fractional transformation $\psi$ s.t.
\begin{displaymath}
\alpha=\{0, 1, i, -1, -i\},\: f(z)=iz.
\end{displaymath}
It is obvious that $f \in\mathcal{A_\alpha}$.

For any $g \in \mathcal{A_\alpha}$, it is easy to see that
\begin{displaymath}
|g(\{\pm 1, \pm i\})\cap\{\pm 1, \pm i\}|\ge3.
\end{displaymath}
So $g$ fixes the unit circle, and thus $g(0)=0$.

The linear fractional transformation
\begin{displaymath}
\phi(z)=\frac{(1-i)z+i-1}{z+i}
\end{displaymath}
takes $1, i, -i$ and $0$ to $0, 1, \infty$ and $1+i$ respectively. Define $\lambda$ to be the image of $-1$, so we have
\begin{displaymath}
\lambda=\phi(-1)=[0,1,\phi(-1),\infty]=[1, i, -1, -i]=2.
\end{displaymath}

As $g$ fixes $0$, we see that
\begin{displaymath}
\phi \circ g \circ \phi^{-1}(\{0,1,\lambda,\infty\})=\{0,1,\lambda,\infty\}, \phi \circ g \circ \phi^{-1}(1+i)=1+i.
\end{displaymath}

As $[0, 1, \lambda, \infty]=\lambda=2$, from \ref{n=4} we know that
\begin{displaymath}
\mathcal{A}_{\{0, 1, \lambda, \infty\}}=\{I, f_{(1,2,3,4)}, f_{(1,3)(2,4)}, f_{(1,4,3,2)}, f_{(1,3)}, f_{(2,4)}, f_{(1,2)(3,4)}, f_{(1,4)(2,3)}\}.
\end{displaymath}
But $f_{(1,2)(3,4)}, f_{(1,4)(2,3)}$ fix $\lambda\pm\sqrt{{\lambda}^2-\lambda}=2\pm\sqrt2$ and $ \pm\sqrt{\lambda}=\pm\sqrt2$ respectively, and $f_{(1,3)}, f_{(2,4)}$ fixes $\pm i$ and $\pm 1$ respectively. So none of them fixes $1+i$. Thus we have
\begin{displaymath}
\phi \circ g \circ \phi^{-1}=f_{(1,2,3,4)},
\end{displaymath}
or
\begin{displaymath}
\phi \circ g \circ \phi^{-1}=f_{(1,3)(2,4)},
\end{displaymath}
or
\begin{displaymath}
\phi \circ g \circ \phi^{-1}=f_{(1,4,3,2)},
\end{displaymath}
or
\begin{displaymath}
\phi \circ g \circ \phi^{-1}=I.
\end{displaymath}
However
\begin{displaymath}
\phi^{-1} \circ f_{(1,2,3,4)}\circ\phi=f,2: \phi^{-1} \circ f_{(1,3)(2,4)}\circ\phi=f^2,\: \phi^{-1} \circ f_{(1,4,3,2)}\circ\phi=f^3,\: \phi^{-1} \circ I\circ\phi=I.
\end{displaymath}
So we have $g\in \langle f \rangle$.

In conclusion we have
\begin{displaymath}
\mathcal{A_\alpha} = \langle z \mapsto iz\rangle \simeq \mathbb Z _4.
\end{displaymath}

%----------------------------------------------------------------
\subsection{$\mathcal{A_\alpha}$ is Isomorphic to $D_3$}

In this case assume that for any $h\in\mathcal{A_\alpha}$, $\sigma_h$ is not of Type $(5)$ or $(4, 1)$, but there exists some $f\in\mathcal{A_\alpha}$ s.t.~$\sigma_f$ is of Type $(3, 1, 1)$. Under this assumption, there exists some linear fractional transformation $\psi$ s.t.
\begin{displaymath}
\psi(\alpha)=\{0, \infty, 1, w, w^2\},\: w=e^{\frac{2\pi}{3}i}.
\end{displaymath}

On the other hand, for any $\alpha=\{z_1, z_2, z_3, z_4, z_5\}\subseteq \widehat{\mathbb C}$, if there exists some linear fractional transformation $\psi$ s.t.~$\psi(\alpha)=\{0, \infty, 1, w, w^2\}$, we see that ${\psi}^{-1} \circ f \circ \psi$ fixes $\alpha$ and is of Type $(3, 1, 1)$, where $f(z)=wz$. However, since no four points are concyclic, then for any $h\in\mathcal{A_\alpha}$, $\sigma_h$ is not of Type $(5)$ or $(4, 1)$. Thus we see that the assumption amounts to the existence of some linear fractional transformation $\psi$ s.t.
\begin{displaymath}
\psi(\alpha)=\{0, \infty, 1, w, w^2\},\: w=e^{\frac{2\pi}{3}i}.
\end{displaymath}

Without lose of generality assume that
\begin{displaymath}
\alpha=\{0, \infty, 1, w, w^2\},\: f(z)=wz.
\end{displaymath}

Define another linear fractional transformation $g$
\begin{displaymath}
g(z)=\frac{1}{z}.
\end{displaymath}
It is obvious that $f,g \in \mathcal{A_\alpha}$. Note that there are two orbits of $\langle f,g\rangle$: $\{0, \infty\}$, and $\{1, w, w^2\}$.

For any $k\in\mathcal{A_\alpha}$, it is easy to see that
\begin{displaymath}
|k(\{1, w, w^2\})\cap\{1, w, w^2\}|\ge1.
\end{displaymath}
So there exists some $i, j\in \mathbb{Z}$ s.t.~$k(w^i)=w^j$. Let $h=f^{1-j}\circ k\circ f^{i-1}$, we see that $h(w)=w$. Thus
\begin{displaymath}
 h (\{0,1,w^2,\infty\})=\{0,1,w^2,\infty\},\: h(w)=w.
\end{displaymath}
As a result, $h \in \mathcal{A}_{\{0, 1, w^2, \infty\}}$.

Set $\lambda=w^2$. As $[0, 1, \lambda, \infty]=\lambda=w^2\ne 2, \frac{1}{2}, -1, \frac{1\pm \sqrt3 i}{2}$, we know from \ref{n=4} that $\mathcal{A}_{\{0, 1, \lambda, \infty\}}=\{I, { f }_{ (1,2)(3,4) },{ f }_{ (1,3)(2,4) },{ f }_{ (1,4)(2,3) }\}.$

But $f_{(1,2)(3,4)}, f_{(1,3)(2,4)}$ fix $\lambda \pm \sqrt { { \lambda  }^{ 2 }-\lambda  }, 1\pm \sqrt{1-\lambda}$ respectively, and
\begin{displaymath}
w \ne \lambda \pm \sqrt { { \lambda  }^{ 2 }-\lambda  },\: 1\pm \sqrt { 1-\lambda  }.
\end{displaymath}
So we have $h=f_{(1,4)(2,3)}$ or $h=I$.

However
\begin{displaymath}
f_{(1,4)(2,3)}=f\circ g\in\langle f,g\rangle,\: \textrm{I}\in\langle f,g\rangle,
\end{displaymath}
thus we conclude that $k=f^{1-i}\circ h\circ f^{j-1}\in\langle f,g\rangle$

In conclusion we have
\begin{displaymath}
\mathcal{A_\alpha} = \langle z \mapsto e^{\frac{2\pi}{3}i}z,z \mapsto \frac{1}{z}\rangle \simeq D_3.
\end{displaymath}

%----------------------------------------------------------------
\subsection{$\mathcal{A_\alpha}$ is Isomorphic to $\mathbb{Z}_2$}

In this case assume that for any $h\in\mathcal{A_\alpha}$, $\sigma_h$ is not of Type $(5),(4,1)$ or $(3,1,1)$, but there exists some $f\in\mathcal{A_\alpha}$ s.t.~$\sigma_f$ is of Type $(2, 2, 1)$. Under this assumption, there exists some linear fractional transformation $\psi$ s.t.
\begin{displaymath}
\psi(\alpha)=\{0, 1, -1, a, -a\},\: a\ne0, \pm 1.
\end{displaymath}

On the other hand, for any $\{0, 1, -1, a, -a\}, a \in \mathbb{C} \backslash \{ 0, \pm 1 \}$, the linear fractional transformation $z\mapsto -z$ is in $\mathcal{A}_{\{0, 1, -1, a, -a\}}$ and of Type $(2, 2, 1)$. Now we aim to find out the specific value $a$ takes when for each element $h\in \mathcal{A}_{\{0, 1, -1, a, -a\}}$, $\sigma_h$ is not of Type $(5), (4, 1)$ or $(3, 1, 1)$.

\emph{Case 1}: There exists some $h\in\mathcal{A}_{\{0, 1, -1, a, -a\}}$ s.t.~$\sigma_h$ is of Type (5).

This amounts to the existence of some linear fractional transformation $\psi$ s.t.
\begin{displaymath}
\psi(\{0, 1, -1, a, -a\})=\{1, w, w^2, w^3, w^4\}, \:w=e^{\frac{2\pi}{5}i}.
\end{displaymath}
We conclude that $0, 1, -1, a, -a$ lie on the same circle: the real axis, and $a\in\mathbb{R}$. Without lose of generality assume that $r=a>0$.

If $r>1$, we may assume that $ \psi(-r)=w^4, \psi(-1)=1, \psi(0)=w, \psi(1)=w^2, \psi(r) =w^3$. We have
\begin{displaymath}
[1,w,w^2,w^3]=[-1,0,1,r],
\end{displaymath}
and thus
\begin{displaymath}
r=2w^4+2w+3=\sqrt{5}+2.
\end{displaymath}

If $0<r<1$, we assume that $\psi(-1)=w^4, \psi(-r)=1, \psi(0)=w, \psi(r)=w^2, \psi(1)=w^3$. By exactly the same means we conclude that
\begin{displaymath}
r=2w^4+2w-1=\sqrt{5}-2.
\end{displaymath}

On the other hand, the linear fractional transformation
\begin{displaymath}
\psi_1(z)=\frac{(w^2-w)z+w^2+w}{(1-w)z+w+1}
\end{displaymath}
maps $-\sqrt 5 -2, -1, 0, 1$ and $\sqrt 5+2$ to $w^4, 1, w, w^2$ and $w^3$ respectively, and
\begin{displaymath}
\psi_2(z)=\frac{(w^2+3w+1)z+w^4-w^3}{(w^3+w^2-2)z+w^3-w^2}
\end{displaymath}
maps $-1,-\sqrt 5 +2, 0, \sqrt 5-2$ and $1$ to $w^4, 1, w, w^2$ and $w^3$ respectively.

And we conclude that there exists some $h\in\mathcal{A}_{\{0, 1, -1, a, -a\}}$ s.t.~$\sigma_h$ is of Type (6) if and only if $z=\pm\sqrt{5}\pm2$.

\emph{Case 2}: There exists some $h\in\mathcal{A}_{\{0, 1, -1, a, -a\}}$ s.t.~$\sigma_h$ is of Type (4, 1).

This amounts to the existence of some linear fractional transformation $\psi$ s.t.
\begin{displaymath}
\psi(\{0, 1, -1, a, -a\})=\{0, 1, i, -1, -i\}.
\end{displaymath}

It is obvious to see that such a $\psi$ exists if and only if $a=\pm i$.

\emph{Case 3}: There exists some $h\in\mathcal{A}_{\{0, 1, -1, a, -a\}}$ s.t.~$\sigma_h$ is of Type (3, 1, 1).

This amounts to the existence of some linear fractional transformation $\psi$ s.t.
\begin{displaymath}
\psi(\{0, 1, -1, a, -a\})=\{0, \infty, 1, w, w^2\}, \:w=e^{\frac{2\pi}{3}i}.
\end{displaymath}

The situation now is a little complex. We can see that the linear fractional transformation $f(z)=-z$ leaves $\{0, 1, -1, a, -a\}$ invariant, is of order 2, and fixes the point $0$. However, from the above section we know that $\mathcal{A}_{\psi(\{0, 1, -1, a, -a\})}=\langle z\mapsto wz, z\mapsto\frac{1}{z}\rangle \simeq D_3 $. The only three linear fractional transformations of order 2 in $\mathcal{A}_{\psi(\{0, 1, -1, a, -a\})}$ are
\begin{displaymath}
z\mapsto\frac{1}{z}, z\mapsto\frac{w}{z}, z\mapsto\frac{w^2}{z},
\end{displaymath}
which fixes $1, w^2, w$ respectively. So $\psi(0)\ne0,\infty$.

Without lose of generality we assume that $\psi(0)=1$. The only element of order 2 fixing $1$
in $\mathcal{A}_{\psi(\{0, 1, -1, a, -a\})}$ is $h(z)=\frac{1}{z}$. However, the element $f(z)=-z$ is of order 2, fixes $0$, and is in $\mathcal{A}_{\{0, 1, -1, a, -a\}}$. So we have $h=\psi \circ f \circ {\psi}^{-1}$.

As $h$ fixes $\{0, \infty \}$ and $\{1, -1\}$, and $f$ fixes $\{1,-1\}$ and $\{a, -a\}$, we see that there are only two possibilities. The first is $\psi(\{1, -1\})=\{0, \infty \}$, $\psi(\{a, -a\})=\{w, w^2\}$, and the second is $\psi(\{1, -1\})=\{w, w^2\}$, $\psi(\{a, -a\})=\{0, \infty \}$.

In the former situation, assume that $\psi(0)=1, \psi(1)=\infty, \psi(-1)=0, \psi(a)=w, \psi(-a)=w^2$, and we have\begin{displaymath}
[0, 1, w, \infty]=[-1, 0, a, 1]
\end{displaymath}
and
\begin{displaymath}
a=\sqrt{3}i.
\end{displaymath}

In the latter situation assume that
$\psi(0)=1, \psi(1)=w, \psi(-1)=w^2, \psi(a)=0, \psi(-a)=\infty$, and we have
\begin{displaymath}
[0, 1, w, \infty]=[a, 0, 1, -a]
\end{displaymath}
and
\begin{displaymath}
a=\frac{i}{\sqrt{3}}.
\end{displaymath}

On the other hand, the linear fractional transformation
\begin{displaymath}
\psi_1(z)=\frac{1+z}{1-z}
\end{displaymath}
maps $0, 1, -1, \sqrt3 i, -\sqrt3 i$ to $1, \infty, 0, w, w^2$ respectively, and
\begin{displaymath}
\psi_2(z)=\frac{i-\sqrt{3}z}{i+\sqrt{3}z}
\end{displaymath}
maps $0, 1, -1, \frac{1}{\sqrt{3}}i, -\frac{1}{\sqrt{3}}i$ to $1, w, w^2, \infty, 0$  respectively.

And we conclude that there exists some $h\in\mathcal{A}_{\{0, 1, -1, a, -a\}}$ s.t.~$\sigma_h$ is of Type (3, 1, 1) if and only if $a=\pm\sqrt{3}i$ or $\pm\frac{1}{\sqrt{3}}i$.

From the above discussion we see that the assumption amounts to the existence of some linear fractional transformation $\psi$ s.t.
\begin{displaymath}
\psi(\alpha)=\{0, 1, -1, a, -a\},\: a \ne0, \pm 1, \pm\sqrt{5}\pm2, \pm i, \pm\sqrt{3}i\mathrm{~or}\pm\frac{i}{\sqrt{3}}.
\end{displaymath}

Without lose of generality assume
\begin{displaymath}
\alpha=\{0, 1, -1, a, -a\},\: a \ne0, \pm 1, \pm\sqrt{5}\pm2, \pm i, \pm\sqrt{3}i\mathrm{~or}\pm\frac{i}{\sqrt3},\: f(z)=-z.
\end{displaymath}

So $\sigma_f =(1)(2,3)(4,5)$.

If $\exists g\in\mathcal{A_\alpha}$ s.t.~$\sigma_g(1)\ne1$, without lose of generality assume that $\sigma_g(1)=2$. We have
\begin{displaymath}
\sigma_g=(1,2)(3,4)(5)=:\pi_1
\end{displaymath}
or
\begin{displaymath}
\sigma_g=(1,2)(4,5)(3)=:\pi_2
\end{displaymath}
or
\begin{displaymath}
\sigma_g=(1,2)(5,3)(4)=:\pi_3.
\end{displaymath}

But
\begin{displaymath}
\pi_1\sigma_f=(1,2,4,5,3)
\end{displaymath}
and
\begin{displaymath}
\pi_2\sigma_f=(1,2,3)
\end{displaymath}
and
\begin{displaymath}
\pi_3\sigma_f=(1,2,5,4,3).
\end{displaymath}
thus such a $g$ does not exist, and every element in $\mathcal{A_\alpha}$ fixes $0$.

If $\exists g \in {\mathcal{A_\alpha}}$ s.t.~$g \ne f, g \ne I$, assume that $g(1)=a, g(-1)=-a, g(a)=1, g(-a)=-1$. So
\begin{displaymath}
g(z)=\frac{a}{z}.
\end{displaymath}
However, $g(0)\ne0$. So such a $g$ does not exist, too.

In conclusion we have
\begin{displaymath}
\mathcal{A_\alpha} = \langle z \mapsto -z\rangle \simeq \mathbb Z _2.
\end{displaymath}

%----------------------------------------------------------------
\subsection{Conclusion}
\label{n=5, con}

From the above discussion we shall drive the following conclusion

\begin{thm}

Set $\alpha=\{z_1,\: z_2,\: z_3,\: z_4,\: z_5\}\subseteq\widehat {\mathbb{C}}$.

\begin{enumerate}

\item If there exists some linear fractional transformation $\psi$ such that
$$
\psi(\alpha)=\{1,\: w,\: w^2,\: w^3\:, w^4\},\:w=e^{\frac{2\pi}{5}i},
$$
then
$$
\mathcal{A_\alpha}={\psi}^{-1}\langle z \mapsto e^{\frac{2\pi}{5}i}z,\:z \mapsto \frac{1}{z}\rangle {\psi}\simeq D_5,
$$
and its multiplicity vector is $(0,\:1,\:0,\:0)$;

\item if there exists some linear fractional transformation $\psi$ such that
$$
\psi(\alpha) =\{0,\: 1,\: i,\: -1,\: -i\},
$$
then
$$
\mathcal{A_\alpha}={\psi}^{-1}\langle z \mapsto iz\rangle {\psi}\simeq \mathbb Z _4,
$$
and its multiplicity vector is $(1,\:1,\:0,\:0)$;

\item if there exists some linear fractional transformation $\psi$ such that
$$
\psi(\alpha)=\{0,\: \infty,\: 1, \:w,\: w^2\},\:w=e^{\frac{2\pi}{3}i},
$$
then
$$
\mathcal{A_\alpha}={\psi}^{-1}\langle z \mapsto e^{\frac{2\pi}{3}i}z,\:z \mapsto \frac{1}{z}\rangle\psi \simeq D_3,
$$
and its multiplicity vector is $(1,\:0,\:0)$;

\item if there exists some linear fractional transformation $\psi$ such that
$$
\psi(\alpha)=\{0,\: 1, \:-1,\: a,\: -a\},\:a\ne0,\:\pm 1,
$$
then

\begin{enumerate}

\item if $a=\pm(\sqrt{5}+2)\text{ or }\pm(\sqrt{5}-2)$, then there exists some linear fractional transformation $\phi$ such that
$$
\phi(\alpha)=\{1,\: w,\: w^2,\: w^3\:, w^4\},\:w=e^{\frac{2\pi}{5}i}
$$
and this is case 1;
\item if $a=\pm i$, then
$$
\alpha =\{0,\: 1,\: i,\: -1,\: -i\}
$$
and this is case 2;

\item if $a=\pm\sqrt{3}i\text{ or }\pm\frac{1}{\sqrt{3}}i$, then there exists some linear fractional transformation $\phi$ such that
$$
\phi(\alpha)=\{0,\: \infty,\: 1, \:w,\: w^2\},\:w=e^{\frac{2\pi}{3}i}
$$
and this is case 3;

\item otherwise,
$$
\mathcal{A_\alpha}={\psi}^{-1}\langle z \mapsto -z\rangle {\psi}\simeq \mathbb Z _2,
$$
and its multiplicity vector is $(1,\:1)$;

\end{enumerate}

\item otherwise,
$$\mathcal{A_\alpha}=\{\mathrm{Id}\}.$$

\end{enumerate}

\end{thm}

%----------------------------------------------------------------
\section{The Group that Fixes Six Points}
\label{n=6}

There is no much difference here from the previous section. When $n=6$, $\alpha=\{z_1, z_2, z_3, z_4, z_5, z_6\}\subseteq \widehat{\mathbb{C}}$. Since any non-trivial $f\in\mathcal{A}_\alpha$ is a rotation of finite order, $\sigma_f$ must be of Type (6), (5, 1), (4, 1, 1), (3, 3), (2, 2, 2), (2, 2, 1, 1) or identity. There are seven possibilities:
\begin{enumerate}
\item There exists some $f\in\mathcal{A_\alpha}$ s.t.~$\sigma_f$ is of Type (6);
\item For any $f\in\mathcal{A_\alpha}$, $\sigma_f$ is not of Type (6), but there exists some $f\in\mathcal{A_\alpha}$ s.t.~$\sigma_f$ is of Type (5, 1);
\item For any $f\in\mathcal{A_\alpha}$, $\sigma_f$ is not of Type (6) or (5, 1), but there exists some $f\in\mathcal{A_\alpha}$ s.t.~$\sigma_f$ is of Type (4, 1, 1);
\item For any $f\in\mathcal{A_\alpha}$, $\sigma_f$ is not of Type (6), (5, 1) or (4, 1, 1), but there exists some $f\in\mathcal{A_\alpha}$ s.t.~$\sigma_f$ is of Type (3, 3);
\item For any $f\in\mathcal{A_\alpha}$, $\sigma_f$ is not of Type (6), (5, 1), (4, 1, 1) or (3, 3), but there exists some $f\in\mathcal{A_\alpha}$ s.t.~$\sigma_f$ is of Type (2, 2, 2);
\item For any $f\in\mathcal{A_\alpha}$, $\sigma_f$ is not of Type (6), (5, 1), (4, 1, 1), (3, 3) or (2, 2, 2), but there exists some $f\in\mathcal{A_\alpha}$ s.t.~$\sigma_f$ is of Type (2, 2, 1, 1);
\item $\mathcal{A_\alpha}$ is the trivial group.
\end{enumerate}

In the rest of this section we shall discuss the five possibilities one by one, and prove the following theorem
\begin{thm}
For $\alpha=\{z_1, z_2, z_3, z_4, z_5, z_6\}\subseteq\widehat{\mathbb{C}}$, $\mathcal{A_\alpha}$ is isomorphic to $D_5$, $\mathbb{Z}_5$, $S_4$, $D_3$, $K_4$, $\mathbb{Z}_2$ or the trivial group $\{\mathrm{Id}\}$.
\end{thm}

%----------------------------------------------------------------
\subsection{$\mathcal{A_\alpha}$ is Isomorphic to $D_6$}
\label{n=6, D_6}

In the first case assume that there exists some $f\in\mathcal{A_\alpha}$ s.t.~$\sigma_f$ is of Type (6). This amounts to the existence of some linear fractional transformation $\psi$ s.t.
\begin{displaymath}
\psi(\alpha)=\{1, w, w^2, w^3, w^4, w^5\}, w=e^{\frac{\pi}{3}i}.
\end{displaymath}

Without lose of generality assume that
\begin{displaymath}
\alpha=\{1, w, w^2, w^3, w^4, w^5\},\: f(z)=wz.
\end{displaymath}

Define another linear fractional transformation $g$
\begin{displaymath}
g(z)=\frac{1}{z}.
\end{displaymath}
It is obvious that $f,g \in \mathcal{A_\alpha}$. Note that $\langle f,g\rangle\simeq D_6$ acts transitively on $\alpha$.

For any $k \in \mathcal{A_\alpha}$, there exists some $l\in  \langle f,g\rangle$ s.t.~$h=l\circ k$ fixes $-1$. We have
\begin{displaymath}
h(\{w^4, w^5, 1, w, w^2\})=\{w^4, w^5, 1, w, w^2\},\: h(-1)=-1.
\end{displaymath}
Thus $h\in\mathcal{A}_{\{w^4, w^5, 1, w, w^2\}}$.

Define another linear fractional transformation $\phi$
\begin{displaymath}
\phi(z)=\frac{(-w^2-w)z+w^2+w}{z+1}.
\end{displaymath}
Notice that
\begin{displaymath}
\phi(w^4)=-3,\: \phi(w^5)=-1,\:  \phi(1)=0,\:  \phi(w)=1, \: \phi(w^2)=3,\:  \phi(-1)=\infty.
\end{displaymath}
Thus we have
\begin{displaymath}
\phi\circ h\circ\phi^{-1}\in\mathcal{A}_{\{0, \pm 1, \pm 3\}}, \:\phi\circ h\circ\phi^{-1}(\infty)=\infty.
\end{displaymath}

From Section \ref{n=5, con} we know that
\begin{displaymath}
\phi\circ h\circ\phi^{-1}(z)=-z \mathrm{~or~} \phi\circ h\circ\phi^{-1}=\mathrm{I},
\end{displaymath}
which amounts to
\begin{displaymath}
h=g \mathrm{~or~} h=\mathrm{I}.
\end{displaymath}
So $k=l^{-1}\circ h\in\langle f,g\rangle$.

In conclusion we have
\begin{displaymath}
\mathcal{A_\alpha} = \langle z \mapsto e^{\frac{\pi}{3}i}z,z \mapsto \frac{1}{z}\rangle \simeq D_6.
\end{displaymath}

%----------------------------------------------------------------
\subsection{$\mathcal{A_\alpha}$ is Isomorphic to $\mathbb{Z}_5$}
\label{n=6, Z_5}

In this case assume that for any $h\in\mathcal{A_\alpha}$, $\sigma_h$ is not of Type (6), but there exists some $f\in\mathcal{A_\alpha}$ s.t.~$\sigma_f$ is of Type (5, 1). Under this assumption, there exists some linear fractional transformation $\psi$ s.t.
\begin{displaymath}
\psi(\alpha) =\{0, 1, w, w^2, w^3, w^4\},\: w=e^{\frac{2\pi}{5}i}.
\end{displaymath}

On the other hand, for any $\alpha=\{z_1, z_2, z_3, z_4, z_5, z_6\}\subseteq \widehat{\mathbb C}$, if there exists some linear fractional transformation $\psi$ s.t.~$\psi(\alpha)=\{0, 1, w, w^2, w^3, w^4\}$, we see that ${\psi}^{-1} \circ f \circ \psi$ fixes $\alpha$ and is of Type (5, 1), where $f(z)=wz$. However, since the six points in $\alpha$ are not concyclic, for any $h\in\mathcal{A_\alpha}$, $\sigma_h$ is not of Type (6). Thus we see that the assumption amounts to the existence of the linear fractional transformation $\psi$ s.t.
\begin{displaymath}
\psi(\alpha)=\{0, 1, w, w^2, w^3, w^4\}.
\end{displaymath}

Without lose of generality we assume that
\begin{displaymath}
\alpha=\{0, 1, w, w^2, w^3, w^4\},\: f(z)=wz.
\end{displaymath}
It is obvious that $f \in\mathcal{A_\alpha}$.

For any $g \in \mathcal{A_\alpha}$, it is easy to see that
\begin{displaymath}
|g(\{1, w, w^2, w^3, w^4\})\cap\{1, w, w^2, w^3, w^4\}|\ge4.
\end{displaymath}
So $g$ fixes the unit circle, and thus $g(0)=0$.

As $g(0)=0$, $g\in\mathcal{A}_{\{1, w, w^2, w^3, w^4\}}$. From Section \ref{n=5, D_5} we know that
\begin{displaymath}
\mathcal{A}_{\{1, w, w^2, w^3, w^4\}}=\langle z \mapsto e^{\frac{2\pi}{5}i}z,z \mapsto \frac{1}{z}\rangle.
\end{displaymath}
As $g$ fixes $0$, we see that $g\in\langle f\rangle$.

In conclusion we have
\begin{displaymath}
\mathcal{A_\alpha} = \langle z \mapsto e^{\frac{2\pi}{5}i}\rangle \simeq \mathbb{Z}_5.
\end{displaymath}

%----------------------------------------------------------------
\subsection{$\mathcal{A_\alpha}$ is Isomorphic to $S_4$}
\label{n=6, S_4}

In this case assume that for any $h\in\mathcal{A_\alpha}$, $\sigma_h$ is not of Type (6) or (5, 1), but there exists some $f\in\mathcal{A_\alpha}$ s.t.~$\sigma_f$ is of Type (4, 1, 1). Under this assumption, there exists some linear fractional transformation $\psi$ s.t.
\begin{displaymath}
\psi(\alpha) =\{0,\infty, 1, i, -1, -i\}.
\end{displaymath}

On the other hand, for any $\alpha=\{z_1, z_2, z_3, z_4, z_5, z_6\}\subseteq \widehat{\mathbb C}$, if there exists some linear fractional transformation $\psi$ s.t.~$\psi(\alpha)=\{0, \infty, 1, i, -1, -i\}$, we see that ${\psi}^{-1} \circ f \circ \psi$ fixes $\alpha$ and is of Type (4, 1, 1), where $f(z)=iz$. However, since no five points in $\alpha$ are concyclic, we see that for any $h\in\mathcal{A_\alpha}$, $\sigma_h$ is not of Type (6) or (5, 1). Thus the assumption amounts to the existence of the linear fractional transformation $\psi$ s.t.
\begin{displaymath}
\psi(\alpha)=\{0, \infty, 1, i, -1, -i\}.
\end{displaymath}

Without lose of generality we assume that
\begin{displaymath}
\alpha=\{0, \infty, 1, i, -1, -i\}.
\end{displaymath}

Let $\varphi$ be the stereographic projection from $\widehat{\mathbb{C}}$ to $S^2$. We have
\begin{displaymath}
\varphi(\alpha)=\{(1,0,0),(-1,0,0),(0,1,0),(0,-1,0),(0,0,1),(0,0,-1)\}.
\end{displaymath}

So $\varphi$ maps $\alpha$ to the six vertices of a regular octahedron with its center at the origin. In Section \ref{S_4} we shall prove that
\begin{displaymath}
\mathcal{A_\alpha}=\langle z\mapsto iz, z\mapsto \frac{iz+1}{z+i} \rangle \simeq S_4.
\end{displaymath}

%---------------------------------------------------------------
\subsection{$\mathcal{A_\alpha}$ is Isomorphic to $D_3$}
\label{n=6, D_3}

In this case assume that for any $h\in\mathcal{A_\alpha}$, $\sigma_h$ is not of Type (6), (5, 1) or (4, 1, 1), but there exists some $f\in\mathcal{A_\alpha}$ s.t.~$\sigma_f$ is of Type (3, 3). Under this assumption, there exists some linear fractional transformation $\psi$ s.t.
\begin{displaymath}
\psi(\alpha)=\{1, w, w^2, a, aw, aw^2\},\: |a|\ge 1,\: a\ne 1, w, w^2,\: w=e^{\frac{2\pi}{3}i}.
\end{displaymath}

On the other hand, for any $\{1, w, w^2, a, aw, aw^2\}, |a|\ge 1, a\ne 1, w, w^2$, the linear fractional transformation $z\mapsto wz$ is in $\mathcal{A}_{\{1, w, w^2, a, aw, aw^2\}}$ and of Type (3, 3). Now we aim to find out the specific value $a$ takes when for each element $h\in \mathcal{A}_{\{1, w, w^2, a, aw, aw^2\}}$, $\sigma_h$ is not of Type (6), (5, 1) or (4, 1, 1).

\emph{Case 1: $|a|=1$.}

Suppose $h\in \mathcal{A}_{\{1, w, w^2, a, aw, aw^2\}}$, and $\sigma_h$ is of Type (6), (5, 1) or (4, 1, 1). As $|a|=1$, the six points in $\{1, w, w^2, a, aw, aw^2\}$ are concyclic. Thus no element in $\mathcal{A}_{\{1, w, w^2, a, aw, aw^2\}}$ is of Type (5, 1) or (4, 1, 1), and $\sigma_h$ has to be of Type (6).

From Section \ref{n=6, D_6} we know that there exists some linear fractional transformation $\psi$ s.t.
\begin{displaymath}
\psi(\{1, w, w^2, a, aw, aw^2\})=\{1, \sqrt w, w, w\sqrt w, w^2, w^2\sqrt w\},\: \sqrt w=e^{\frac{\pi}{3}i}.
\end{displaymath}
It is obvious that this amounts to $a=\sqrt w, w\sqrt w$ or $w^2\sqrt w$.

\emph{Case 2: $|a|>1$.}

Suppose $h\in \mathcal{A}_{\{1, w, w^2, a, aw, aw^2\}}$, and $\sigma_h$ is of Type (6), (5, 1) or (4, 1, 1). As $|a|\ge 1$, no five points in $\{1, w, w^2, a, aw, aw^2\}$ are concyclic. Thus no element in $\mathcal{A}_{\{1, w, w^2, a, aw, aw^2\}}$ is of Type $(6)$ or $(5, 1)$, and $\sigma_h$ has to be of Type $(4, 1, 1)$.

From Section \ref{n=6, S_4} we know that there exists some linear fractional transformation $\psi$ s.t.
\begin{displaymath}
\psi(\{1, w, w^2, a, aw, aw^2\})=\{0, \infty, \pm 1, \pm i\}.
\end{displaymath}
Thus there exists a subset $\Sigma \subseteq \alpha$, $|\Sigma|=4$ s.t.~all four points in $\Sigma$ lie on the same circle $C$, and the two points in $\alpha\backslash\Sigma$ lie on different sides of $C$. As $|a|>1$, it is easy to see that
\begin{displaymath}
|\Sigma\cap\{1, w, w^2\}|=|\Sigma\cap\{a, aw, aw^2\}|=2.
\end{displaymath}
Without lose of generality we assume that
\begin{displaymath}
\Sigma=\{w, aw, w^2, aw^2\}.
\end{displaymath}
Thus we have
\begin{displaymath}
\frac{-3a}{(1-a)^2}=[w, aw, w^2, aw^2]\in \mathbb{R}.
\end{displaymath}
As $|a|>1$, the above condition equals to $a\in\mathbb{R}$.
\begin{figure}[!h]
\centering
\includegraphics[width=2.3in]{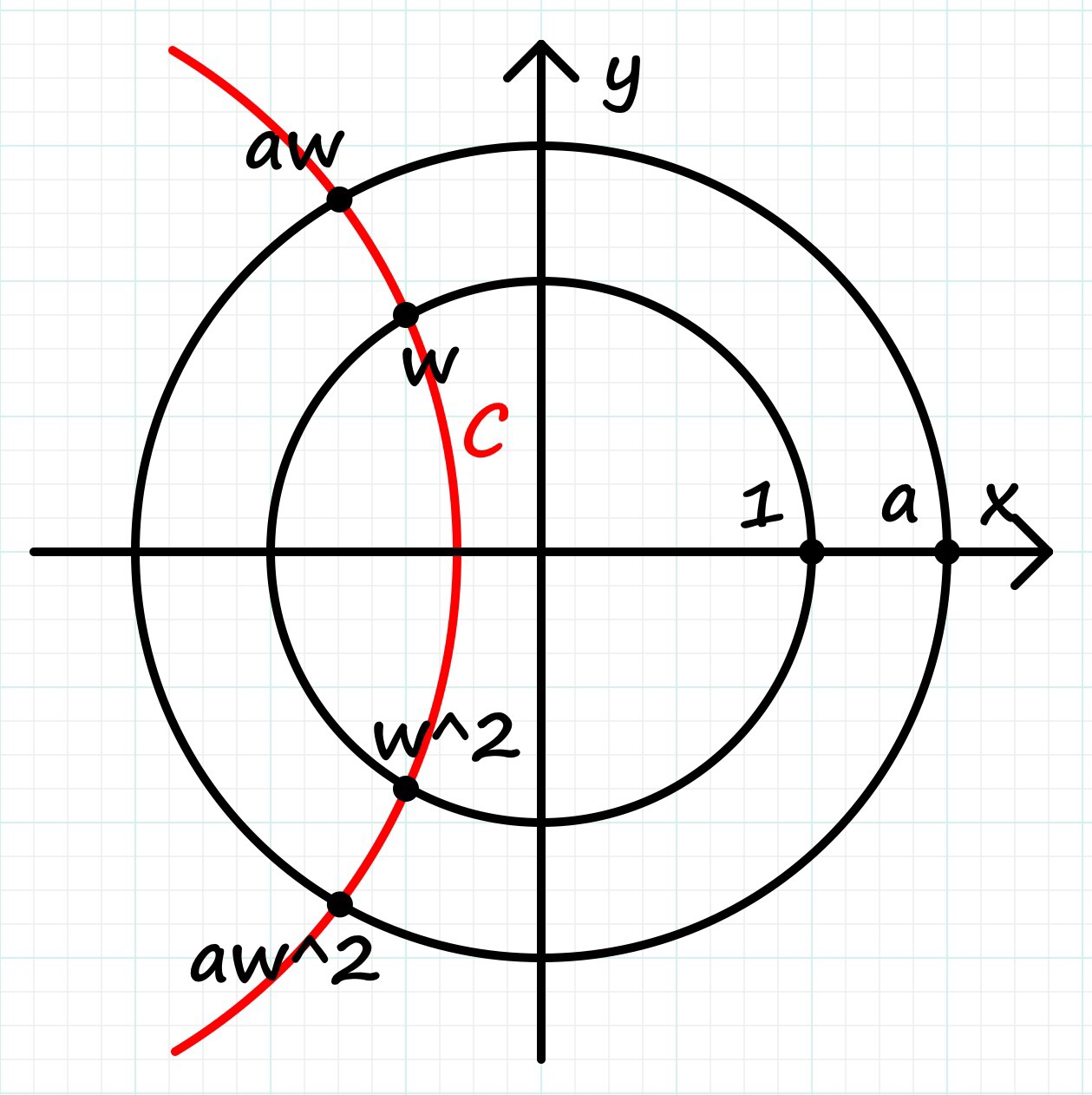}
\includegraphics[width=2.3in]{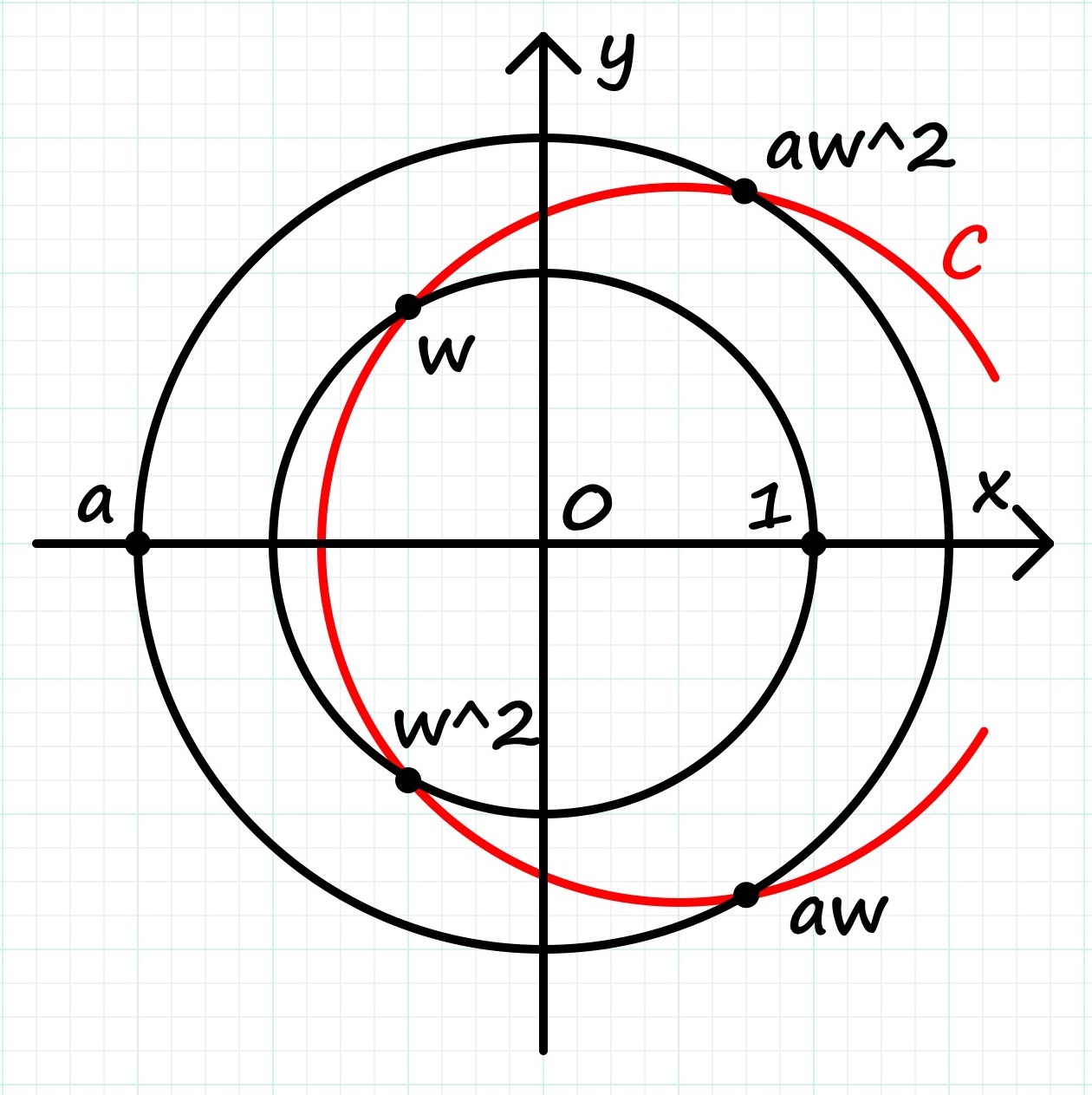}
\caption{$a\in \mathbb{R}$.}
\end{figure}

If $a>1$, $1$ and $a$ lie on the same side of $C$. It con not be the case.

As $a<-1$, we have
\begin{displaymath}
[w, w^2, aw, aw^2]=2,
\end{displaymath}
which equals to that
\begin{displaymath}
a=-2-\sqrt3.
\end{displaymath}

Define the linear fractional transformation $\varphi$
\begin{displaymath}
\varphi(z)=\frac{-(1+\sqrt3)(1+i)}{2}\frac{z-1}{z+2+\sqrt3},
\end{displaymath}
then we have
\begin{displaymath}
\varphi(1)=0, \varphi(-2-\sqrt3)=\infty, \varphi(w)=1, \varphi(w^2)=i, \varphi((-2-\sqrt3)w)=-1, \varphi((-2-\sqrt3)w^2)=-i.
\end{displaymath}

From the above discussion we see that the assumption amounts to the existence of some linear fractional transformation $\psi$ s.t.
\begin{displaymath}
\psi(\alpha)=\{1, w, w^2, a, aw, aw^2\},
\end{displaymath}
where
\begin{displaymath}
w=e^{\frac{2\pi}{3}i},\: |a|\ge 1,\: a\ne(\sqrt w)^j \mathrm{~or~} (-2-\sqrt3)w^j \mathrm{~for~any~} j \in \mathbb{Z}.
\end{displaymath}

Without lose of generality assume that
\begin{displaymath}
\alpha=\{1, w, w^2, a, aw, aw^2\}, \: f(z)=wz, \: g(z)=\frac{a}{z}.
\end{displaymath}
It is obvious that $f,g \in \mathcal{A_\alpha}$. Note that $\langle f,g\rangle\simeq D_3$ acts transitively on $\alpha$.

For any $k \in \mathcal{A_\alpha}$, there exists some $l\in  \langle f,g\rangle$ s.t.~$h=l\circ k$ fixes $1$. We have
\begin{displaymath}
h(\{w, w^2, a, aw, aw^2\})=\{w, w^2, a, aw, aw^2\},\: h(1)=1.
\end{displaymath}
Thus $h\in\mathcal{A}_{\{w, w^2, a, aw, aw^2\}}$. Since no element in $\mathcal{A_\alpha}$ is of Type (6), (5, 1) or (4, 1, 1), and $h(1)=1$, $\sigma_h$ is of Type (2, 2, 1, 1) or identity. Next we shall prove that $h$ has to be the identity.

Suppose that $\sigma_h$ is of Type (2, 2, 1, 1). Thus there exists some linear fractional transformation $\varphi$ s.t.
\begin{displaymath}
\varphi(\alpha)=\{0, \infty, 1, -1, z_0, -z_0\},\: |z_0|\ge1.
\end{displaymath}
Since $h$ fixes $1$, assume
\begin{displaymath}
\varphi(1)=0.
\end{displaymath}
As $0, \infty, 1, -1$ are concyclic, there exists some subset $\Sigma\subseteq\alpha, |\Sigma|=4, 1\in\Sigma$ s.t.~all elements in $\Sigma$ lie on the same circle $C$.

\emph{Case 1: $|a|=1$.}

Without lose of generality assume that $\arg a \in (0,\pi/3)\cup(\pi/3, 2\pi/3)$.
\begin{figure}[!h]
\centering
\includegraphics[width=3in]{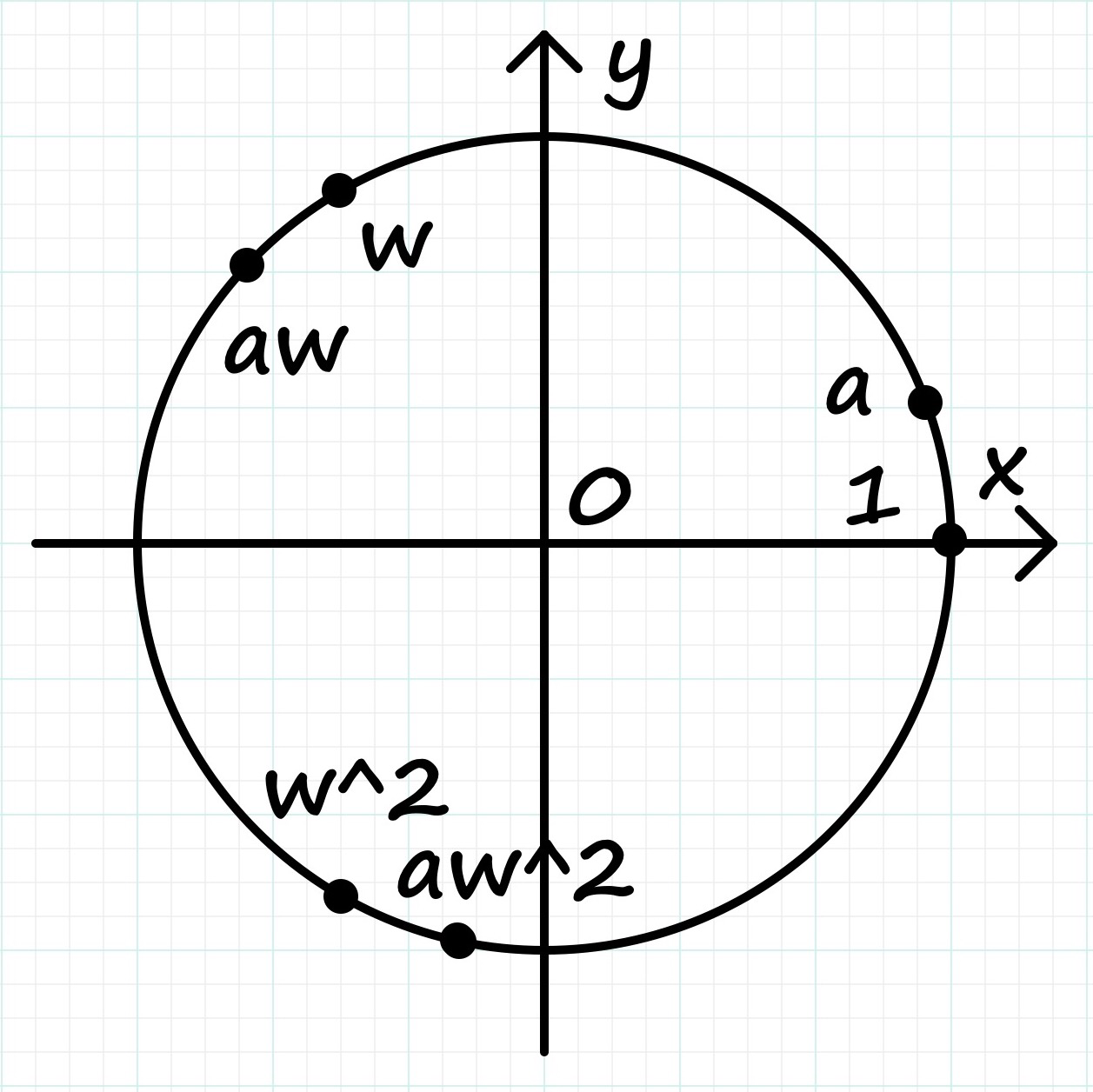}
\caption{$|a|=1$.}
\end{figure}

By observation we have
\begin{displaymath}
\varphi(1)=0,\: \varphi(aw)=\infty,\: \varphi(a)=1, \:\varphi(aw^2)=-1,
\end{displaymath}
or
\begin{displaymath}
\varphi(1)=0, \:\varphi(aw)=\infty, \:\varphi(a)=-1, \:\varphi(aw^2)=1.
\end{displaymath}

So we have
\begin{displaymath}
[1,aw,a,aw^2]=[0,\infty,1,-1],
\end{displaymath}
or
\begin{displaymath}
[1,aw,a,aw^2]=[0,\infty,-1,1],
\end{displaymath}
which equals to
\begin{displaymath}
a=\sqrt w.
\end{displaymath}
This contradicts the assumption that $a\ne \sqrt w$, so $h$ can not be of Type (2, 2, 1, 1).

\emph{Case 2: $|a|>1$.}

As $|a|>1$, it is easy to see that
\begin{displaymath}
|\Sigma\cap\{1,w,w^2\}|=|\Sigma\cap\{a,aw,aw^2\}|=2.
\end{displaymath}
Without lose of generality assume that
\begin{displaymath}
\Sigma=\{1,w,a,aw\}.
\end{displaymath}
Then we have
\begin{displaymath}
\frac{-3a}{(1-a)^2}=[1,w,a,aw]\in\mathbb R,
\end{displaymath}
which amounts to
\begin{displaymath}
a\in\mathbb R.
\end{displaymath}
\begin{figure}[!h]
\centering
\includegraphics[width=2.3in]{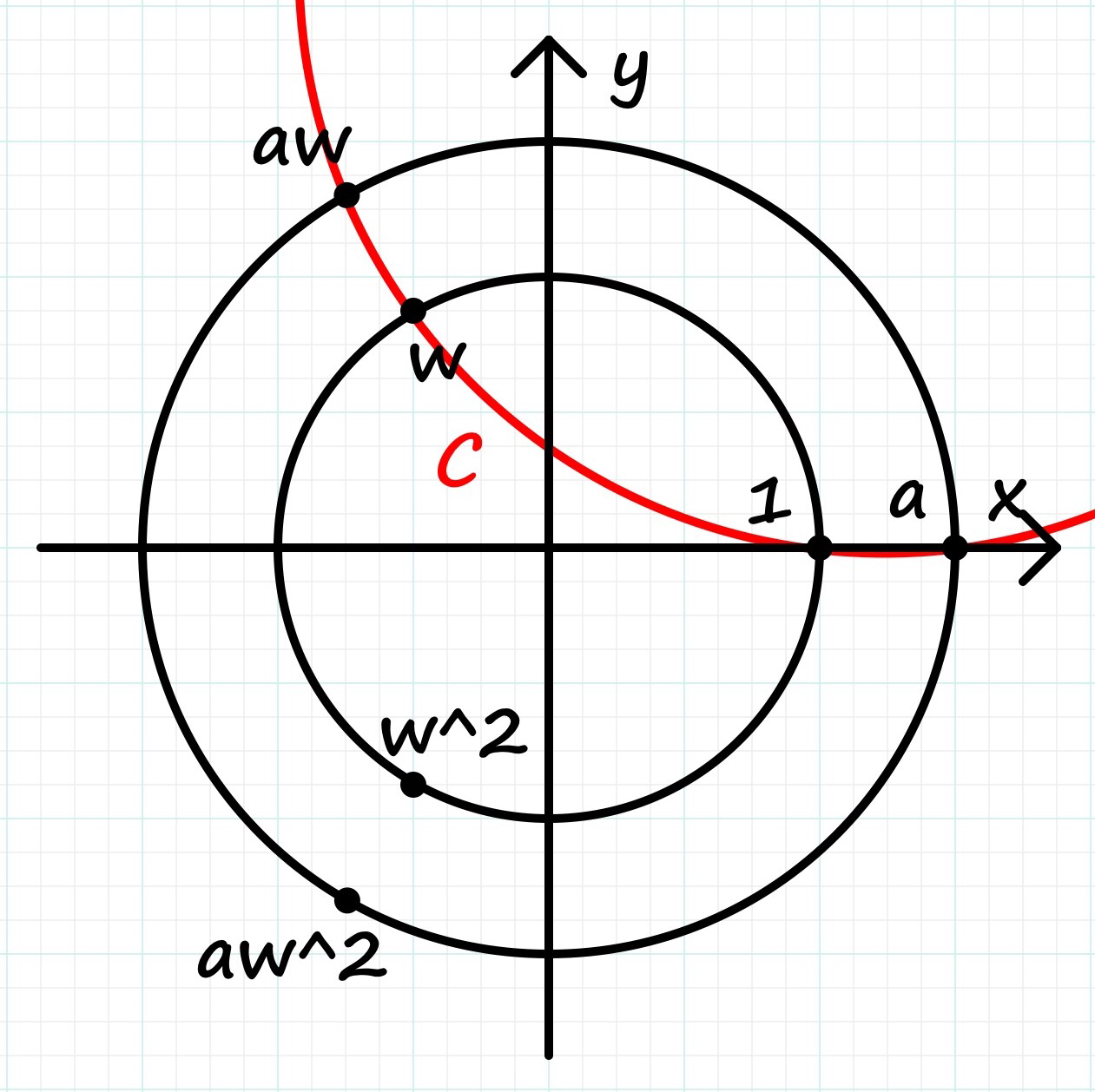}
\includegraphics[width=2.3in]{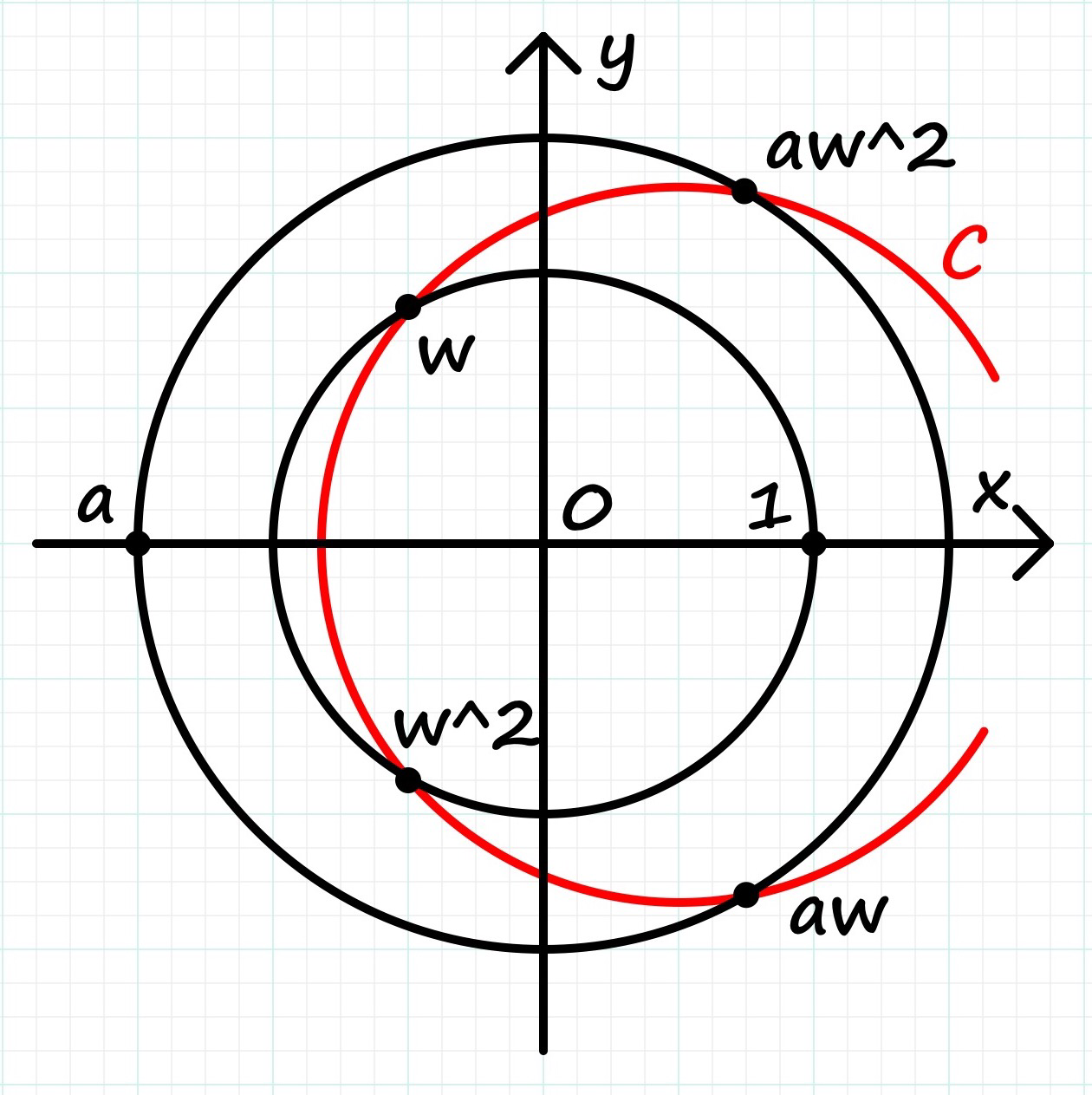}
\caption{$|a|>1$.}
\end{figure}

If $a>1$, $w^2$ and $aw^2$ lie on the same side of $C$, which contradicts the assumption that $\varphi(\alpha)=\{0, \infty, 1, -1, z_0, -z_0\}$.

If $a<-1$, by observation we have
\begin{displaymath}
\varphi(1)=0,\:\varphi(a)=\infty,\:\varphi(w)=1,\:\varphi(aw)=-1,
\end{displaymath}
or
\begin{displaymath}
\varphi(1)=0,\:\varphi(a)=\infty,\:\varphi(w)=-1,\:\varphi(aw)=1.
\end{displaymath}
So we have
\begin{displaymath}
[1, a, w, aw]=[0, \infty, 1, -1],
\end{displaymath}
or
\begin{displaymath}
[1, a, w, aw]=[0, \infty, -1, 1],
\end{displaymath}
which amounts to
\begin{displaymath}
a=-2-\sqrt3.
\end{displaymath}
This also contradicts our assumption that $a\ne-2-\sqrt3$.

So $h$ has to be the identity, and $k=l^{-1}\in\langle f,g\rangle$. In conclusion
\begin{displaymath}
\mathcal A_\alpha=\langle z\mapsto wz, z\mapsto \frac{a}{z}  \rangle \simeq D_3.
\end{displaymath}

%----------------------------------------------------------------
\subsection{$\mathcal{A_\alpha}$ is Isomorphic to $K_4$ or $\mathbb{Z}_2$}

In this case assume that for any $h\in\mathcal{A_\alpha}$, $\sigma_h$ is not of Type (6), (5, 1), (4, 1, 1) or (3, 3), but there exists some $f\in\mathcal{A_\alpha}$ s.t.~$\sigma_f$ is of Type (2, 2, 2). Under this assumption, there exists some linear fractional transformation $\psi$ s.t.
\begin{displaymath}
\psi(\alpha)=\{1, -1, a, -a, b, -b\},\: a\ne\pm b, \:a,b \in \mathbb{C} \backslash \{ 0, \pm 1 \}.
\end{displaymath}

On the other hand, for any $\{1, -1, a, -a, b, -b\},a\ne\pm b, a,b \in \mathbb{C} \backslash \{ 0, \pm 1 \}$, the linear fractional transformation $z\mapsto -z$ is in $\mathcal{A}_{\{\pm 1, \pm a, \pm b\}}$ and of Type (2, 2, 2). Now we aim to find out the specific value $a$ and $b$ takes when for each element $h\in \mathcal{A}_{\{\pm 1, \pm a, \pm b\}}$, $\sigma_h$ is not of Type Type (6), (5, 1), (4, 1, 1) or (3, 3).

%----------------------------------------------------------------
%----------------------------------------------------------------
\emph{Case 1}: There exists some $h\in\mathcal{A}_{\{\pm 1, \pm a, \pm b\}}$ s.t.~$\sigma_h$ is of Type (6).

This amounts to the existence of some linear fractional transformation $\psi$ s.t.
\begin{displaymath}
\psi(\{\pm 1, \pm a, \pm b\})=\{1, w, w^2, w^3, w^4, w^5\}, \:w=e^{\frac{\pi}{3}i}.
\end{displaymath}
We conclude that $\pm 1, \pm a, \pm b$ lie on the same circle: the real axis or the unit circle.

If $\pm 1, \pm a, \pm b$ lie on the unit circle, we have
\begin{displaymath}
\{\pm 1, \pm a, \pm b\}=\{1, w, w^2, w^3, w^4, w^5\}.
\end{displaymath}

If $\pm 1, \pm a, \pm b$ lie on the real axis, without lose of generality assume that
\begin{displaymath}
0<a<b.
\end{displaymath}

If $a<b<1$, we have
\begin{displaymath}
[a,1,-1,-a]=[1,w^2,w^3,w^5],\:[b,1,-1,-b]=[1,w,w^2,w^3],
\end{displaymath}
which amounts to
\begin{displaymath}
a=7-4\sqrt3, \:b=2-\sqrt3.
\end{displaymath}

If $a<1<b$, we have
\begin{displaymath}
[a,1,-1,-a]=[1,w,w^4,w^5],\:[b,1,-1,-b]=[1,w^5,w^2,w],
\end{displaymath}
which amounts to
\begin{displaymath}
a=2-\sqrt3, \:b=2+\sqrt3.
\end{displaymath}

If $1<a<b$, we have
\begin{displaymath}
[a,1,-1,-a]=[1,w^5,w^4,w^3],\:[b,1,-1,-b]=[1,w^4,w^3,w],
\end{displaymath}
which amounts to
\begin{displaymath}
a=2+\sqrt3, \:b=7+4\sqrt3.
\end{displaymath}

On the other hand, the linear fractional transformation
\begin{displaymath}
f_1(z)=\frac{\sqrt3 z-(2-\sqrt3)}{(2+\sqrt3)z+\sqrt3}
\end{displaymath}
is of Type (6) and in $\mathcal{A}_{\{\pm(7-4\sqrt3), \pm(2-\sqrt3),\pm 1\}}$ and
\begin{displaymath}
f_2(z)=\frac{\sqrt3 z-1}{z+\sqrt3}
\end{displaymath}
is of Type (6) and in $\mathcal{A}_{\{\pm(2-\sqrt3),\pm 1,\pm(2+\sqrt3)\}}$ and
\begin{displaymath}
f_1(z)=\frac{\sqrt3 z-(2+\sqrt3)}{(2-\sqrt3)z+\sqrt3}
\end{displaymath}
is of Type (6) and in $\mathcal{A}_{\{\pm 1, \pm(2+\sqrt3),\pm(7+4\sqrt3)\}}$.

Thus there exists some $h\in\mathcal{A}_{\{\pm 1, \pm a, \pm b\}}$ s.t.~$\sigma_h$ is of Type (6) if and only if
\begin{displaymath}
\{\pm 1, \pm a, \pm b\}=\{\pm(7-4\sqrt3), \pm(2-\sqrt3),\pm 1\},
\end{displaymath}
or
\begin{displaymath}
\{\pm 1, \pm a, \pm b\}=\{\pm(2-\sqrt3),\pm 1,\pm(2+\sqrt3)\},
\end{displaymath}
or
\begin{displaymath}
\{\pm 1, \pm a, \pm b\}=\{\pm 1, \pm(2+\sqrt3),\pm(7+4\sqrt3)\}.
\end{displaymath}

%----------------------------------------------------------------
%----------------------------------------------------------------
\emph{Case 2}: There exists some $h\in\mathcal{A}_{\{\pm 1, \pm a, \pm b\}}$ s.t.~$\sigma_h$ is of Type (5, 1).

From \ref{n=6, Z_5} we know that $\mathcal{A}_{\{\pm 1, \pm a, \pm b\}}\simeq\mathbb{Z}_5$, and does not contain elements of order two. This contradicts the assumption that $\sigma_h$ is of Type (2, 2, 2).

%----------------------------------------------------------------
%----------------------------------------------------------------
\emph{Case 3}: There exists some $h\in\mathcal{A}_{\{\pm 1, \pm a, \pm b\}}$ s.t.~$\sigma_h$ is of Type (4, 1, 1).

This amounts to the existence of some linear fractional transformation $\psi$ s.t.
\begin{displaymath}
\psi(\{\pm 1, \pm a, \pm b\})=\{0, \infty, \pm 1, \pm i\}.
\end{displaymath}

Set $\Sigma=\psi^{-1}(\{1, i, -1, -i\})$. So all four points in $\Sigma$ lie on the circle $C$, and the two points in $\{\pm 1, \pm a, \pm b\}\backslash\Sigma$ are inverse points with respect to $C$.

If $|\Sigma\cap\{1,-1\}|=0$, we have $\Sigma=\{\pm a, \pm b\})$, and thus $C$ is a Euclidean circle with its center at the origin, or a line through the origin. As $\pm 1$ are inverse points with respect to $C$, $C$ has to be the imaginary line. Assume that \begin{displaymath}
a=si,\:b=ti.
\end{displaymath}
Without lose of generality assume that
\begin{displaymath}
0<s<t.
\end{displaymath}
So we have
\begin{displaymath}
[ti,si,-si,-ti]=[1,ti,-si,-1]=[0,1,\infty,-1]=2,
\end{displaymath}
which amounts to
\begin{displaymath}
s=\sqrt2-1,\: t=\sqrt2+1.
\end{displaymath}

If $|\Sigma\cap\{1,-1\}|=1$, without lose of generality assume that
\begin{displaymath}
\Sigma=\{1, \pm a, b\}.
\end{displaymath}
So we have
\begin{displaymath}
[-1,1,-b,b]=[-1,a,-b,-a]=[0,1,\infty,-1]=2,
\end{displaymath}
or
\begin{displaymath}
[-1,1,-b,-a]=[-1,b,-b,a]=[0,1,\infty,-1]=2,
\end{displaymath}
or
\begin{displaymath}
[-1,1,-b,a]=[-1,b,-b,-a]=[0,1,\infty,-1]=2,
\end{displaymath}
which amounts to
\begin{displaymath}
a=\pm(\sqrt2+1)i,\:b=-3-2\sqrt2\mathrm{~or~}a=\pm(\sqrt2-1)i, \:b=-3+2\sqrt2,
\end{displaymath}
or
\begin{displaymath}
a=2+\sqrt3,\: b=-7-4\sqrt3\mathrm{~or~}a=2-\sqrt3,\: b=-7+4\sqrt3\
\end{displaymath}
or
\begin{displaymath}
a=-2-\sqrt3,\: b=-7-4\sqrt3\mathrm{~or~}a=-2+\sqrt3,\: b=-7+4\sqrt3\
\end{displaymath}
respectively. As the six points in $\{\pm 1, \pm a, \pm b\}$ are not concyclic, $a, b$ can not both be real.

If $|\Sigma\cap\{1,-1\}|=2$, let $\{\pm 1, \pm a, \pm b\}\backslash\Sigma=\{x, y\}$ and there are two possibilities. The first is that $x+y\ne0$, and the second is that $x+y=0$.

When $x+y\ne0$, without lose of generality assume that
\begin{displaymath}
\Sigma=\{\pm 1, a, b\}.
\end{displaymath}
As the six points in $\{\pm 1, \pm a, \pm b\}$ are not concyclic, it is easy to see that
\begin{displaymath}
a, b\notin\mathbb{R}.
\end{displaymath}
As $-a, -b$ are inverse points with respect to $C$, we see that
\begin{displaymath}
\mathrm{Im}(a)\mathrm{Im}(b)<0.
\end{displaymath}
\begin{figure}[!h]
\centering
\includegraphics[width=3in]{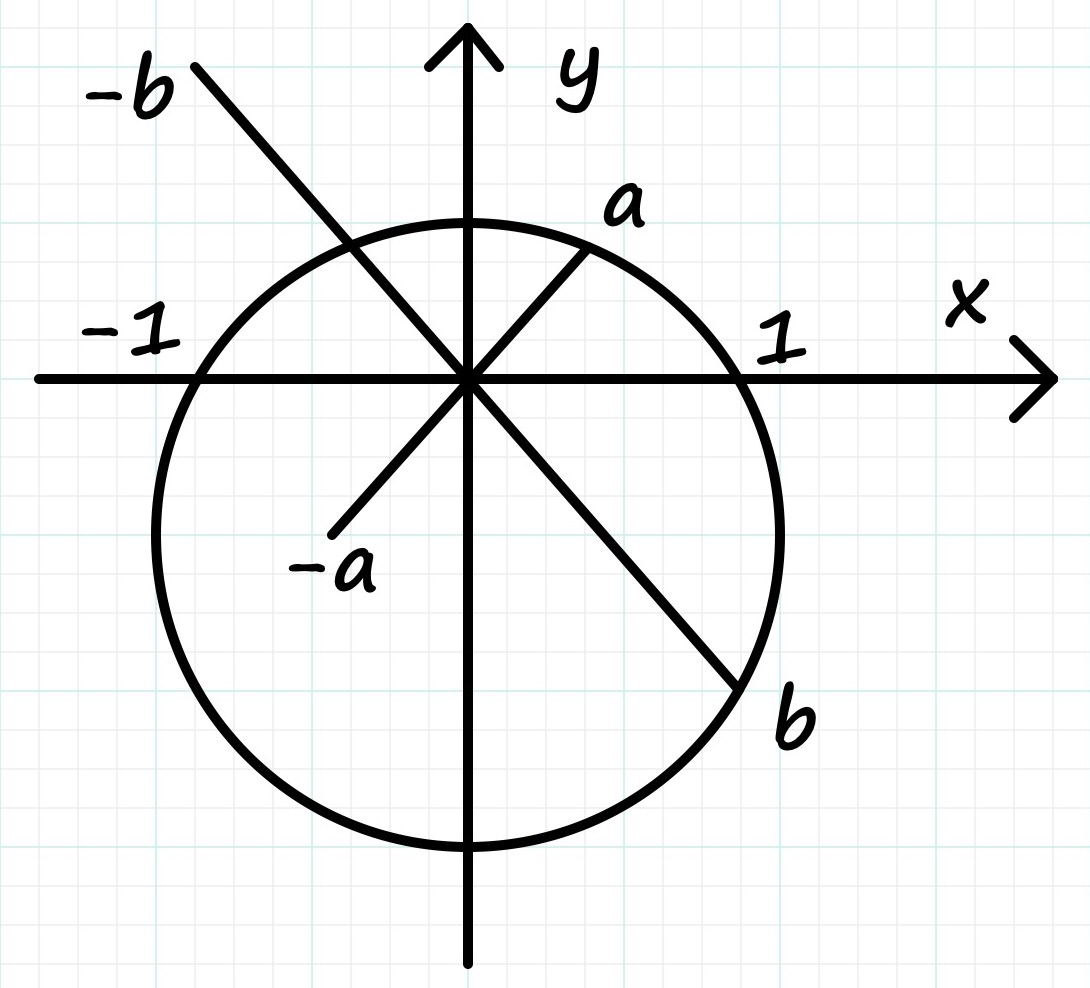}
\caption{$x+y\ne0$.}
\end{figure}

So we have
\begin{displaymath}
[-a,1,-b,-1]=[-a,a,-b,b]=[0,1,\infty,-1]=2
\end{displaymath}
which amounts to
\begin{displaymath}
a=(\pm\sqrt2\pm1)i,\:b=\frac{1}{a}.
\end{displaymath}

When $x+y=0$, without lose of generality assume that
\begin{displaymath}
\Sigma=\{\pm 1, \pm a\}.
\end{displaymath}
So $C$ is a Euclidean circle with its center at the origin or is th real axis. As $\pm b$ are inverse points with respect to $C$, $C$ has to be the real axis. Without lose of generality assume that
\begin{displaymath}
a>0.
\end{displaymath}

So we have
\begin{displaymath}
[b,1,-b,-a]=[1,a,-a,-1]=[0,1,\infty,-1]=2
\end{displaymath}
which amounts to
\begin{displaymath}
a=3+2\sqrt2,\:b=\pm(\sqrt2+1)i\mathrm{~or~}a=3-2\sqrt2,\:b=\pm(\sqrt2-1)i.
\end{displaymath}

On the other hand, the linear fractional transformation
\begin{displaymath}
f_1(z)=\frac{iz+1}{z+i}
\end{displaymath}
is of Type (4, 1, 1) and in $\mathcal{A}_{\{\pm (\sqrt2-1)i,\pm(\sqrt2+1)i,\pm 1\}}$ and
\begin{displaymath}
f_2(z)=\frac{z+\sqrt2-1}{-(\sqrt2+1)z+1}
\end{displaymath}
is of Type (4, 1, 1) and in $\mathcal{A}_{\{\pm(3-2\sqrt2), \pm 1,\pm(\sqrt2-1)i\}}$ and
\begin{displaymath}
f_3(z)=\frac{z+\sqrt2+1}{-(\sqrt2-1)z+1}
\end{displaymath}
is of Type (4, 1, 1) and in $\mathcal{A}_{\{\pm 1, \pm(3+2\sqrt2),\pm(\sqrt2+1)i\}}$.

Thus there exists some $h\in\mathcal{A}_{\{\pm 1, \pm a, \pm b\}}$ s.t.~$\sigma_h$ is of Type (4, 1, 1) if and only if
\begin{displaymath}
\{\pm 1, \pm a, \pm b\}=\{\pm (\sqrt2-1)i,\pm(\sqrt2+1)i,\pm 1\},
\end{displaymath}
or
\begin{displaymath}
\{\pm 1, \pm a, \pm b\}=\{\pm(3-2\sqrt2), \pm 1,\pm(\sqrt2-1)i\},
\end{displaymath}
or
\begin{displaymath}
\{\pm 1, \pm a, \pm b\}=\{\pm 1, \pm(3+2\sqrt2),\pm(\sqrt2+1)i\}.
\end{displaymath}

%----------------------------------------------------------------
%----------------------------------------------------------------
\emph{Case 4}: There exists some $h\in\mathcal{A}_{\{\pm 1, \pm a, \pm b\}}$ s.t.~$\sigma_h$ is of Type (3, 3).

This amounts to the existence of some linear fractional transformation $\psi$ s.t.
\begin{displaymath}
\psi(\{\pm 1, \pm a, \pm b\})=\{1, w, w^2, z_0, z_0w, z_0w^2\},
\end{displaymath}
where
\begin{displaymath}
z_0\ne(\sqrt w)^j \mathrm{~or~} (-2-\sqrt3)w^j \mathrm{~for~any~} j \in \mathbb{Z}, \: w=e^{\frac{2\pi}{3}i}.
\end{displaymath}

From Section \ref{n=6, D_3} we know that
\begin{displaymath}
\mathcal A_{\psi(\{\pm 1, \pm a, \pm b\})}=\langle z\mapsto wz, z\mapsto \frac{z_0}{z}  \rangle \simeq D_3.
\end{displaymath}
The only three elements of order two in $\mathcal A_{\psi(\{\pm 1, \pm a, \pm b\})}$ are
\begin{displaymath}
f_1(z)=\frac{z_0}{z},\:f_2(z)=\frac{wz_0}{z},\:f_3(z)=\frac{w^2z_0}{z}.
\end{displaymath}
So we have
\begin{displaymath}
f_i(\{1, w, w^2\})=\{z_0, z_0w, z_0w^2\},\:f_i(\{z_0, z_0w, z_0w^2\})=\{1, w, w^2\},\:i=1,2,3.
\end{displaymath}

Define the linear fractional transformation $g$
\begin{displaymath}
g(z)=-z.
\end{displaymath}
Then $g\in\mathcal{A}_{\{\pm 1, \pm a, \pm b\}}$ and is of order two. The we have
\begin{displaymath}
g=\psi^{-1}\circ f_i \circ \psi
\end{displaymath}
for some $i=1, 2,$ or $3$.
So we have
\begin{displaymath}
|\{1, w, w^2\}\cap\psi(\{\pm1\})|=|\{1, w, w^2\}\cap\psi(\{\pm a\})|=|\{1, w, w^2\}\cap\psi(\{\pm b\})|=1.
\end{displaymath}

Without lose of generality assume that
\begin{displaymath}
\psi(\{1, a, b\})=\{1, w, w^2\}
\end{displaymath}
and
\begin{displaymath}
\psi(1)=1,\:\psi(a)=w,\:\psi(b)=w^2.
\end{displaymath}

Now there are two possibilities. The first is
\begin{displaymath}
\psi(-a)=w\psi(-1),\:\psi(-b)=w\psi(-a).
\end{displaymath}
and the second is
\begin{displaymath}
\psi(-1)=w\psi(-a),\:\psi(-a)=w\psi(-b).
\end{displaymath}

Define the linear fractional transformation $\varphi$
\begin{displaymath}
\varphi(z)=\frac{((a-2b+1)w+2a-b-1)z+(ab-2a+b)w-ab-a+2b}{((-2a+b+1)w+-a+2b-1)z+(ab+a-2b)w-ab+2a-b}
\end{displaymath}
satisfies
\begin{displaymath}
\varphi(1)=1,\:\varphi(a)=w,\:\varphi(b)=w^2.
\end{displaymath}

The equations
$$
\left\{\begin{array}{rl}
        \varphi(-a)=w\varphi(-1) \\
        \varphi(-b)=w\varphi(-a) \\
       \end{array}
\right.
$$
reduces to
$$
\left\{\begin{array}{rl}
        a^2b^2+a^3-2a^2b-2ab+b^2+a=0 \\
        a^2b^2+b^3-2ab^2-2ab+a^2+b=0 \\
       \end{array}
\right.
$$
which amounts to
\begin{displaymath}
a=w, b=w^2\mathrm{~or~}a=w^2, b=w
\end{displaymath}
which can not be the case.

The equations
$$
\left\{\begin{array}{rl}
        \varphi(-1)=w\varphi(-a) \\
        \varphi(-a)=w\varphi(-b) \\
       \end{array}
\right.
$$
reduces to
$$
\left\{\begin{array}{rl}
       a^3b+ a^2b^2+a^3-5a^2b+2ab^2+2a^2-5ab+b^2+a+b=0 \\
       ab^3+ a^2b^2+b^3-5ab^2+2a^2b+2b^2-5ab+a^2+a+b=0. \\
       \end{array}
\right.
$$

As
\begin{displaymath}
a^3b+ a^2b^2+a^3-5a^2b+2ab^2+2a^2-5ab+b^2+a+b=(a+1)(a^2b+ab^2+a^2-6ab+b^2+a+b)
\end{displaymath}
and
\begin{displaymath}
ab^3+ a^2b^2+b^3-5ab^2+2a^2b+2b^2-5ab+a^2+a+b=(b+1)(a^2b+ab^2+a^2-6ab+b^2+a+b),
\end{displaymath}
the above equations amounts to
\begin{displaymath}
a^2b+ab^2+a^2-6ab+b^2+a+b=0.
\end{displaymath}

Thus there exists some $h\in\mathcal{A}_{\{\pm 1, \pm a, \pm b\}}$ s.t.~$\sigma_h$ is of Type (3, 3) if and only if
\begin{displaymath}
a^2b+ab^2+a^2-6ab+b^2+a+b=0.
\end{displaymath}

%----------------------------------------------------------------
%----------------------------------------------------------------
From the above discussion we see that the assumption amounts to the existence of some linear fractional transformation $\psi$ s.t.
\begin{displaymath}
\psi(\alpha)=\{1, -1, a, -a, b, -b\},\: a\ne\pm b, \:a,b \in \mathbb{C} \backslash \{ 0, \pm 1 \},
\end{displaymath}
where $\{\pm1,\pm a,\pm b\}$ is not the sets we discussed above.

Without lose of generality assume
\begin{displaymath}
\alpha=\{\pm1,\pm a,\pm b\},\: f(z)=-z,
\end{displaymath}
where $\{\pm1,\pm a,\pm b\}$ is not the sets we discussed above.

So $f\in\mathcal A_\alpha$ and $\sigma_f =(1,2)(3,4)(5,6)$.

If $\exists g\in\mathcal{A_\alpha}$ s.t.~$g\ne f, g\ne \mathrm{I}$, without lose of generality assume that $\sigma_g(1)\ne\sigma_f(1)=2$.

\emph{Case 1:} $\sigma_g(1)=1$.

We see that $g$ is of Type (2, 2, 1, 1), and it has another fixed point besides $1$. There are two possibilities.

If $\sigma_g(2)=2$, it could be that
\begin{displaymath}
\sigma_g=(3,4)(5,6)=:\pi_1
\end{displaymath}
or
\begin{displaymath}
\sigma_g=(3,5)(4,6)=:\pi_2
\end{displaymath}
or
\begin{displaymath}
\sigma_g=(3,6)(4,5)=:\pi_3.
\end{displaymath}

If $\sigma_g(2)\ne2$, as $3,4,5,6$ are congruent, we may assume without lose of generality that
\begin{displaymath}
\sigma_g(3)=3.
\end{displaymath}
So it could be that
\begin{displaymath}
\sigma_g=(2,4)(5,6)=:\pi_4
\end{displaymath}
or
\begin{displaymath}
\sigma_g=(2,5)(4,6)=:\pi_5
\end{displaymath}
or
\begin{displaymath}
\sigma_g=(2,6)(4,5)=:\pi_6.
\end{displaymath}

However
\begin{displaymath}
\pi_1\sigma_f=(1,2),\:\pi_2\sigma_f=(1,2)(3,6)(4,5),\:\pi_3\sigma_f=(1,2)(3,5)(4,6),
\end{displaymath}
\begin{displaymath}
\pi_4\sigma_f=(1,4,3,2),\:\pi_1\sigma_f=(1,5,4,3,6,2),\:\pi_1\sigma_f=(1,6,4,3,5,2).
\end{displaymath}
As a result
\begin{displaymath}
\sigma_g=\pi_2\mathrm{~or~}\pi_3.
\end{displaymath}

\emph{Case 2:} $\sigma_g(1)\ne1$.

As $3,4,5,6$ are congruent, we may assume without lose of generality that
\begin{displaymath}
\sigma_g(1)=3.
\end{displaymath}
So it could be that
\begin{displaymath}
\sigma_g=(1,3)(2,4)(5,6)=:\rho_1
\end{displaymath}
or
\begin{displaymath}
\sigma_g=(1,3)(2,5)(4,6)=:\rho_2
\end{displaymath}
or
\begin{displaymath}
\sigma_g=(1,3)(2,6)(4,5)=:\rho_3
\end{displaymath}
or
\begin{displaymath}
\sigma_g=(1,3)(2,4)=:\rho_4
\end{displaymath}
or
\begin{displaymath}
\sigma_g=(1,3)(2,5)=:\rho_5
\end{displaymath}
or
\begin{displaymath}
\sigma_g=(1,3)(2,6)=:\rho_6
\end{displaymath}
So it could be that
\begin{displaymath}
\sigma_g=(1,3)(4,5)=:\rho_7
\end{displaymath}
or
\begin{displaymath}
\sigma_g=(1,3)(4,6)=:\rho_8
\end{displaymath}
or
\begin{displaymath}
\sigma_g=(1,3)(5,6)=:\rho_9.
\end{displaymath}

However
\begin{displaymath}
\rho_1\sigma_f=(1,4)(2,3),\:\rho_2\sigma_f=(1,5,4)(2,3,6),\:\rho_3\sigma_f=(1,6,4)(2,3,5)\:
\end{displaymath}
\begin{displaymath}
\rho_4\sigma_f=(1,4)(2,3)(5,6),\:\rho_5\sigma_f=(1,5,6,2,3,4),\:\rho_6\sigma_f=(1,6,5,2,3,4),\:
\end{displaymath}
\begin{displaymath}
\rho_7\sigma_f=(1,2,3,5,6,4),\:\rho_8\sigma_f=(1,2,3,6,5,4),\:\rho_9\sigma_f=(1,2,3,4).
\end{displaymath}

As a result
\begin{displaymath}
\sigma_g=\rho_1\mathrm{~or~}\rho_4.
\end{displaymath}

Notice that
\begin{displaymath}
\langle \sigma_f, \pi_2\rangle=\{\mathrm{I},(1,2)(3,4)(5,6),(3,5)(4,6),(1,2)(3,6)(4,5)\},
\end{displaymath}
\begin{displaymath}
\langle \sigma_f, \pi_3\rangle=\{\mathrm{I},(1,2)(3,4)(5,6),(3,6)(4,5),(1,2)(3,5)(4,6)\},
\end{displaymath}
\begin{displaymath}
\langle \sigma_f, \rho_1\rangle=\{\mathrm{I},(1,2)(3,4)(5,6),(1,4)(2,3),(1,3)(2,4)(5,6)\},
\end{displaymath}
\begin{displaymath}
\langle \sigma_f, \rho_4\rangle=\{\mathrm{I},(1,2)(3,4)(5,6),(1,3)(2,4),(1,4)(2,3)(5,6)\}.
\end{displaymath}

If $\mathcal A_\alpha\ne\{\mathrm{I}, f\}$, then there exists some $h\in\mathcal A_\alpha$ s.t.~$h$ is of Type (2, 2, 1, 1), and the two fixed points of $h$ are on the same orbit of $f$. Without lose of generality assume that $h$ fixes $\pm1$. Thus
\begin{displaymath}
h(z)=\frac{pz+q}{qz+p},\: p,q \in\mathbb C,\:a^2\ne b^2,\:b\ne0.
\end{displaymath}
We have
\begin{displaymath}
h(a)=b,\:h(-a)=-b
\end{displaymath}
or
\begin{displaymath}
h(a)=-b,\:h(-a)=b.
\end{displaymath}
Anyway
\begin{displaymath}
h(a)+h(-a)=0
\end{displaymath}
which means that
\begin{displaymath}
p=0.
\end{displaymath}
So we have
\begin{displaymath}
ab=1 (\mathrm{or~}-1).
\end{displaymath}

In conclusion, if $ab=\pm1$
\begin{displaymath}
\mathcal{A_\alpha} = \langle z \mapsto -z, z \mapsto \frac{1}{z} \rangle \simeq K_4.
\end{displaymath}
Otherwise
\begin{displaymath}
\mathcal{A_\alpha} = \langle z \mapsto -z\rangle \simeq \mathbb Z _2.
\end{displaymath}

%----------------------------------------------------------------
\subsection{Another Possibility}

In this case assume that for any $h\in\mathcal{A_\alpha}$, $\sigma_h$ is not of Type (6), (5, 1), (4, 1, 1), (3, 3) or (2, 2, 2) but there exists some $f\in\mathcal{A_\alpha}$ s.t.~$\sigma_f$ is of Type (2, 2, 1, 1). Under this assumption, there exists some linear fractional transformation $\psi$ s.t.
\begin{displaymath}
\psi(\alpha)=\{0, \infty, \pm1, \pm a\},\: a\ne 0, \pm1.
\end{displaymath}

Define the linear fractional transformation $g$
\begin{displaymath}
g(z)=\frac{a}{z}.
\end{displaymath}
We see that $g\in \mathcal{A_\alpha}$ and is of Type (2, 2, 2), which contradicts our assumption. This case is not possible.

%----------------------------------------------------------------
\subsection{Conclusion}
\label{n=6, con}

From the above discussion we shall drive the following theorem
\begin{thm}

Set $\alpha=\{z_1,\: z_2,\: z_3,\: z_4,\: z_5,\: z_6\}\subseteq \widehat{\mathbb{C}}$.

\begin{enumerate}

\item If there exists some linear fractional transformation $\psi$ such that
$$
\psi(\alpha)=\{1,\: w,\: w^2,\: w^3,\: w^4, \:w^5\},\:w=e^{\frac{\pi}{3}i},
$$
then
$$
\mathcal{A_\alpha}={\psi}^{-1}\langle z \mapsto e^{\frac{\pi}{3}i}z,\:z \mapsto \frac{1}{z}\rangle {\psi}\simeq D_6,
$$
and its multiplicity vector is $(0,\:1,\:0,\:0,\:0,\:1)$;

\item if there exists some linear fractional transformation $\psi$ such that
$$
\psi(\alpha) =\{0,\: 1,\: w, \:w^2, \:w^3,\: w^4\},\:w=e^{\frac{2\pi}{5}i},
$$
then
$$
\mathcal{A_\alpha}={\psi}^{-1}\langle z \mapsto e^{\frac{2\pi}{5}i}z\rangle {\psi}\simeq \mathbb Z _5,
$$
and its multiplicity vector is $(1,\:1,\:1,\:0,\:0)$;

\item if there exists some linear fractional transformation $\psi$ such that
$$\psi(\alpha) =\{0,\: \infty, \:1, \:i, \:-1,\: -i\},$$
then
$$\mathcal{A_\alpha}={\psi}^{-1}\langle z\mapsto iz,\: z\mapsto \frac{iz+1}{z+i}\rangle {\psi}\simeq S_4,
$$
and its multiplicity vector is $(0,\:0,\:1,\:0,\:0)$;

\item if there exists some linear fractional transformation $\psi$ such that
$$
\psi(\alpha)=\{1,\:w,\:w^2,\:a,\:aw,\:aw^2\},\:|a|\ge1,\:a\ne1,\:w,\:w^2,\:w=e^{\frac{2\pi}{3}i},
$$
\begin{enumerate}

\item if $a=\sqrt w,\:w\sqrt w\text{ or }w^2\sqrt w$,
$$\alpha=\{1,\:\sqrt w,\: w,\:w\sqrt w,\: w^2,\:w^2\sqrt w\}$$
and this is case 1 ;

\item if $a=-(2+\sqrt3),\:-(2+\sqrt3)w\text{ or }-(2+\sqrt3)w^2$, then there exists some linear fractional transformation $\phi$ such that
$$\phi(\alpha) =\{0,\: \infty, \:1, \:i, \:-1,\: -i\}$$
and this is case 3;

\item otherwise
$$\mathcal{A_\alpha}={\psi}^{-1}\langle z \mapsto e^{\frac{2\pi}{3}i}z,\:z \mapsto \frac{a}{z}\rangle\psi \simeq D_3,
$$
and its multiplicity vector is $(1,\:1,\:0)$; \label{a}

\end{enumerate}

\item if there exists some linear fractional transformation $\psi$ such that
$$\psi(\alpha)=\{1,\: -1, \:a,\: -a,\: b, \:-b\},\:a\ne\pm b, \:a,\:b \in \mathbb{C} \backslash \{ 0, \pm 1 \},$$

\begin{enumerate}

\item if $\{\pm1,\:\pm a,\: \pm b\}=$
$$\{\pm(7-4\sqrt3), \:\pm(2-\sqrt3),\:\pm 1\},\:\{\pm(2-\sqrt3),\:\pm 1,\:\pm(2+\sqrt3)\},\:\{\pm 1,\: \pm(2+\sqrt3),\:\pm(7+4\sqrt3)\}$$
then there exists some linear fractional transformation $\phi$ such that
$$\phi(\alpha)=\{1,\: w,\: w^2,\: w^3,\: w^4, \:w^5\},\:w=e^{\frac{\pi}{3}i}$$
and this is case 1;

\item if $\{\pm1,\:\pm a,\: \pm b\}=$
$$\{\pm (\sqrt2-1)i,\:\pm(\sqrt2+1)i,\:\pm 1\},\:\{\pm(3-2\sqrt2), \:\pm 1,\pm(\sqrt2-1)i\},\:\{\pm 1, \:\pm(3+2\sqrt2),\:\pm(\sqrt2+1)i\},$$
then there exists some linear fractional transformation $\phi$ such that
$$\phi(\alpha) =\{0,\: \infty, \:1, \:i, \:-1,\: -i\}$$
and this is case 3;

\item if $a^2b+ab^2+a^2-6ab+b^2+a+b=0,$ then there exists some linear fractional transformation $\phi$ such that
$$\phi(\alpha)=\{1,\:w,\:w^2,\:c,\:cw,\:cw^2\},\:c\ne0,\:1,\:w,\:w^2,\:w=e^{\frac{2\pi}{3}i}$$
and this is case 4c;

\item otherwise,

\begin{enumerate}
 \item if $ab=\pm1$,
$$
\mathcal{A_\alpha}={\psi}^{-1}\langle z \mapsto -z,\:z \mapsto  \frac{1}{z}\rangle\psi \simeq K_4,
$$
and its multiplicity vector is $(1,\:0,\:1,\:1)$ (viewed as type 2+mC);
 \item if $ab\ne\pm1$,
$$
\mathcal{A_\alpha}={\psi}^{-1}\langle z \mapsto -z\rangle\psi \simeq \mathbb{Z}_2,
$$
and its multiplicity vector is $(1,\:2)$;
\end{enumerate}

\end{enumerate}

\item otherwise,
$$\mathcal{A_\alpha}=\{\mathrm{Id}\}.$$

\end{enumerate}

\end{thm}

%----------------------------------------------------------------
\section{Each Finite Subgroup of $\text{PSL}(2,\:\mathbb C)$ is a Stabilizer of Certain Finite Subset of $\widehat{\mathbb C}$}
\label{all}

It is already known that for any finite subset $\alpha=\{z_1, z_2, \cdots, z_n\}\subseteq \widehat{\mathbb{C}}$, $n\ge4$, $\mathcal{A}_{\alpha}$ is finite. In this section we aim to prove that
\begin{thm}
For any finite non-trivial group $G$ of linear fractional transformations, there exists a finite subset $\alpha=\{z_1, z_2, \cdots, z_n\}\subseteq \widehat{\mathbb{C}}$ such that $\mathcal{A}_{\alpha}\simeq G$.
\end{thm}
\begin{proof}
There are only five kinds of finite non-trivial linear fractional transformation groups: the icosahedral group $A_5$, the octahedral group $S_4$, the tetrahedral group $A_4$, the dihedral group $D_k$ and the finite cyclic group $\mathbb Z_k$, $k\ge2$.

In Section \ref{A_5}, \ref{S_4} and \ref{A_4} we can find $\alpha$ such that $\mathcal{A}_{\alpha}\simeq A_5,\:S_4$ and $A_4$. In Corollary \ref{D_n1}, \ref{D_n2} and \ref{cZ_n} we can find $\alpha$ such that $\mathcal{A}_{\alpha}\simeq D_m$ and $\mathbb Z_n$, $m\ge5$, $n\ge4$. In Section \ref{n=4} we can find $\alpha$ such that $\mathcal{A}_{\alpha}\simeq D_4$ and $D_2$. And in Section \ref{n=5, con} we can find $\alpha$ such that $\mathcal{A}_{\alpha}\simeq D_3$ and $\mathbb Z_2$. So all we need to do is to find an $\alpha$ such that $\mathcal{A}_{\alpha}\simeq \mathbb Z_3$.

Set
$$
\alpha=\{0,\:1,\:w,\:w^2,\:2,\:2w,\:2w^2\},\:w=e^{\frac{2\pi}{3}i},
$$
and
$$
G=\langle z\mapsto wz\rangle\simeq \mathbb Z_3.
$$
We have $G\subseteq\mathcal{A}_{\alpha}$.

As $|\alpha|=7$, from Theorem \ref{tA_5}, \ref{tS_4} and \ref{tA_4}, we see that $\mathcal{A}_{\alpha}$ is not $A_5$, $S_4$ or $A_4$. Thus $G$ is isomorphic to $D_k$ or $\mathbb Z_{k}$.

From Theorem \ref{tD_n} we see that if $\mathcal A_\alpha\simeq D_k$, $|\alpha|=nk$ or $nk+2$, $n\in\mathbb N^*$. As $|\alpha|=7$, $\mathcal A_\alpha\simeq D_7$ or $D_5$.

From Theorem \ref{tZ_n} we see that if $\mathcal A_\alpha\simeq \mathbb Z_k$, $|\alpha|=nk$, $nk+1$ or $nk+2$, $n\in\mathbb N^*$. As $|\alpha|=7$, $\mathcal A_\alpha\simeq \mathbb Z_7$, $\mathbb Z_6$, $\mathbb Z_5$, $\mathbb Z_3$ or $\mathbb Z_2$.

It is easy to see that no five points in $\alpha$ are concyclic. Thus $\mathcal{A}_{\alpha}$ is not $D_7$, $D_5$, $\mathbb Z_7$, $\mathbb Z_6$ or $\mathbb Z_5$. As $G\subseteq\mathcal{A}_{\alpha}$ and $G\simeq\mathbb Z_3$, $\mathcal{A}_{\alpha}$ is not $\mathbb Z_2$.

Thus $\mathcal{A}_{\alpha}$ is isomorphic to $\mathbb Z_3$.
\end{proof}

%----------------------------------------------------------------
%\bibliography{paper}

\begin{thebibliography}{1}

\bibitem{De-Mu}
E. Arbarello, M. Cornalba, P. A. Griffiths and J. Harris,
\newblock {\em Geometry of Algebraic Curves}, Volume I, Springer-Verlag New York Inc. 1985.

\bibitem{Bolza} O. Bolza, \newblock{\em Ueber Binarformen sechster Ordnung mit linearen Substitutionen in sich.}
Math. Ann. {\bf 30} (1887) 546-552. 

\bibitem{Dol} I. V. Dolgachev, \newblock{\em Classical Algebraic Geometry. A modern view.} Cambridge University Press, Cambridge, 2012.

\bibitem{Mu-Pe}
M. Mulase and M. Penkava,
\newblock {\em Combinatorial Structure of the Moduli Space of Riemann Surfaces and the KP Equations},
Unpublished lecture notes, https://www.math.ucdavis.edu/mulase/texfiles/1997moduli.pdf

\bibitem{Wiman} A. Wiman, \newblock{\em Zur Theorie der endlichen Gruppen von birationalen Transformationen in der Ebene.}
Math. Ann. {\bf 48} (1896) 195-240
\end{thebibliography}

\bibliographystyle{plain}

%----------------------------------------------------------------
\small{\noindent YUE WU\\
SCHOOL OF MATHEMATICAL SCIENCES\\
UNIVERSITY OF SCIENCE AND TECHNOLOGY OF CHINA\\
HEFEI 230026 CHINA\\
wuyuee15@mail.ustc.edu.cn}\\

\small{\noindent BIN XU\\
WU WEN-TSUN KEY LABORATORY OF MATH, USTC, CHINESE ACADEMY OF SCIENCE\\
SCHOOL OF MATHEMATICAL SCIENCES\\
UNIVERSITY OF SCIENCE AND TECHNOLOGY OF CHINA\\
HEFEI 230026 CHINA\\
bxu@ustc.edu.cn}

%----------------------------------------------------------------
\end{document}